\documentclass[nohyperref]{article}

\usepackage{microtype}
\usepackage{graphicx}
\usepackage{enumerate}
\usepackage{hyperref}

\usepackage[accepted]{icml2022}

\usepackage{amsmath}
\usepackage{amssymb}
\usepackage{mathtools}
\usepackage{amsthm}

\theoremstyle{plain}
\newtheorem{theorem}{Theorem}[section]

\newtheorem{lemma}[theorem]{Lemma}
\newtheorem{corollary}[theorem]{Corollary}
\theoremstyle{definition}
\newtheorem{definition}[theorem]{Definition}
\newtheorem{assumption}[theorem]{Assumption}
\theoremstyle{remark}

\newcommand{\clA}{\mathcal{A}}
\newcommand{\clB}{\mathcal{B}}
\newcommand{\clF}{\mathcal{F}}
\newcommand{\clK}{\mathcal{K}}
\newcommand{\clM}{\mathcal{M}}
\newcommand{\clP}{\mathcal{P}}
\newcommand{\clS}{\mathcal{S}}
\newcommand{\bbE}{\mathbb{E}}
\newcommand{\bbF}{\mathbb{F}}
\newcommand{\bbN}{\mathbb{N}}
\newcommand{\bbP}{\mathbb{P}}
\newcommand{\bbR}{\mathbb{R}}

\icmltitlerunning{Convergence of Policy Gradient for Entropy Regularized MDPs in the Mean-Field Regime}

\begin{document}

\twocolumn[
\icmltitle{Convergence of Policy Gradient for Entropy Regularized MDPs \\ with Neural Network Approximation in the Mean-Field Regime}



\icmlsetsymbol{equal}{*}

\begin{icmlauthorlist}
\icmlauthor{Bekzhan Kerimkulov}{equal,uoe}
\icmlauthor{James-Michael Leahy}{equal,imp}
\icmlauthor{David \v{S}i\v{s}ka}{equal,uoe,vega}
\icmlauthor{{\L}ukasz Szpruch}{equal,uoe,ati}
\end{icmlauthorlist}

\icmlaffiliation{uoe}{School of Mathematics, University of Edinburgh, Edinburgh, United Kingdom}
\icmlaffiliation{imp}{Department of Mathematics, Imperial College London, London, United Kingdom}
\icmlaffiliation{ati}{The Alan Turing Institute, London, United Kingdom}
\icmlaffiliation{vega}{Vega Protocol, Gibraltar, Gibraltar}

\icmlcorrespondingauthor{James-Michael Leahy}{j.leahy@imperial.ac.uk}
%\icmlcorrespondingauthor{Firstname2 Lastname2}{first2.last2@www.uk}

\icmlkeywords{MDP, policy gradient, entropy regularization, mean-field, non-linear Fokker--Planck--Kolmogorov equation}

\vskip 0.3in
]



\printAffiliationsAndNotice{\icmlEqualContribution} % otherwise use the standard text.

\begin{abstract}
We study the global convergence of policy gradient for infinite-horizon, continuous state and action space, and entropy-regularized Markov decision processes (MDPs). 
We consider a softmax policy with (one-hidden layer) neural network approximation in a mean-field regime. 
Additional entropic regularization in the associated mean-field probability measure is added, and the corresponding gradient flow is studied in the 2-Wasserstein metric. 
We show that the objective function is increasing along the gradient flow.
Further, we prove that if the regularization in terms of the mean-field measure is sufficient, the gradient flow converges exponentially fast to the unique stationary solution, which is the unique maximizer of the regularized MDP objective. Lastly, we study the sensitivity of the value function along the gradient flow with respect to regularization parameters and the initial condition.
Our results rely on the careful analysis of the non-linear Fokker--Planck--Kolmogorov equation and extend the pioneering work of \cite{mei2020global} and \cite{agarwal2020optimality}, which quantify the global convergence rate of policy gradient for entropy-regularized MDPs in the tabular setting. 
\end{abstract}

\section{Introduction}

\subsection{Overview}

In the last decades, reinforcement learning (RL) algorithms with neural network approximation have demonstrated incredible performance. Notable successes have been reported in \cite{mnih2015human,silver2018general,vinyals2019grandmaster}.  
Various versions of the policy gradient algorithms have been demonstrated to be particularly effective. 
However, a mathematical theory that provides guarantees for the convergence of these algorithms has been elusive. 
In particular, we are not aware of any work that has established convergence of a policy gradient algorithm in the continuous state and action setting with neural network approximation. The motivation behind this work is to shed more light on this challenging open question in the setting of entropy regularized MDPs (Markov decision processes) where the policy is of softmax type and parameterized by a one-hidden layer neural network in the mean-field regime.

Entropy regularized MDPs have been widely studied due to their excellent empirical performance and desirable theoretical properties  \cite{haarnoja2017reinforcement, geist2019theory}. Convergence of policy gradient with softmax policies was first studied for the entropy regularized problem in \cite{agarwal2020optimality}. 
The convergence has been further quantified in \cite{mei2020global}, where the authors showed that softmax policy gradient converges at an exponential rate.  
Encouraged by the recent convergence results of \cite{mei2020global} in the tabular setting, we  consider the continuous state and action space setting with policies approximated mean-field neural networks. The mean-field setting was also considered in \cite{agazzi2020global}. The authors show that if the mean-field policy gradient flow converges, then, under appropriate assumptions, the limiting policy is optimal.

We consider an infinite horizon Markov decision model $\mathcal{M}=(S,A,P,r,\gamma)$,
where $S$ is the state space, $A$ the action space with a fixed finite reference measure $\mu$, $P$ the transition probability kernel, $r$ is a bounded reward function, and $\gamma$ is the discount factor. 
For a given stochastic policy $\pi: S\rightarrow \mathcal{P}(A)$ and initial distribution $\rho\in \clP(S)$, we consider the entropy-regularized value function
\[
V^{\pi}_{\tau}(\rho)= \mathbb{E}_{\rho}^{\pi}\left[\sum_{n=0}^{\infty}\gamma^{n} \left(r(s_n,a_n)-\tau\,\ln\frac{d\pi}{d\mu}(a_n|s_n)\right)\right]\,,
\]
where $\tau \geq 0$ determines the intensity of the entropy regularization. Denoting $V^{\pi}(s)=V^{\pi}(\delta_s)$, we have
\begin{equation}\label{eq:intro:policy_Bellman_KL_occ}
	V^{\pi}_{\tau}(s)= V^{\pi}_0(s)-\frac{\tau}{1-\gamma}\int_{S}d^{\pi}(ds'|s)\textnormal{KL}(\pi(\cdot|s')|\mu)\,,
\end{equation}
where $d^\pi$ is the occupancy kernel under the policy $\pi$.
For full details on our assumptions and notation, we refer to Sections~\ref{sec:notation} and~\ref{sec:MDP}. 
There are two implications of having $\tau>0$. 
The first is that the optimal policy satisfies 
\[
\pi^*_{\tau}(da|s) = \exp\left(\frac1\tau\left(Q^{\ast}_{\tau}(s,a)-V^{\ast}_{\tau}(s)\right)\right)\mu(da)\,,
\]
where $V^*_{\tau}$ and $Q^*_{\tau}$ denote the optimal value and state-action value functions, respectively (see Theorem \ref{thm:DPP} or \cite{ziebart2010modeling, haarnoja2017reinforcement, geist2019theory}). 
Second, since the entropy term   is strictly convex \cite{dupuis1997weak}[Sec.\ 1.4], the addition (see, e.g., \eqref{eq:intro:policy_Bellman_KL_occ}) is expected to improve the convergence when optimizing $V^{\pi}_{\tau}(\rho)$ with respect $\pi$ with a gradient-descent-type method such as softmax policy gradient.
While the latter point may seem intuitive, the analysis is far from being straightforward even in the tabular case studied in \cite{mei2020global}, where the entropic regularization  is shown to lead to exponential convergence of the softmax policy gradient algorithm.
The difficulty arises chiefly from the fact that $V^{\pi}_\tau$ depends on $\pi$ in a non-linear and, in general, non-convex way.

In the  case where $S$ and $A$ are finite sets (i.e., the tabular setting), \cite{mei2020global} considered softmax policies $\pi_F: S\rightarrow \clP(A)$ of the form \begin{equation}\label{eq:softmax}
	\pi_F(a|s)\propto \exp(F(s,a))
\end{equation}
for $F: S\times A\rightarrow \bbR$, and where we have identified $\pi_F$ with its probability mass function. 
Since $F$ is equivalent to a parameter $\theta \in \bbR^{|S|\times |A|}$, 
policy gradient is a gradient flow in $\bbR^{|S|\times |A|}$. This approach becomes computationally intractable as the size of the sets $S$ and $A$ grow large.
To overcome this and also to cover the continuous state-action setting, we use function approximation.  
More specifically, we consider policies approximated by a one hidden layer neural network  with mean-field scaling:
\begin{equation*}
	\pi_\theta(da|s) \propto \exp\left(\frac{1}{m}\sum_{i=1}^ m f(\theta^i,s,a)\right)\mu(da) \,,\theta \in \bbR^{d\times m}\,,
\end{equation*}
where $f: \bbR^d\times S\times A\rightarrow \bbR$.
The continuous-time policy gradient  is given by
\[
\frac{d}{dt}\theta_t = \nabla V^{\pi_{\theta_t}}_{\tau}(\rho)\,,
\]
where here and later $\nabla$ means $\nabla_\theta$ unless stated otherwise. One can equivalently view the policy as a function of an empirical measure over parameters $\nu_t^m := \frac{1}{m}\sum_{i=1}^m \delta_{\theta_t^i}$:
\[
\pi_{\nu_t^m}(da|s) \propto \exp\left(\int_{\mathbb{R}^d}f(\theta,s,a)\nu^m_t(d\theta)\right)\mu(da)\,.
\]
While the one-hidden-layer neural network is typically non-convex as a function of the parameters $\theta$, once lifted to the space of measures it becomes linear as a function of measure.

The idea of using mean-field scaling to understand the convergence of gradient descent algorithms used to train neural networks has been extensively studied in a recent series of works, see~\cite{mei2018mean,rotskoff2018neural,chizat2018global,sirignano2021mean,hu2019mean}. 
These works show that in the context of supervised learning, finding the optimal weights in deep neural networks can be viewed as a sampling problem.
What emerges is the idea that the aim of the learning algorithm is to find optimal distribution over the parameter space rather than optimal values of the parameters. 
As a consequence, individual values of the parameters are not important in the sense that different sets of weights sampled from the correct (i.e., optimal) distribution are equally good.

The key feature of the supervised learning task studied by the aforementioned mentioned works is that a finite-dimensional non-convex optimization problem becomes convex when lifted to an infinite-dimensional space of measures; that is, working with the limiting measure $\nu^m \rightarrow \nu$ as $m\rightarrow \infty$. This observation, unfortunately, does not hold in RL problem with policies $\pi_{\nu}: S\rightarrow \clP(A)$ indexed by probability measures $\nu\in \clP(\bbR^d)$ of the form
\[
\pi_{\nu}(da|s) \propto \exp\left(\int_{\mathbb{R}^d}f(\theta,s,a)\nu(d\theta)\right)\mu(da)\,
\]
due to the non-linearity and non-convexity of the problem  in $\pi$, and hence $\nu$. Despite this, in \cite{agazzi2020global}, the authors show that if the policy gradient flow for the mean-field neural network parameters $\nu_t$ converges to $\nu^\ast$ with full support, then the resulting parameterized policy $\pi_{\nu^*}$ is globally optimal for the entropy regularized problem (i.e., $V^*=V^{\pi_{\nu^{\ast}}}$) under the assumption that the activation function used leads to a sufficiently expressive class of functions. 
However, the work of \cite{agazzi2020global} does not establish that the flow converges. 

In general,   $\nu \in \mathcal{P}(\mathbb{R}^d) \mapsto V^{\pi_{\nu}}_{\tau}$ is not convex, and hence one should not expect convergence to the global optimum  for the associated policy gradient algorithm.
To alleviate this issue and establish global convergence rates for policy gradient, we introduce further entropic regularization of $\nu$ relative to a prior $e^{-U}$ for some normalized potential $U$. 
For arbitrarily given $\tau\ge0$, $\sigma\ge0$, and $\rho\in \mathcal{P}(S)$, we define $J^{\tau,\sigma}:\mathcal P_2(\mathbb R^d) \rightarrow \mathbb R$ by
\begin{equation} \label{eq objective-intro}
	J^{\tau,\sigma}(\nu)=V^{\pi_\nu}_{\tau}(\rho)-\frac{\sigma^2}{2}\text{KL}(\nu|e^{-U})\,.
\end{equation}
In Theorem \ref{thm:increasing_objective}, we  prove that $J^{\tau,\sigma}(\nu_t)$ increases as function of $t$ along the gradient flow for any $\tau \geq 0$, $\sigma\geq 0$. This fact is missing in \cite{agazzi2020global}, which considered the case $\sigma=0$. Moreover, for any $\tau \geq 0$ and for sufficiently large $\sigma > 0$, we show that $\nu_t \to \nu^\ast$ exponentially fast and that $\nu^\ast$ is the optimal parameter measure, see Theorem~\ref{thm:ergodicity}. While having $\tau>0$ motivates considering policies of softmax type (see Theorem \ref{thm:DPP} below), this condition is not needed to prove any of our results. How large $\sigma$ needs to be taken depends on the bounds we establish in Theorem \ref{thm:bnd_reg_Lion_deriv}. In Corollary \ref{cor:sens_value}, we estimate the sensitivity of the value function along the gradient flow with respect to regularization parameters and the initial condition. A heuristic derivation of the gradient flow and conditions for optimality is provided in Section \ref{sec:heuristic} for the readers convenience.

\subsection{Literature Review}
There is an enormous amount of research literature on RL and we cannot hope to do it justice here.  For this reason, we focus on the subset of RL that we feel is most related to this work.
Entropy regularized RL has demonstrated both good algorithmic performance and desirable theoretical properties~\cite{haarnoja2017reinforcement, haarnoja2018soft, geist2019theory,vieillard2020leverage,neu2017unified, fox2015taming, ziebart2010modeling}.
It has been shown that softmax policies are optimal in the entropy regularized setting. In fact, many authors consider softmax parameterized policies and the policy gradient algorithm without  adding entropic regularization.
In the tabular and ``compatible" function approximation case, \cite{agarwal2020optimality} proved convergence of softmax policy gradient at rate $O(1/\sqrt{t})$.  In the tabular case, \cite{mei2020global} showed that softmax policy gradient converges with rate $O(1/t)$. Moreover, using \L ojasiewicz inequalities, they showed that entropic regularization improves the rate of convergence to $O(e^{-t})$. We remark that the authors in \cite{li2021softmax} establish lower bounds for the policy gradient algorithms studied in~\cite{agarwal2020optimality} and~\cite{mei2020global}, and show that these may scale poorly with the cardinality of the state space. In \cite{liu2019neural} and \cite{wang2020neural}, the authors consider consider over-parameterized  neural-network policies and  show that neural proximal/trust region and actor-critic policy optimization, respectively, converges with rate $O(1/\sqrt{t})$ in the idealized setting (exact action-value setting).

As explained above, the tabular softmax policy \eqref{eq:softmax} gradient method  becomes computationally intractable as the size of the sets $S$ and $A$ grow large.
In this paper, we consider continuous state and action spaces and softmax policies approximated by an idealized infinitely-wide one hidden layer neural network (i.e., the mean-field regime). Working in the mean-field regime is motivated in part by the success of a series of recent works which prove the convergence of (noisy) stochastic gradient algorithms arising in the training of neural networks in the static and supervised learning setting, see~\cite{jordan1998variational, mei2018mean,rotskoff2018neural,chizat2018global,sirignano2021mean,hu2019mean} and by the recent work \cite{agazzi2020global} in RL. 
We will discuss the connection~\cite{hu2019mean}
and~\cite{agazzi2020global} in more detail.

In \cite{hu2019mean}, the authors studied the supervised learning setting for a one-hidden layer neural network in the mean-field regime. The key feature of this setting is the convexity of the unregularized (e.g.,  $\sigma=0$) objective. By taking $\sigma>0$, the objective function becomes strictly convex, which allows them to establish convergence to the invariant measure for any $\sigma>0$.
In the RL setting, the objective function is highly non-linear and, in general, non-convex, and hence the analysis of the problem is significantly more involved. 
This can be seen in Theorem \ref{thm:bnd_reg_Lion_deriv}, where necessary boundedness and regularity results of the objective function are obtained. The analysis of the gradient flow \eqref{eq:gf_intro} in the current work is done using PDE techniques that are different from the tools used in \cite{hu2019mean} and require less regularity.
This enables us to establish that the objective function is decreasing along the flow even when $\sigma=0$, which, even in a supervised setting, is a new result.

In~\cite{agazzi2020global}, the authors consider mean-field softmax parameterized policies and derive the gradient flow for the parameter measure, see \eqref{eq:gf_intro} with $\sigma=0$. 
Their main theorem says that if $\nu^{\ast}$ is a stationary solution of~\eqref{eq:gf_intro} with full support and the family of functions $(s,a)\mapsto f(s,a,\theta)$ indexed by $\theta\in \bbR^d$  gives rise to a sufficiently expressive class of functions, then $V^{\pi_{\nu^{\ast}}}_{\tau}=V^{\ast}_{\tau}$, which implies that the parameterized policy is optimal. 
Whether there exists a solution to the stationary equation in the space of probability measures is unknown.
Optimality is not established by directly showing $\nu^{\ast} \in\operatorname{argmax}_{\nu}J^{\tau,0}(\nu)$, but rather by showing $V^{\pi_{\nu^{\ast}}}_{\tau}$ satisfies the optimal Bellman equation \eqref{eq:BellmanDPP} for $V^{\ast}_{\tau}$, and hence that $\pi_{\nu^{\ast}} \in\operatorname{argmax}_{\pi}V^{\pi}_{\tau}(\rho)$. 
Since the gradient flow preserves the support of the initial condition $\nu_0$, they are able to conclude that \emph{if} the gradient flow $\nu_t$ converges  to  $\nu^{\ast}$, then $\pi_{\nu^{\ast}}$ is optimal.
Crucially, the work of~\cite{agazzi2020global} does not establish that the flow converges. 
Moreover, it does not establish that the gradient flow is non-decreasing; our Theorem \ref{thm:apriori_grad_flow} establishes this for $\tau,\sigma \ge 0$. 
The insight gained by~\cite{agazzi2020global} is that the value function plays a vital role as the solution to the Bellman equation. 
We no longer study the original value function $J^{\tau,0}(\nu)=V^{\pi_{\nu}}_{\tau}$, but rather $J^{\tau,\sigma}$, which with $\sigma >0$ introduces regularization of the parameter measure, and does not solve a Bellman equation. 
Hence, we must employ different techniques to show optimality (see the proof of Theorem \ref{thm:ergodicity}). In particular, we make use of Theorem~\ref{thm:apriori_grad_flow}, which is a step toward understanding what is happening at the level of the value function. 
Our other main theorem, Theorem \ref{thm:ergodicity}, is meant to complement the work of~\cite{agazzi2020global} by identifying sufficient conditions under which one can conclude that $\nu_t\to\nu^\ast$ at a rate $O(e^{-\beta t})$, where $\beta > 0$ is a constant quantified in our analysis.

\subsection{Heuristic Derivation of the Gradient Flow and Conditions for Optimality}
\label{sec:heuristic}

Our aim is to maximize $J^{\tau,\sigma}:\mathcal P_2(\mathbb R^d) \rightarrow \mathbb R$  defined in~\eqref{eq objective-intro}. 
In this section, we present a heuristic computation in the spirit of Otto calculus \cite{villani2009optimal} through which we identify the gradient flow equation. For simplicity, in this section, we will restrict to absolutely continuous measures $\nu\in \mathcal P_2^{\textnormal{ac}}(\mathbb R^d)$ and abuse notation and identify $\nu$ with its density $d\nu/d\theta$.  At the end of this section, we explain how this abstract gradient flow relates to a familiar noisy gradient ascent.  
Let $E:\bbR_+ \times \mathbb R^{d} \rightarrow \mathbb R^{d}$ denote a time-dependent vector field and consider the gradient flow 
\[
\partial_t \nu_t = \nabla \cdot ( E_t \nu_t)\,.
\]
Our aim is to identify the vector field $E$ so that $J^{\tau,\sigma}(\nu_t)$ is increasing as a function of $t$.  To do this, we will use differential calculus on $\mathcal{P}_2(\mathbb{R}^d)$ equipped with the 2-Wasserstein distance, see~\cite{ambrosio2008gradient,carmona2018probabilistic}.  Ignoring for the moment that entropy is only lower semi-continuous, and hence that the linear functional derivative  $\frac{\delta J^{\tau,\sigma}}{\delta \nu}: \clP_2(\bbR^d)\times \bbR^d\rightarrow \bbR$ does not exists (see Definition \ref{def:lin_func_der}), we have 
\[
\begin{split}
 \partial_t J^{\tau,\sigma}(\nu_t) & = \lim_{h \rightarrow 0} \frac{J^{\tau,\sigma}(\nu_{t +h}) - J^{\tau,\sigma}(\nu_{t})}{h}\\
& =  \lim_{h \rightarrow 0}\frac{1}{h} \! \int^{1}_0 \!\! \int_{\mathbb{R}^d} \!\! \frac{\delta J^{\tau,\sigma}}{\delta \nu}(\nu^{\varepsilon,h}_t) (\nu_{t+h} - \nu_t )\,d\theta \,d\varepsilon\,,
\end{split}
\]
where $\nu_t^{\varepsilon,h}:=  \nu_t + \varepsilon(\nu_{t+h} - \nu_t)$.
Since $\nu_t^{\varepsilon,h} \rightarrow \nu_t$ as $\varepsilon \rightarrow 0$, integrating by parts, we obtain
\[
\begin{split}
	\partial_t J^{\tau,\sigma}(\nu_t) 
	&	=  \int_{\mathbb R^d}  \frac{\delta J^{\tau,\sigma}}{\delta \nu}(\nu_t)\nabla \cdot ( E_t \nu_t)\,d\theta \\
	&= - \int_{\mathbb R^d}  \nabla\frac{\delta J^{\tau,\sigma}}{\delta \nu}(\nu_t)  E_t \nu_t \, d\theta\,.
\end{split}
\]
To ensure that $J^{\tau,\sigma}(\nu_t)$ is increasing, one must take (up to a multiplicative constant) $
E_t := -\nabla\frac{\delta J^{\tau,\sigma}}{\delta \nu}(\nu_t)$
so that 
\[
\partial_t J^{\tau,\sigma}(\nu_t) = \int_{\mathbb R^d} \left|\nabla\frac{\delta J^{\tau,\sigma}}{\delta \nu}(\nu_t)\right|^2 \nu_t\, d\theta \geq 0\,.
\]
It can be shown that along the gradient flow, we have
$$
\frac{\delta J^{\tau,\sigma}}{\delta \nu} (\nu_t,\theta) = \frac{\delta J^{\tau,0}}{\delta \nu}(\nu_t,\theta) - \frac{\sigma^2}{2}\left(U(\theta) + \ln \nu_t(\theta) + 1 \right)\,,
$$
where we recall that $J^{\tau,0}(\nu)=V^{\pi_\nu}_{\tau}$.
In Lemma \ref{lem:func_deriv_J}, we compute $\frac{\delta J^{\tau,0}}{\delta \nu}(\nu_t,\theta)$, and if $\tau=0$,  the  expression is equivalent to the expression arising in policy gradient, see~\cite{sutton2018reinforcement}[Ch.\ 13], with the difference being that our policies  are parameterized by measures, and so measure derivatives appear instead of classical gradients. 
Thus, the gradient flow is given by
\begin{equation} \label{eq:gf_intro}
	\partial_t \nu_t = \nabla \cdot \left[ \left(-\nabla\frac{\delta J^{\tau,0}}{\delta \nu}(\nu_t) + \frac{\sigma^2}{2} \nabla U \right) \nu_t \right] + \frac{\sigma^2}{2} \Delta \nu_t\,.  
\end{equation}

In Theorem~\ref{thm:optimality} and Corollary~\ref{cor:local_max_PDE}, below, we show that any local maximum $\nu^\ast$ of $J^{\tau,\sigma}$ must satisfy the first order condition
$
\theta \mapsto \frac{\delta J^{\tau,\sigma}}{\delta \nu} (\nu^{\ast},\theta)  \textnormal{ is  constant},\,
$
or equivalently that $\nu^\ast$ is a stationary solution of~\eqref{eq:gf_intro}. This also implies that in the case $\sigma>0$, $\nu^{\ast}$ satisfies 
$$
\nu^{\ast}(\theta) \propto \exp\left(\frac{2}{\sigma^2} \frac{\delta J^{\tau,0}}{\delta \nu}(\nu^{\ast},\theta) -U(\theta) \right)\,,
$$
which can be viewed as a posterior distribution over the parameters space given the prior $e^{-U}$. 

On the other hand, if $\nu_t$ converges to a stationary solution of \eqref{eq:gf_intro}, then 
using the fact that  $J^{\tau,\sigma}(\nu_t)$ is increasing, we can conclude that $\nu^\ast$ is a maximum as long as there is a unique stationary solution.
Thus, we see that the question of the existence of the unique global minimizer is related to the existence of a unique stationary solution of \eqref{eq:gf_intro}. 
However, from the general theory of nonlinear Fokker--Planck--Kolmogorov equations, one only expects existence and uniqueness of solutions to the stationary equation if $\sigma>0$ is large (see, e.g., \cite{bogachev2019convergence}[Ex.\ 1.1], \cite{bogachev2018distances}[Ex.\ 4.3], \cite{manita2015uniqueness}[Sec.\ 6]).

To implement the gradient flow \eqref{eq:gf_intro} one needs to compute or estimate $\nabla \frac{\delta J^{\tau,0}}{\delta \nu}(\nu,\theta)$, where $\frac{\delta J^{\tau,0}}{\delta \nu}(\nu,\theta)$ is specified in Lemma \ref{lem:func_deriv_J}.
The policy gradient method is a cornerstone of RL precisely because the terms $d^{\pi_{\nu}}_{\rho}$ and $Q_{\tau}^{\pi_{\nu}}$ can be estimated using Monte Carlo (e.g., Reinforce) from roll-outs without having to explicitly know the transition probability.
Instead of working with the non-linear PDE \eqref{eq:gf_intro}, one can use the probabilistic representation  $\nu=\textnormal{Law}(\theta)$, where  $\theta: \Omega\times  \bbR_+\rightarrow \mathbb{R}^d$ is the solution of the McKean--Vlasov stochastic differential equation (SDE)
\begin{equation*}
	d\theta_t= \left(\nabla \frac{\delta J^{\tau,0}}{\delta \nu}(\textnormal{Law}(\theta_t),\theta_t) - \frac{\sigma^2}{2} \nabla U(\theta_t)\right)dt+\sigma dW_t\,,
\end{equation*}
and where $W$ is a $d$-dimensional Wiener process. Approximating $\nu$ with its empirical measure and discretizing in time with a learning rate $\eta$,  we arrive at the familiar noisy gradient ascent algorithm 
\begin{equation}\label{eq particle system}
	\left\{
	\begin{aligned}
		\theta^i_{k+1} & =\theta^i_{k} + \eta \left(  \nabla\frac{\delta J^{\tau,0}}{\delta \nu}(\nu_k^m,\theta^i_k) - \frac{\sigma^2}{2} \nabla U(\theta^i_k) \right) \\
		&\quad + \sqrt{\eta}\sigma \zeta^i_{k+1}\,, \;\;k\in \mathbb{N}_0\,,\\
		\nu^m_k & :=\frac{1}{m}  \sum_{i=1}^m \delta_{\theta_k^i}\,,
	\end{aligned}
	\right.
\end{equation}
where $\{\zeta^i_k\}_{1\le i\le m, k\in \bbN_0} \overset{\textnormal{i.i.d.}}{\sim} N(0,1)$. 
The convergence rates of \eqref{eq particle system} to the gradient flow \eqref{eq:gf_intro} (or its probabilistic representation) are well understood under general conditions. For example, we refer a reader to \cite{jabir2019mean}[Thms.\ 8 and 9], in which such analysis has been carried out in the context of training  recurrent neural networks, and to \cite{delarue2021uniform}, where uniform in time weak particles approximation error has been studied. The approximations errors between $\nu^m_k$ given by~\eqref{eq particle system} and $\nu_t$, with $t=k\eta$, are of $O(1/m)+O(\eta)$ uniformly in time, and this can be seen as proxy for algorithmic complexity.

\subsection{Outline of the Paper}

The main contributions of this work are
\begin{enumerate}
	\item Theorem~\ref{thm:apriori_grad_flow}, which proves that the objective $J^{\tau,\sigma}(\nu_t)$ increases along the gradient flow equation for the evolution of $(\nu_t)_{t\in \bbR_+}$ corresponding to policy gradient;
	\item Theorem~\ref{thm:ergodicity}, which proves exponential convergence of $\nu_t$ to $\nu^{\ast} $ if $\sigma > 0 $ is sufficiently large.
	\item Theorem~\ref{thm:cts_dep_flow} and Corollary~\ref{cor:sens_value}, where  the dependence of the gradient flow and its stationary solutions and of the value function on $\tau$ and $\sigma$ is quantified.
\end{enumerate}
The notation and essential definitions are introduced in Section~\ref{sec:notation}. Entropy regularized MDPs are introduced in Section~\ref{sec:MDP}.  The main results of the paper are stated Section~\ref{sec:param_MDP}, the proofs of which will be given in Section \ref{sec:proofs}.
Conclusions and possible future research directions are discussed in Section~\ref{sec:conc_and_future_work}. A heuristic derivation of our results is also provided in Section \ref{sec:heuristic}, which sheds more light on our work and the mean-field approach.

\subsection{Notation and Definitions}\label{sec:notation}

Let $\mathbb R_+ := [0,\infty)$. 
For given $k\in \mathbb N_0\cup{\infty}$, let $C^k(\bbR^d)$ denote the space of $k$-times continuously differentiable functions and $C^k_b(\bbR^d)$ the subspace of functions of $C^k(\bbR^d)$ for which the function and all the derivatives up to order $k$ are bounded.
Let $C^k_c(\bbR^d)$ denote the subspace of $C^k(\mathbb R^d)$ of functions with compact support. 
We will use standard notation $L^p$, $p\in [1,\infty]$, for Lebesgue spaces of integrable functions. 

Let $(E,d)$ denote a complete separable metric space (i.e., a Polish space). We always equip a Polish space with its Borel sigma-field $\mathcal{B}(E).$ Denote by $B_b(E)$ the space of bounded strongly measurable functions $f:E\rightarrow \mathbb{R}$ endowed with the supremum norm $|f|_{B_b(E)}=\sup_{x\in E}|f(x)|$. Denote by $\mathcal{M}(E)$ the Banach space of signed measures (finite) $\mu$ on $E$ endowed with the total variation norm $|\mu|_{\mathcal{M}(A)}=|\mu|(E)$, where $|\mu|$ is the total-variation measure. 
We note that if $\mu=f d\rho$, where $\rho \in \mathcal{M}_+(E)$ is a non-negative measure  and $f\in L^1(E,\rho)$, then  $|\mu|_{\mathcal{M}(E)}=|f|_{L^1(E,\rho)}$. 
We denote by $\mathcal{P}(E)\subset\mathcal{M}(E)$ the convex subset of probability measures on $E$. For $\mu,\mu'\in \mathcal{P}(E)$ such that $\mu$ is absolutely continuous with respect to $\mu'$, the relative entropy of $\mu$ with respect to $\mu'$ (or Kullback-Liebler divergence of $\mu$ relative to $\mu'$) is defined by 
$$
\textnormal{KL}(\mu|\mu')=\int_{E}\ln\frac{d\mu}{d\mu'}(x)\mu(dx)\,.
$$
For measure $\mu\in \mathcal{M}_+(E)$ and measurable function $f:E\rightarrow \mathbb{R}$, let 
$$
\mu-\underset{x\in E}{\textnormal{ess sup}} f = \inf \{ c\in \mathbb{R}: \mu\{x\in E: f(x)>c\}=0\}\,.
$$

It is convenient to have notation for measurable functions $k:E_1\rightarrow \mathcal{M}(E_2)$ for given Polish spaces $(E_1,d_1)$ and $(E_2,d_2)$. For example, $P:S\rightarrow \mathcal{P}(S\times A)$ will denote a controlled transition probability and $\pi: S\rightarrow \mathcal{P}(A)$ will denote a stochastic policy. Denote by $b\mathcal{K}(E_1|E_2)$ the Banach space of bounded signed kernels $k: E_2\rightarrow \mathcal{M}(E_1)$ endowed with the norm $|k|_{b\mathcal{K}(E_1|E_2)}=\sup_{x\in E_2}|k(x)|_{\mathcal{M}(E_1)}$; that is, $k(U|\cdot): E_2\rightarrow \mathbb{R}$ is measurable for all $U\in \mathcal{M}(E_1)$ and $k(\cdot|x)\in \mathcal{M}(E_1)$ for all $x\in E_2$. For a fixed positive reference measure $\mu \in \mathcal{M}(E_1)$, we denote by $b\mathcal{K}_{\mu}(E_1|E_2)$ the space of bounded kernels that are absolutely continuous with respect to $\mu$.

Every kernel $k\in b\mathcal{K}(E_1|E_2)$ induces bounded linear operators $T_k \in \mathcal{L}(\mathcal{M}(E_2),\mathcal{M}(E_1))$ and $S_k \in \mathcal{L}(B_b(E_1),B_b(E_2))$ defined by
$$
T_k\mu(dy)=\mu k(dy)=\int_{E_2}\mu(dx)k(dy|x)$$
and 
$$S_kf(x)=\int_{E_1}k(dy|x)f(y)\,,
$$
respectively.
Moreover, by \cite{kunze2011pettis}[Ex.\ 2.3 and Prop.\ 3.1], we have
\begin{align*}
|k|_{b\mathcal{K}(E_1|E_2)}&=\sup_{x\in E_2}\underset{|h|_{B_b(E_1)}\le 1}{\sup_{h\in B_b(E_1)}}\int_{E_1}h(y)k(dy|x)\\
&=|S_k|_{\mathcal{L}(B_b(E_1),B_b(E_2))}\\
&=|T_k|_{\mathcal{L}(\mathcal{M}(E_2),\mathcal{M}(E_1))}\,,
\end{align*}
where the latter are operator norms. Thus, $b\mathcal{K}(E|E)$ is a Banach algebra with the product defined via composition of the corresponding linear operators; in particular, for a given $k\in b\mathcal{K}(E|E)$,
\[
\begin{split}
& T_k^n\mu(dy)=\mu k^n(dy)\\
& = \int_{E^{n}}\mu(dx_0)k(dx_1|x_0)\cdots k(dx_{n-1}|x_{n-2}) k(dy|x_{n-1})\,.
\end{split}
\]
Notice that if $f\in L^\infty(E_1,\mu)$ and $k\in b\mathcal{K}_{\mu}(E_1|E_2)$, then for all $x\in E_2$,
\begin{align}
S_kf(x)&=\int_{E_1}\mu(dy)\frac{dk}{d\mu}(y|x)f(y)\notag\\
&\le |f|_{L^\infty(E_1,\mu)} \left|\frac{dk}{d\mu}(\cdot|x)\right|_{L^1(E_1,\mu)}\notag \\
&\le |f|_{L^\infty(E_1,\mu)} |k|_{b\mathcal{K}(E_1|E_2)}\,. \label{ineq:essup_kernel}
 \end{align}

We denote by $\mathcal{P}(E_1|E_2)$ the convex subspace of $P\in b\mathcal{K}(E_1|E_2)$ such that $P(\cdot|x)\in \mathcal{P}(E_1)$ for all $x\in E_2$; such kernels are referred to as stochastic kernels. A  stochastic kernel $P\in \mathcal{P}(E_1|E_2)$ is said to be strongly Feller if $\int_{E_1}P(dy|x)f(y)$ is continuous in $x\in E_2$ for all $f\in B_b(E_1)$. For a fixed positive reference measure $\mu \in \mathcal{M}(E_1)$, we denote by $\mathcal{P}_{\mu}(E_1|E_2)$ the space of kernels that are absolutely continuous with respect to $\mu$. 
A bounded kernel $k\in b\mathcal{K}(E_1|E_2)$ is thus strongly Feller if the range of $S_k$ lies in the space of continuous functions on $E_2$.

\begin{definition}
\label{def:lin_func_der_kernel}
We say a function $u:\clP(E_1|E_2)\rightarrow \bbR$ has a linear functional derivative if there exists a continuous and bounded function $\frac{\delta u}{\delta \pi}: b\clP(E_1|E_2)\rightarrow  b\clK(E_1|E_2)$ such that for all $\pi,\pi'\in \clP(E_1|E_1)$,
\[
\begin{split}
& \underset{\varepsilon \in [0,1]}{\lim_{\varepsilon\rightarrow 0}}\frac{u(\pi+\varepsilon(\pi'-\pi))-u(\pi)}{\varepsilon} \\
& =\int_{E_1\times E_2} \frac{\delta u}{\delta \pi}(ds|a)(\pi'-\pi)(da|s)\,.
\end{split}
\]
\end{definition}

For given $p\in \bbN$, let $\mathcal P_p(\mathbb R^d)$ denote the set of probability measures with finite $p$-th moments.  The $p$-Wasserstein distance between $\nu,\nu' \in \mathcal{P}_p(\mathbb{R}^d)$ is given by
\begin{align*}
	W_p(\nu,\nu')&=\inf_{\pi \in \Pi(\nu,\nu')}\left(\int_{\mathbb{R}^d\times \mathbb{R}^d}|\theta-\theta'|^p\pi(d\theta,d\theta')\right)^{\frac{1}{p}}\,
\end{align*}
where $\Pi(\nu,\nu')$ denotes the set of probability measures on $\mathbb{R}^d\times \mathbb{R}^d$ with $\nu$ and $\nu'$ as first and second marginals, respectively. For all $p,p'\in \bbN$ with $p'>p$ and $\nu,\nu'\in \mathcal{P}_{p'}(E)$, we have 
\begin{equation}\label{ineq:W1W2}
	W_p(\nu,\nu') \le W_{p'}(\nu,\nu')\,.
\end{equation}
The set $\mathcal P_p(\mathbb R^d)$ endowed with the topology induced by the $p$-Wasserstein distance is a complete separable metric space~\cite{villani2009optimal}[Thm.\ 6.18]. 
The Kantorovich--Rubinstein duality theorem~\cite{villani2009optimal}[Thm.\ 5.10] implies that for all $\nu,\nu'\in \clP_1(\bbR^d),$
\begin{equation}\label{eq:KRW1}
	W_1(\nu,\nu')=\sup_{\phi\in \textnormal{Lip}_1(\bbR^d)} \int_{\mathbb{R}^d} \phi(\theta)(\nu-\nu')(d\theta)\,,
\end{equation} 
where $\textnormal{Lip}_1(\bbR^d)$ denotes the space of all functions $\phi:\mathbb R^d\to \mathbb R$ with a Lipschitz constant $1$.
Henceforth, we let $\clP_0(\bbR^d)=\clP(\bbR^d)$ and recall that this convex subset of $\clM(\bbR^d)$ is endowed with the topology of total variation distance. 
We denote the Lebesgue measure on $\bbR^d$ by $\lambda$. Denote by $\mathcal P_p^{\text{ac}}(\mathbb R^d)$  the subset of $\mathcal P_p(\mathbb R^d)$ consisting of measures that absolutely continuous with respect to $\lambda$. We abuse notation and identify $\nu=\frac{d\nu}{d\theta}:=\frac{d\nu}{d\lambda}$ in $\mathcal P_p^{\text{ac}}(\mathbb R^d)$. 
Let $C(\bbR_+; \clP_p(\bbR^d))$ denote the space of continuous functions $\nu:\mathbb R^+ \to \mathcal P_p(\mathbb R^d)$.

\begin{definition}\label{def:lin_func_der} Let $p\in \bbN_0.$
	Let $(V,|\cdot|_V)$ denote a  Banach space.   We say a function $u: \mathcal{P}_p(\mathbb{R}^d)\rightarrow V$ has a linear functional derivative if there exists a continuous and bounded function $\frac{\delta u}{\delta \nu}:  \mathcal{P}_p(\mathbb{R}^d)\times \mathbb{R}^d\rightarrow V$ such that  for all $\nu,\nu'\in\mathcal{P}_p(\mathbb{R}^d)$, 
	\begin{gather*}
	\underset{\varepsilon \in [0,1]}{\lim_{\varepsilon\rightarrow 0}}\frac{u(\nu+\varepsilon(\nu'-\nu))-u(\nu)}{\varepsilon}\\
 =\int_{\mathbb{R}^d} \frac{\delta u}{\delta \nu}(\nu,\theta)(\nu'-\nu)(d\theta)\,.
	\end{gather*}
\end{definition}
Owing to the fundamental theorem of calculus a function $u: \mathcal{P}_p(\mathbb{R}^d)\rightarrow V$ has a linear functional derivative if and only if there exists a continuous and bounded function $\frac{\delta u}{\delta \nu}:  \mathcal{P}_p(\mathbb{R}^d)\times \mathbb{R}^d\rightarrow V$ such that  for all $\nu,\nu'\in\mathcal{P}_p(\mathbb{R}^d)$, 
\begin{equation}
\begin{split}
& u(\nu')-u(\nu)\\
&=\int_0^1 \int_{\mathbb{R}^d} \frac{\delta u}{\delta \nu}(\nu+\varepsilon (\nu'-\nu), \theta)(\nu'-\nu)(d\theta)\,d\varepsilon\,.\label{eq:lin_der_diff_form}
\end{split}
\end{equation}

\section{Formulation and the Statement of the Main Results}

\subsection{Entropy Regularized Markov Decision Processes}
\label{sec:MDP}
We refer readers to~\cite{puterman2014markov,bertsekas2004stochastic,hernandez2012discrete} for thorough introduction Markov decision processes.
Let $S$ and $A$ denote Polish spaces. Let $P\in \mathcal{P}(S|S\times A)$.
Let $r\in B_b(S\times A)$. Let  $\gamma\in [0,1)$.  Let $\tau\in \bbR_+$ and $\mu\in \mathcal{M}_+(A)$ denote a strictly positive finite measure. 
The seven-tuple $\mathcal{M}=(S,A,P,r,\gamma,\tau, \mu)$ determines an infinite horizon $\tau$-entropy regularized $\gamma$-discounted Markov decision model. Here, we note, that $\tau=0$ is considered, and if $\tau=0$, then $\clM$ is a infinite horizon Markov decision model $\mathcal{M}=(S,A,P,r,\gamma)$ since $\mu$ is not needed to formulate the control problem. If $\tau=0$, we further assume that $A$ is compact, $P(\cdot|\cdot,a)\in \mathcal{P}(S|S)$ is strongly Feller for all $a\in A$, and that $r(s,\cdot):A\rightarrow \mathbb{R}$ is upper semi-continuous for every $s\in S$ so that Condition 3.3.3 in \cite{hernandez2012discrete} holds, and thus measurable selection condition holds. 

Let $((S\times A)^{\bbN},\mathcal{F})$ denote the canonical sample space, where $\mathcal{F}$ is the corresponding $\sigma$-algebra. Elements of $(S\times A)^{\bbN}$ are of the form $(s_0,a_0,s_1,a_1,\ldots)$ with $s_n\in S$ and $a_n\in A$ denoting the projections and called the state and action variables, at time $n\in \mathbb{N}_0$, respectively. By Proposition 7.28 in \cite{bertsekas2004stochastic}, for an arbitrarily given initial distribution $\rho\in \mathcal{P}(S)$ and randomized stationary policy $\pi \in \mathcal{P}(A|S),$ there exists a unique product probability measure $\bbP^{\pi}_{\rho}$ on the canonical sample space with expectation denoted $\mathbb{E}^{\pi}_{\rho}$ such that for every time $n\in \mathbb{N}_0$,  $\bbP^{\pi}_{\rho}(s_0\in \clS)=\rho(\clS)$, $\bbP^{\pi}(a_n\in \clA|(s_0,a_0,\ldots,s_n))=\pi(a_n|s_n)$, and $$\bbP^{\pi}_{\rho}(s_{n+1}\in \clS|(s_0,a_0,\ldots,s_n,a_n))=\clP(\clS|s_n,a_n)$$ for all $\clS\in \clB(S)$ and $\clA\in \clB(A)$.
In particular, $\{s_n\}_{n\in \mathbb{N}_0}$ is a Markov chain with kernel $P_{\pi}\in \mathcal{P}(S|S)$ given by
$$
P_{\pi}(ds'|s)=\int_{A}P(ds'|s,a')\pi(da'|s)\,.
$$
For given $s\in S$, we denote $\mathbb{E}^{\pi}_{s}=\mathbb{E}^{\pi}_{\delta_s}$, where $\delta_s\in \mathcal{P}(S)$ denotes the Dirac measure at $s\in S$. 

Let $\Pi_{0}= \mathcal{P}(A|S)$ and $$\Pi_{\tau}=\{\exp(F(s,a))\mu(da)\in \mathcal{P}_{\mu}(A|S) | F\in B_b(S\times A)\}$$ if $\tau>0$.
For a given policy $\pi\in\Pi_{\tau}$, we define the $\tau$-entropy regularized value function $V^{\pi}_{\tau}:S\rightarrow \mathbb{R}$  by
$$
V^{\pi}_{\tau}(s)= \mathbb{E}_{s}^{\pi}\left[\sum_{n=0}^{\infty}\gamma^{n} \left(r(s_n,a_n)-\tau\,\ln\frac{d\pi}{d\mu}(a_n|s_n)\right)\right]\,.
$$
For a given policy $\pi\in\Pi_{\tau}$, we define the regularized state-action value function $Q^{\pi}_{\tau}\in B_b(S\times A)$ by
\begin{equation}\label{def:Qpi}
	Q^{\pi}_{\tau}(s,a)=r(s,a)+\gamma\int_{S}V^{\pi}_{\tau}(s')P(ds'|s,a)\,.
\end{equation}
The occupancy kernel $d^{\pi}\in \mathcal{P}(S|S)$ is defined by
$$
d^{\pi}(ds'|s)=(1-\gamma)\sum_{n=0}^{\infty}\gamma^n P^n_{\pi}(ds'|s)\,,
$$
 where $P^0_{\pi}(ds'|s):=\delta_s(ds')$, $P^n_{\pi}$ is understood as a product of kernels, and convergence is understood in $b\mathcal{K}(S|S)$. 
It follows that for all $s\in S$, 
\begin{align}
	V^{\pi}_{\tau}(s)&=\int_{A}\left(r(s,a)-\tau\ln\frac{d\pi}{d\mu}(a|s)\right.\notag \\
	&\qquad \quad 	 \left.+\gamma\int_{S}V^{\pi}_{\tau}(s')P(ds'|s,a)\right)\pi(da|s) \label{eq:policy_Bellman}\\
	&=\frac{1}{1-\gamma}\int_{S}\int_{A}\left(r(s',a')\right. \notag \\
	&\qquad	\left.-\tau \ln \frac{d\pi}{d\mu}(a'|s')\right)\pi(da'|s')d^{\pi}(ds'|s)\,.\label{eq:value_function_occ}
\end{align}
For a given initial distribution $\rho\in\mathcal{P}(S)$,  we define $$V^{\pi}_{\tau}(\rho)=\int_{S}V^{\pi}_{\tau}(s)\rho(ds), \;  d^{\pi}_{\rho}(ds)=\int_{S}d^{\pi}(ds|s')\rho(ds')\,.$$

For each $(s,a)\in S\times A$, we define the measurable optimal value and state-action value functions by
$$
V^{\ast}_{\tau}(s)=\sup_{\pi \in \Pi_{\tau}}V^{\pi}_{\tau}(s) \quad \textnormal{and} \quad  Q^*_{\tau}(s,a)=\sup_{\pi \in \Pi_{\tau}}Q^{\pi}_{\tau}(s,a)\,.
$$
By virtue of \cite{hernandez2012discrete}[Thm.\ 4.2.3] in the $\tau=0$ case and \cite{geist2019theory}[Thm.\ 1] in the $\tau>0$ case (using \cite{dupuis1997weak}[Prop.\ 1.4.2.), we obtain the following dynamical programming principle. See, also, \cite{haarnoja2017reinforcement}[Thm.\ 1 and 2].

\begin{theorem}[Dynamic programming principle for $\tau\ge 0$]\label{thm:DPP} The optimal value function $V^{\ast}_{\tau}\in B_b(S)$ is the unique solution
	of the regularized Bellman equation given by
	\begin{equation}
	\begin{aligned}\label{eq:BellmanDPP}
		V^{\ast}_{\tau}(s) &= \max_{m \in \mathcal P(A)}\left[\int_{A}\left(r(s,a)-\tau\ln\frac{dm}{d\mu}(a)\right.\right.\\
		&\qquad\quad  \left. \left.+\gamma\int_{S}V^{\ast}_{\tau}(s')P(ds'|s,a)\right)m(da)\right]\,.
	\end{aligned}
\end{equation}
	Moreover, 
	$
	Q^*_{\tau}(s,a)=r(s,a)+\gamma\int_{S}V^*_{\tau}(s')P(ds'|s,a)
	$.
	
	If $\tau=0$, then there exists a measurable function  $f^{\ast}:S\rightarrow A$ called a selector such that $f^{\ast}(s)\in \operatorname{argmax}_{a\in A}Q^*_{0}(s,a)$ and the induced policy $\pi^* \in \Pi_{0}$ defined by $\pi^*(da|s)=\delta_{f^*(s)}(da)$ for all $s\in S$ satisfies  $V^*_{0}=V^{\pi^*}_{0}$.
	
	If $\tau>0$, then for all $s\in S$, 
	$$
	V^{\ast}_{\tau}(s)=\tau\ln\int_{A}\exp\left\{Q^{\ast}_{\tau}(s,a)/\tau\right\}\mu(da),
	$$
	and $\pi^*_{\tau} \in \Pi_{\tau}$ defined by
	$$
	\pi^*_{\tau}(da|s) = \exp\left((Q^{\ast}_{\tau}(s,a)-V^{\ast}_{\tau}(s))/\tau\right)\mu(da)
	$$
	is the unique policy satisfying  $V^*_{\tau}=V^{\pi^*_{\tau}}_{\tau}$.
\end{theorem}

\subsection{Softmax Mean-Field Policy and the Entropy Regularized Objective}
\label{sec:param_MDP}
For arbitrarily given $k\in \bbN_0$, let  $\clA_k$ consist of jointly measurable functions $h: \mathbb{R}^d\times S\times A\rightarrow \mathbb{R}$ such that for all $s\in S$, $\mu$-almost-all $a\in A$, $h$ is $k$-times  differentiable in $\theta$, and satisfies
$$
|h|_{\clA_k}:=\max_{0\le j\le k}\underset{a\in A}{\operatorname{ess\, sup}}\sup_{s\in  S}\underset{\theta \in \bbR^d}{\operatorname{ess\, sup}}|\nabla^{j} h(\theta,s,a)|<\infty\,,
$$
where the essential supremum over $A$ is defined relative to the reference measure $\mu$, the essential supremum over $\bbR^d$ is defined relative to the Lebesgue measure $\lambda$, and here and henceforth $\nabla=\nabla_{\theta}$.
For  given $f\in \clA_0$, define  $\pi: \mathcal{P}(\mathbb{R}^d)\rightarrow  \mathcal{P}_{\mu}(A|S)$ by
$$
\pi_{\nu}(da|s)=\frac{\exp\left(\int_{\mathbb{R}^d}f(\theta,s,a)\nu(d\theta)\right)}{\int_{A}\exp\left(\int_{\mathbb{R}^d}f(\theta,s,a')\nu(d\theta)\right)\mu(da')}\mu(da)\,,
$$
which we refer to as a mean-field softmax policy.
It follows that $\ln(\frac{d\pi_{\nu}}{d\mu})\in B_b(S\times A)$ (see Lemma \ref{lem:bounds}), and hence $\pi_{\nu} \in \Pi_{\tau}$ for $\tau>0$.

If $S\times A$ is a compact subset of $\mathbb{R}^{d_S}\times \mathbb{R}^{d_A}$, then we may take $f$ to be of the form
\[
f(\theta, s, a) = f(\theta, x)
=\psi(\theta_0)\cdot g(\theta_1 x)\,,
\]
where  $x=(s_1,\ldots,s_{d_S}, a_1,\ldots,a_{d_A},1)^\top\in \mathbb{R}^{d_S+d_A+1}$, $\theta = (\theta_0,\theta_1) \in \mathbb{R}^{d'}\times \mathbb{R}^{d'\times (d_S+d_A+1)}$, $g:\mathbb{R}\rightarrow \mathbb{R}$ is an activation function such as a sigmoid or hyperbolic tangent applied component-wise, and $\psi:\mathbb{R}\rightarrow [-C, C]$, $C>0$, is a smooth rescaling function. 
We see that  $f\in \mathcal A_k$ for any $k\in \mathbb N_0$.
Here, we may also take the reference measure $\mu$ to be the Lebesgue measure.

We now introduce a  measurable positive potential $U:\bbR^d\rightarrow \bbR$ satisfying
$
\int_{\mathbb{R}^d}e^{-U(\theta)}d\theta =1.
$
The prototypical example one should have in mind is the quadratic potential $U(\theta)=\frac{d}{2}\ln(2\pi)+\frac{1}{2}|\theta|^2$. 
For given $\rho \in \mathcal{P}(S)$, $\tau\ge0$, and $\sigma\ge0$, we define the entropy regularized functional $J^{\tau,\sigma}(\rho): \mathcal{P}(\mathbb{R}^d)\rightarrow \mathbb{R}$ by 
\begin{align*}
	J^{\tau,\sigma}(\rho)(\nu)&=V^{\pi_\nu}_{\tau}(\rho)-\frac{\sigma^2}{2}\text{KL}(\nu|e^{-U})\,.
\end{align*}
Henceforth we fix $\rho\in \clP(S)$ and write $J^{\tau,\sigma}=J^{\tau,\sigma}(\rho)$, unless otherwise specified.
In the case $\sigma>0$, we maximize $J^{\tau,\sigma}$ over $\mathcal P_2(\mathbb R^d)$ and, without loss of generality, it suffices to maximize over 
\[
\mathcal{P}_2^{\textnormal{fe}}(\mathbb R^d) := \{ \nu \in \mathcal{P}_2^{\textnormal{ac}}(\mathbb R^d) : \textnormal{KL}(\nu|e^{-U})<\infty\}\,.
\]

\subsection{Statement of Main Results}

Our main results are Theorems~\ref{thm:optimality}, \ref{thm:apriori_grad_flow},~\ref{thm:ergodicity}, and \ref{thm:cts_dep_flow}, which establish  necessary conditions for optimality, the well-posedness of a gradient flow along which the objective increases, convergence of the gradient flow in a regularized regime, and upper bounds on $W_2$-distance between two gradient flows with different $\sigma,\tau,$ and initial conditions.
Before stating these, we highlight our core auxiliary results, namely Lemmas \ref{lem:func_deriv_pi} and \ref{lem:func_deriv_J} and Theorem \ref{thm:bnd_reg_Lion_deriv}, which  enable us to apply methods from the analysis of non-linear Fokker--Planck PDEs and McKean--Vlasov dynamics to the MDP problem described in Section \ref{sec:MDP}, and specifically the parameterized problem described in Section \ref{sec:param_MDP}.

\begin{lemma}[Functional derivative of $\pi$]\label{lem:func_deriv_pi}
	If $f\in \clA_0$, then the function $\pi: \mathcal{P}(\mathbb{R}^d)\rightarrow b\mathcal{K}_{\mu}(A|S)$ has a linear functional derivative $\frac{\delta \pi}{\delta \nu}:\mathcal{P}(\mathbb{R}^d)\times \mathbb{R}^d\rightarrow b\mathcal{K}_{\mu}(A|S)$  given by 
	\begin{equation}
	\begin{aligned}\label{eq:func_deriv_pi}
		\frac{\delta \pi}{\delta \nu}(\nu,\theta)&(da|s)=\left(f(\theta,s,a)\right.\\
		&\left. -\int_{A}f(\theta,s,a')\pi_{\nu}(da'|s)\right)\pi_{\nu}(da|s)\,.
	\end{aligned}
\end{equation}
\end{lemma}
Lemma~\ref{lem:func_deriv_pi} is proved in Section~\ref{sec:proofof:lem:func_deriv_pi}. 

\begin{lemma}[Functional derivative of $J^{\tau,0}$]
	\label{lem:func_deriv_J}
	If $f\in \clA_1$, then
	the function $J^{\tau,0}: \mathcal{P}_1(\mathbb{R}^d)\rightarrow \mathbb{R}$ has a linear functional derivative $\frac{\delta J^{\tau,0}}{\delta \nu}: \mathcal{P}_1(\mathbb{R}^d)\times \mathbb{R}^d\rightarrow \mathbb{R}$ given by 
	\begin{equation}
	\begin{aligned}\label{eq:func_deriv_J}
		\frac{\delta J^{\tau,0}}{\delta \nu}&(\nu,\theta)=\frac{1}{1-\gamma}\int_{S}\int_{A}\left(Q^{\pi_{\nu}}_{\tau}(s,a)\right.\\
		&\left.-\tau\ln\frac{d\pi_{\nu}}{d\mu}(a|s)\right)\frac{\delta \pi}{\delta \nu}(\nu,\theta)(da|s)d^{\pi_{\nu}}_{\rho}(ds)\,.
	\end{aligned}
    \end{equation}
\end{lemma}
Lemma~\ref{lem:func_deriv_J} is proved in Section~\ref{sec:proofof:lem:func_deriv_J}.
If $\tau=0$, then \eqref{eq:func_deriv_J} is  the policy gradient theorem in \cite{sutton2018reinforcement}. Noting that $J^{\tau,0}(\nu)= V^{\pi_{\nu}}_{\tau}(\rho)$ and understanding $V^{\cdot}_{\tau}(\rho): \clP(A|S)\rightarrow \bbR$ defined by $\pi \mapsto V^{\pi}_{\tau}$, \eqref{eq:func_deriv_J} can be interpreted as the chain rule
$$
\frac{\delta J^{\tau,0}}{\delta \nu}(\nu,\theta)=\int_{S\times A}\frac{\delta V^{\pi_{\nu}}_{\tau}(\rho)}{\delta \pi}(ds|a)\frac{\delta \pi}{\delta \nu}(\nu,\theta)(da|s)\,,
$$
where $\frac{\delta V^{\cdot}_{\tau}(\rho)}{\delta \pi}: b\clP(A|S)\rightarrow  b\clK(S|A)$ understood in the sense of Definition~\ref{def:lin_func_der_kernel} is given by
\begin{gather*}
\frac{\delta V^{\pi}_{\tau}(\rho)}{\delta \pi}(ds|a)\\
=\frac{1}{1-\gamma}\left(Q^{\pi}_{\tau}(s,a)-\tau\ln \frac{d\pi}{d\mu}(a|s)\right)d^{\pi}_{\rho}(ds)\,.
\end{gather*}

\begin{theorem}[Boundedness and regularity of $\frac{\delta J^{\tau,0}}{\delta \nu}$]
	\label{thm:bnd_reg_Lion_deriv}
	There are constants 
	$
	C_k=C(\gamma,|r|_{B_b(S\times A)},\tau,\mu(A),|f|_{\clA_k}), \; k\in \bbN\,,
	$
	$
	L=L(\gamma,|r|_{B_b(S\times A)},\tau,\mu(A),|f|_{\clA_1})\, ,
	$
	and $D=D(\gamma,\mu(A),|f|_{\clA_1})$
	such that for all $\tau,\tau'\ge 0$, $\theta\in \mathbb{R}^d$, $\nu,\nu'\in \mathcal{P}_1(\mathbb{R}^d)$, and $k\in \bbN$,
	\begin{gather}
		\left|\nabla^k \frac{\delta J^{\tau,0}}{\delta \nu}(\nu,\theta)\right|\le C_k\,,  \label{ineq:bound_J0}\\
		 |J^{\tau,0}(\nu')-J^{\tau,0}(\nu)| \le  C_1 W_1(\nu',\nu)\,,\notag\\
		\left|\nabla \frac{\delta J^{\tau,0}}{\delta \nu}(\nu',\theta)-\nabla \frac{\delta J^{\tau,0}}{\delta \nu}(\nu,\theta)\right| \le L W_1(\nu',\nu)\,,\notag\\
		\textrm{and} \quad \left|\nabla \frac{\delta J^{\tau',0}}{\delta \nu}(\nu,\theta)-\nabla \frac{\delta J^{\tau,0}}{\delta \nu}(\nu,\theta)\right| \le D|\tau'-\tau|\,.\notag
	\end{gather}
\end{theorem}
Theorem~\ref{thm:bnd_reg_Lion_deriv} is proved in Section~\ref{sec:proof_of:thm:bnd_reg_Lion_deriv}. 
As a result of \eqref{ineq:bound_J0}, we note $J^{\tau,\sigma}(\nu)=J^{\tau, \sigma}(\nu')$ if $\nu=\nu'$, $\lambda$-a.e.. We will always work under the following assumption on $U$.

\begin{assumption}[Growth of $\nabla U$]\label{asm:U_lin_growth}
	There is a constant $C_U>0$ such that for all $\theta\in \bbR^d$, $|\nabla U(\theta)|\le C_U(1+|\theta|)$.
\end{assumption}

\begin{theorem}[Necessary condition for optimality]
	\label{thm:optimality}
	Let $f\in \clA_1$ and Assumption \ref{asm:U_lin_growth} hold. If
	$\nu\in \mathcal{P}_1(\mathbb R^d)$ is a local maximum of $J^{\tau,\sigma}$, then
	$$
	\theta \mapsto \frac{\delta J^{\tau.0}}{\delta \nu}(\nu,\theta)-\frac{\sigma^2}{2}  U(\theta)-\frac{\sigma^2}{2}\ln\nu(\theta)
	$$
	is constant $\nu$-a.e.. Moreover, if $\sigma>0$, then $\nu$ is equivalent to the Lebesgue measure $\lambda$, and for $\lambda$-a.a.\ $\theta\in \bbR^d$,
	\begin{equation}\label{eq:max_exp_form}
		\nu(\theta)=\mathcal{Z}^{-1}e^{\frac{2}{\sigma^2}\frac{\delta J^{\tau.0}}{\delta \nu}(\nu,\theta)-U(\theta)}\,,
	\end{equation}
$$
 \quad \textnormal{where} \quad 
\mathcal{Z}:=\int_{\bbR^d} e^{\frac{2}{\sigma^2}\frac{\delta J^{\tau.0}}{\delta \nu}(\nu,\theta')-U(\theta')}\,d\theta'\,.
$$
\end{theorem}
Theorem~\ref{thm:optimality} is proved in Section~\ref{sec:proofof:thm:optimality}. Assume $f\in \clA_1$ and
for given $\nu\in  \clP_1(\mathbb{R}^d)$, define the linear operator $L_{\nu}$ on $C_c^\infty(\mathbb{R}^d)$ by 
\begin{equation*}
L_{\nu} \phi =\frac{\sigma^2}{2}\Delta \phi  +\left(\nabla \frac{\delta J^{\tau,0}}{\delta \nu}(\nu) - \frac{\sigma^2}{2} \nabla U\right)\cdot \nabla \phi\,. 
\end{equation*}
We denote the adjoint of $L_{\nu}$ by $L_{\nu}^{\ast}$, which acts on probability measures $\clP(\bbR^d)$.

\begin{corollary}[Local maximum and elliptic PDE]\label{cor:local_max_PDE}
Let $f\in \clA_1$ and Assumption \ref{asm:U_lin_growth} hold. If
$\nu\in \mathcal{P}_1(\mathbb R^d)$ is a local maximum of $J^{\tau,\sigma}$ and has full support, then $\nu$ is a measure-valued solution of  
\begin{equation}\label{eq:stationary_eq}
L_{\nu}^* \nu = 0\,.
\end{equation}
\end{corollary}
Corollary~\ref{cor:local_max_PDE} is proved in Section~\ref{sec:proofof:thm:optimality}. We will also require the following additional assumption on $U$.
\begin{assumption}[Lipschitzness of $\nabla U$]\label{asm:GradU_Lip}
	There exists a constant $L_U>0$ such that for all $\theta,\theta'\in \bbR^d$, $|\nabla U(\theta)-\nabla U(\theta')|\le L_U|\theta-\theta'|$.
\end{assumption}

The following theorem establishes precise conditions for the well-posedness of \eqref{eq:gf_intro} and for the objective function to be increasing along the flow. The formula \eqref{eq:objective_along_flow}, even in $\sigma=0$ case, is the first of its kind in the literature. In order to establish \eqref{eq:objective_along_flow} for non-smooth and non-compactly-supported initial conditions $\nu_0$, we have extended the argument in \cite{bogachev2016distances}. 

\begin{theorem}[Gradient flow]\label{thm:apriori_grad_flow}
	Let $f\in \clA_1$ and Assumption \ref{asm:U_lin_growth} hold. If $\nu_0\in\clP_p(\bbR^d)$ for some $p\in \bbN$, then there exists a measure-valued solution $\nu \in C(\bbR_+;\clP_p(\bbR^d))$ of
	\begin{equation}\label{eq:gradient_flow_cauchy}
			\partial_t \nu_t = L_{\nu_t}^*\nu_t\,, \quad 	\nu|_{t=0}=\nu_0\,.
	\end{equation}
	If  $\sigma = 0$ or  $\nu_0\in \clP_2^{\textnormal{fe}}(\bbR^d)$, then 
	\begin{equation}\label{eq:objective_along_flow}
	\begin{aligned}
		&J^{\tau,\sigma}(\nu_t)= J^{\tau,\sigma}(\nu_0)\\
		& +\int_0^t\int_{\mathbb{R}^d}\left|\nabla \frac{\delta J^{\tau,0}}{\delta \nu}(\nu_s,\theta) -\frac{\sigma^2}{2}\nabla \ln \frac{\nu_s(\theta)}{e^{-U(\theta)}}\right|^2\nu_{s}(d\theta) \,ds\,.
	\end{aligned}
	\end{equation}
	Moreover, if  both $f\in \clA_2$ and Assumption \ref{asm:GradU_Lip} holds or both $\sigma>0$ and $p\ge 4$ holds, then the solution of~\eqref{eq:gradient_flow_Cauchy_proof} is unique.
\end{theorem}
Theorem~\ref{thm:apriori_grad_flow} will be proved in Section~\ref{sec:proof_of_energy_identity}.

\begin{theorem}[McKean--Vlasov SDE well-posedness]
	\label{thm:MVSDE}
	Let  $f\in \clA_2$ and Assumptions \ref{asm:U_lin_growth} and \ref{asm:GradU_Lip} hold. Let $(\Omega, \clF, \bbF=(\clF_t)_{t\in \bbR_+}, \bbP)$ denote a filtered probability triple supporting an $\bbF$-adapted Wiener process $(W_t)_{t\in \bbR_+}$ and $\clF_0$-measurable random variable $\theta_0$ independent of $(W_t)_{t\in \bbR_+}$. Then there exists a unique continuous $\bbF$-adapted  solution $\theta: \Omega\times  \bbR_+\rightarrow \mathbb{R}^d$ of the McKean--Vlasov SDE
	\begin{equation}\label{eq:MVSDE}
	d\theta_t= \left(\nabla \frac{\delta J^{\tau,0}}{\delta \nu}(\textnormal{Law}(\theta_t),\theta_t) - \frac{\sigma^2}{2} \nabla U(\theta_t)\right)dt+\sigma dW_t\,,
	\end{equation}
where  $ \theta|_{t=0}=\theta_0$.  Moreover, $(\nu_t)_{t\ge 0}:=(\operatorname{Law}(\theta_t))_{t\ge 0}$ is the unique solution of~\eqref{eq:gradient_flow_cauchy} with $\nu_0:=\operatorname{Law}(\theta_0)$. 
\end{theorem}

If $f\in \clA_2$ and Assumptions \ref{asm:U_lin_growth} and \ref{asm:GradU_Lip} hold, then by Theorem~\ref{thm:bnd_reg_Lion_deriv}, we can conclude by \eqref{ineq:b_growth}, \eqref{ineq:b_lip_measure}, and \eqref{ineq:b_lip_theta} that the drift of \eqref{eq:MVSDE} has linear growth and is Lipschitz continuous in $\nu$ and $\theta$ (see Corollary \ref{cor:b_growth_and_Lip}).
Theorem~\ref{thm:MVSDE} then follows from~\cite{carmona2018probabilistic}[Thm.\ 4.21], It\^o's formula (\cite{krylov2008controlled}[Ch 2. Sec 10]) and Theorem~\ref{thm:apriori_grad_flow}.

The following theorem says that in the highly regularized regime, there is a unique solution of \eqref{eq:stationary_eq} which is a global optimizer of $J^{\tau,\sigma}$, and we have exponential convergence of the gradient flow \eqref{eq:gradient_flow_cauchy} to this unique optimizer.

\begin{assumption}[Dissipativity of $\nabla U$]\label{asm:GradU_diss} There exists a constant $\kappa>0$ such that for all $\theta,\theta' \in \mathbb{R}^d$,
	$$
	\left(\nabla U(\theta) - \nabla U(\theta')\right)\cdot (\theta - \theta')\ge \kappa |\theta -\theta'|^2.
	$$
\end{assumption}

\begin{theorem}[Convergence  in the regularized regime]\label{thm:ergodicity}
	Let $f\in \clA_2$ and Assumptions \ref{asm:U_lin_growth}, \ref{asm:GradU_Lip}, and \ref{asm:GradU_diss} hold. 
	Assume further that $\beta := \frac{\sigma^2}{2}\kappa-C_2-L>0$,  where $C_2$ and $L$ are the constants given in Theorem \ref{thm:bnd_reg_Lion_deriv}.
	Then there exists a unique solution $\nu^{\ast}$ of \eqref{eq:stationary_eq} which is the global maximizer $\nu^{\ast}$ of  $J^{\tau,\sigma}$ in $\mathcal{P}_2(\mathbb R^d)$. Moreover, if $(\nu_t)_{t\ge 0}$ is the solution of \eqref{eq:gradient_flow_cauchy} for a given $\nu_0\in \clP_2(\bbR^d)$, then for all  $t\in \bbR_+$,
	$$
	W_2(\nu_t,\nu^{\ast})\le e^{-\beta t}W_2(\nu_0,\nu^{\ast})\,.
	$$
\end{theorem}
Theorem~\ref{thm:ergodicity} will be proved in Section~\ref{sec:proof_of_exp_conv}.
The exponential convergence holds also in the total variation norm~\cite{butkowsky2013ergodic}[Thm.\ 3.1], in $W_1$~\cite{bogachev2018distances}[Remark 4.2], and in $W_p$ using the method of Theorem~\ref{thm:cts_dep_flow} with the It\^o formula applied to higher powers, see also~\cite{vsivska2020gradient}.  From the general theory of nonlinear Fokker--Planck--Kolmogorov equations, one only expects existence and uniqueness of solutions to the stationary equation if $\sigma>0$ is large relative to the constants appearing in Theorem \ref{thm:bnd_reg_Lion_deriv}  (see, e.g., \cite{bogachev2019convergence}[Ex.\ 1.1], \cite{bogachev2018distances}[Ex.\ 4.3], \cite{manita2015uniqueness}[Sec.\ 6]). The exponential convergence of the policy gradient flow for large $\sigma$ can be regarded as an interpolation between the neural tangent kernel regime, where the neural network  can be linearized around the initialization/prior, and the mean-field regime, where the distribution is evolving with the training time (see \cite{mei2019mean}).

The following theorem estimates the sensitivity of solutions of \eqref{eq:gradient_flow_cauchy} and \eqref{eq:stationary_eq} on $\nu_0$, $\tau$, and $\sigma$.

\begin{theorem}[Sensitivity in $W_2$]
	\label{thm:cts_dep_flow}
	Let $f\in \clA_2$ and Assumptions \ref{asm:U_lin_growth}, \ref{asm:GradU_Lip}, and \ref{asm:GradU_diss} hold. Let $\sigma,\tau,\sigma',\tau'\ge 0$ and $\nu_0,\nu_0'\in \clP_2(\bbR^d)$. Let $(\nu_t)_{t\ge 0}$ and $(\nu'_t)_{t\ge 0}$ be the solutions of \eqref{eq:gradient_flow_cauchy} with parameters and initial data  $\sigma, \tau, \nu_0$ and $\sigma',\tau',\nu_0'$, respectively. Then for all $\ell>0$ and  $t\in \bbR_+,$
	\begin{equation}\label{ineq:W2_dist_flow_main}
		\begin{aligned}
			&W_2^2(\nu_t,\nu_t')\le  e^{-2\beta_{\ell} t} W_2^2(\nu_0,\nu_0')\\
	&\quad+ \frac{|\sigma^2-\sigma'^{2}|}{8\ell}\int_0^t\int_{\bbR^d}e^{2\beta_{\ell} (s-t)} |\nabla U(\theta)|^2\nu_s'(d\theta)\,ds\\
	&\quad + \frac{1}{2\beta_{\ell}}\left(D|\tau-\tau'|+d|\sigma-\sigma'|^2  \right)(1-e^{-2\beta_{\ell} t})\,,
		\end{aligned}
	\end{equation}
	where $\beta_{\ell}:=\frac{\sigma^2}{2}\kappa -C_2(\tau)-L(\tau)-\ell|\sigma^2-\sigma'^{2}|$ and $C_2(\tau)$ and $L(\tau)$ are the constants obtained in Theorem~\ref{thm:bnd_reg_Lion_deriv}.
	Moreover, if $\beta:=\frac{\sigma^2}{2}\kappa -C_2(\tau)-L(\tau)>0$ and $\nu^{\ast}$ and $\nu'^{\ast}$ are solutions of \eqref{eq:stationary_eq}  with  $\sigma, \tau$ and  $\sigma',\tau'$, respectively, then for all $\ell>0$ such that $\beta_{\ell}=\beta-\ell|\sigma^2-\sigma'^{2}|>0$, we have
	\begin{equation}\label{ineq:W2_dist_stationary_main}
		\begin{aligned}
		&W_2^2(\nu^{\ast},\nu'^{\ast})\le  \frac{|\sigma^2-\sigma'^{2}|}{16\ell \beta_{\ell}}\int_{\bbR^d}|\nabla U(\theta)|^2\nu'^{\ast}(d\theta)\\
		&\quad + \frac{1}{2\beta_{\ell}}\left(D|\tau-\tau'|+d|\sigma-\sigma'|^2  \right)\,.
	\end{aligned}
	\end{equation}
\end{theorem}
Theorem~\ref{thm:cts_dep_flow} will be proved in Section~\ref{sec:proof_of_cts_dep}.  Assumption \ref{asm:U_lin_growth} and Lemma \ref{lem:moment_est}, proved in Section \ref{sec:proof_of_cts_dep}, yield bounds on $\int_{\bbR^d}|\nabla U(\theta)|^2\nu'_t(d\theta)$ for all $\sigma',\tau'\ge 0$ and $\nu'_0\in \clP_2(\bbR^d)$ and bounds on $\int_{\bbR^d}|\nabla U(\theta)|^2\nu'^{\ast}(d\theta)$ for all $\sigma',\tau'\ge 0$ such that either  $\sigma'^2\kappa>2C_2(\tau')$ or both  $\nabla U(0)=0$ and $\sigma'>0$.

As a corollary of Theorems \ref{thm:bnd_reg_Lion_deriv} and  \ref{thm:cts_dep_flow}, we  obtain the sensitivity with respect to the  MDP value function. 

\begin{corollary}[Sensitivity of value function]\label{cor:sens_value}
	Under the assumptions of Theorem \ref{thm:cts_dep_flow} with $(\nu_t)_{t\ge 0}$, $(\nu'_t)_{t\ge 0}$, $\nu^{\ast}$, and $\nu'^{\ast}$ defined accordingly, for all $\hat{\tau}\ge 0$, 
	$$
	|V^{\pi_{\nu_t}}_{\hat{\tau}}(\rho)-V^{\pi_{\nu_t'}}_{\hat{\tau}}(\rho)|\le C_1(\hat{\tau})W_2(\nu_t,\nu_t')
	$$
	$$ \quad \textrm{and} \quad |V^{\pi_{\nu^{\ast}}}_{\hat{\tau}}(\rho)-V^{\pi_{\nu'^{\ast}}}_{\hat{\tau}}(\rho)|\le  C_1(\hat{\tau})W_2(\nu^{\ast},\nu'^{\ast})\,,
	$$
	where $W_2(\nu_t,\nu_t')$ and $W_2(\nu^{\ast},\nu'^{\ast})$ can be estimated by the square-roots of \eqref{ineq:W2_dist_flow_main} and \eqref{ineq:W2_dist_stationary_main}, respectively.
\end{corollary}

As discussed in the literature review, the authors of \cite{agazzi2020global} showed that if $(\nu^{0}_t)_{t\in \bbR_+}$ is a solution of \eqref{eq:gradient_flow_cauchy} with $\sigma=0$  and $\tau>0$ such that $\nu^{0}_t$ converges to a solution $\nu^{0,\ast}$ of \eqref{eq:stationary_eq} in $W_2$ with full support, then $V^{\pi_{\nu^{0,\ast}}}_{\tau}(s)=V^*_{\tau}(s)$ for all $s\in S$, provided $f$ is expressive enough (see \cite{agazzi2020global}[Asm.\ 1]). 
Letting $\nu^{\sigma,\ast}$ be as in Theorem \ref{thm:ergodicity}, by Corollary \ref{cor:sens_value}, we find that for all $\ell>0$,
\[
\begin{split}
& |V^{\pi_{\nu^{\sigma,\ast}}}_{\tau}(\rho)-V^*_{\tau}(\rho)|\\
& \le  \frac{\sigma C_1}{4\sqrt{\ell\beta}}\left(\int_{\bbR^d}|\nabla U(\theta)|^2\nu^{0,\ast}(d\theta)\right)^{\frac{1}{2}}+ \frac{\sigma C_1}{\sqrt{2\beta}}\,,
\end{split}
\]
where $\beta:=\frac{\sigma^2}{2}\kappa -C_2-L-\ell|\sigma^2|>0$. 
However, as discussed above, is not clear why, in the setting of~\cite{agazzi2020global}, one would expect $\lim_{t\rightarrow \infty}\nu^{0}_t= \nu^{0,\ast}$ in $W_2$. 

\section{Examples}

\subsection{The Effect of Regularization on Rate of Convergence}

Consider the bandit setting. Let $S=\emptyset$, $A=\mathbb{R}^d$, $\gamma=0$, $\mu(da)=da$, $\ell\in \mathbb R^d$, $\lambda>0$, and $r(a)=\ell \cdot a -\lambda |a|^2.$ While this example does not formally satisfy our assumptions due to unboundedness, we could extend our analysis to include this example with more cumbersome assumptions. The optimal policy in this setting is
$
\pi_{\tau}^{\ast}(da) \sim \mathcal{N}(\frac{\ell}{2\lambda}, \frac{\tau}{2\lambda}I_{d})
$
for $\tau>0$.
We let $f(\theta,a)=-\frac{\lambda}{\tau} |a-\theta|^2$, which implies  that 
$\pi_{\nu}\sim \mathcal{N}(\int_{\mathbb{R}^d}\theta \nu(d\theta),\frac{\tau}{2\lambda}I_{d})$. 
For a given $m_U\in \mathbb{R}^d$ and $\sigma_U>0,$ we let  $U(\theta)=\frac{d}{2}\ln 2\pi \sigma_U^2 +\frac{1}{2\sigma_U^2}|\theta-m_U|^2.$ The first-order condition for the  objective
\begin{equation*}
\begin{aligned}
J^{\tau,\sigma}(\nu)&= \int_{\mathbb{R}^d}\theta^{\top}\ell \nu(d\theta) - \lambda \left(\frac{\tau}{2\lambda}+\left|\int_{\mathbb{R}^d} \theta \nu(d\theta)\right|^2\right)\\
&\qquad  + \frac{\tau d}{2}\left(\ln 2\pi \sigma_f^2+1\right)-\frac{\sigma^2}{2}\textnormal{KL}(\nu|e^{-U})
\end{aligned}
\end{equation*}
is $
\int_{\mathbb{R}^d} \theta \nu^{\ast}(d\theta)=\frac{\ell}{2\lambda}
$ if $\sigma=0$ and 
$$
\nu^{\ast}
\sim e^{-\frac{1}{2\sigma_U^2}\left|\theta-\frac{m_U\sigma^2}{\sigma^2+4\lambda\sigma_U^2}-\frac{2\sigma_U^2\ell}{\sigma^2+4\lambda\sigma_U^2}\right|^2}
$$
if $\sigma>0$, where we have solved for the mean in the RHS of \eqref{eq:max_exp_form} using that the measure $\nu^{\ast}$ is Gaussian. We see that if $\sigma=0$, then there are infinitely many critical points. 
Moreover, if $\sigma >0$, then for any critical point, we have
$$\pi_{\nu^{\ast}}\sim \mathcal{N}\left(\frac{m_U\sigma^2}{\sigma^2+4\lambda\sigma_U^2}+\frac{2\sigma_U^2\ell}{\sigma^2+4\lambda\sigma_U^2},\frac{\tau}{2\lambda}I_{d}\right).$$
We see that if $\sigma=0$, then any critical point satisfies $\pi_{\nu^{\ast}}=\pi^{\ast}_{\tau}$, and thus $\nu^{\ast}\in \textnormal{argmax}_{\nu\in \mathcal P_2(\mathbb R^d)}V^{\pi_{\nu}}$.
For $m_U=0$ and small $\sigma>0$, the mean of $\pi_{\nu^{\ast}}$ is approximately $\ell/2\lambda$, and hence $\pi_{\nu^{\ast}} \approx \pi_{\tau}^{\ast}$, which Theorem \ref{thm:cts_dep_flow} and Corollary \ref{cor:sens_value} illustrate. 
Furthermore, the probabilistic representation of the gradient flow \eqref{eq:MVSDE} is given by
\begin{equation*}
d\theta_t = \left( \ell -2\lambda \mathbb{E}\theta_t-\frac{\sigma^2}{2\sigma_U^2}(\theta_t-m_U)\right)dt +\sigma dW_t\,
\end{equation*}
which is an Ornstein--Uhlenbeck-like process; the mean can be solved for explicitly and substituted back into the equation. For $\sigma>0,$ the process has a unique invariant measure and  converges exponentially fast, which can be shown using the same proof as that of Theorem \ref{thm:ergodicity}. We expect a similar phenomenon in the linear quadratic regulator setting, provided the function $f$ is chosen appropriately so that its mean is $Kx$, where $K$ is the control gain. For more complicated examples without any special structure (e.g., convexity) on the unregularized objective, we expect that one needs to take $\sigma$ larger to obtain exponential convergence at the level of the parameterization.

\subsection{Explicit Constants in Theorem~\ref{thm:bnd_reg_Lion_deriv}} 
The following example, while a bit contrived, illustrates the general procedure to  determine the constants in Theorem~\ref{thm:bnd_reg_Lion_deriv}.
The bandit setting is characterized by $S=\emptyset$ and $\gamma=0$ with a bounded reward function $r:A\rightarrow \mathbb{R}.$
We let 
$
f(\theta,a)=\psi(\theta)\cdot \tanh(g(a)),
$
where $\psi:\mathbb{R}\rightarrow \mathbb{R}$ is a smooth rescaling function such that $\psi$ and its derivatives up to order two are bounded by $|\psi|_{\infty}$,  and where for simplicity we take $g: A\rightarrow \mathbb{R}^d$ to be a random feature map. For given $m\in \mathbb{R}^d$ and  $\Sigma\in \mathbb{R}^{d\times d}$ satisfying $\Sigma^{-1}  \succeq \kappa I_{d}$, we  let $U(\theta):=2^{-1}\ln \textnormal{det}(2\pi \Sigma)+(\theta - m_U)^{\top}\Sigma^{-1}(\theta - m_U).
$
In this setting,  using Lemma \ref{lem:func_deriv_pi} and \ref{lem:func_deriv_J}, we find
\begin{equation*}
\begin{aligned}
&\nabla^k_{\theta}\frac{\delta J^{\tau,0}}{\delta \nu}(\nu,\theta)\\
&=\nabla^k_{\theta}\psi(\theta)\int_{A}\left(r(a)-\tau\tanh(g(a))\cdot  \int_{\mathbb{R}^d} \psi(\theta') \nu(d\theta')\right)\\
& \qquad \times \left( \tanh(g(a))-\int_{A}\tanh(g(a'))\pi_{\nu}(da')\right)\pi_{\nu}(da)\,.
\end{aligned}
\end{equation*}
Using \eqref{ineq:Lip_pi}, we obtain
$$
C=C_k := 2 |\psi|_{\infty} (|r|_{\infty}+\tau |\psi|_{\infty}), \quad k\in \mathbb{N},
$$
$$
\textnormal{ and } \qquad  L:=(6|r|_{\infty}+2\tau)|\psi|_{\infty}^2 + 6\tau |\psi|_{\infty}^3,
$$
and hence $
\frac{\sigma^2}{2}\kappa > C+L
$ in Theorem \ref{thm:ergodicity}
is  explicit.

\section{Conclusion}
\label{sec:conc_and_future_work}

We identified conditions that allow us to extend the work of~\cite{mei2020global}, where exponential convergence of the policy gradient method has been established in the tabular case to the continuous state and action setting with policies parameterized by two-layer neural networks in the mean-field regime. 
This was enabled by introducing entropic regularization in the space of parameterizations ($\sigma > 0)$ rather than just space of policies as is commonly done in entropy regularized in RL and careful analysis of the corresponding non-linear Fokker--Planck--Kolmogorov equations. The results and techniques of this paper open up many possible research directions of which we mention a few. It should be possible to extend the one-hidden-layer mean-field setting to recurrent neural network approximation, see~\cite{weinan2017proposal, hu2019meanode, jabir2019mean}. 
Moreover, it should be possible to extend the techniques identified here to actor-critic-type algorithms where the policy and the $Q$ function are approximated by their own mean-field neural networks~\cite{agarwal2020optimality,sirignano2019asymptotics}. Furthermore, since the main focus of RL is in the regime where the model is not known, it would be interesting to explore the non-idealized setting, where regret bounds are proved in terms of the number of samples of state and action pairs.

\section*{Acknowledgements}
We would like to thank the reviewers for their thoughtful comments and efforts toward improving our manuscript. JML is grateful for support from US AFOSR Grant FA8655-21-1-7034. LS is grateful for support from support from UKRI Prosperity Partnership Scheme (FAIR) under the EPSRC Grant EP/V056883/1.

\bibliography{softmax_policy_gradient}
\bibliographystyle{icml2022}

\newpage
\appendix
\onecolumn
\section{Appendix}

\subsection{Proofs}\label{sec:proofs}

\subsubsection{Auxiliary results}

We need the following version of the fundamental theorem of calculus of variations. 
\begin{lemma}[Lem.\ 33 in \cite{jabir2019mean}]\label{lem:ft_var}
	Let $\nu\in \mathcal{P}(\mathbb{R}^d)$ and $u: \mathbb{R}^d\rightarrow \mathbb{R}$ be measurable. Assume that $\int_{\bbR^d}u(\theta)\nu(d\theta)$ exists and is finite. If for all $\nu'\in \mathcal{P}(\mathbb{R}^d)$
	\begin{equation}\label{ineq:ft_var}
		\int_{\mathbb{R}^d}u(\theta)(\nu'-\nu)(d\theta)\ge 0\,,
	\end{equation}
	then $u$ is a constant $\nu$-a.e.. Moreover, if $\nu\in \mathcal{P}_2^{\textnormal{fe}}(\bbR^d)$ and \eqref{ineq:ft_var} holds for all $\nu'\in  \mathcal{P}_2^{\textnormal{fe}}(\bbR^d)$, then the conclusion still holds.
\end{lemma}

We will also require the following property of relative entropy.
\begin{lemma}\label{lem:relative_entropy}
	Let $U:\bbR\to \mathbb R_+$ be measurable such that 
	$
	\int_{\mathbb{R}^d}e^{-U(\theta)}d\theta =1.
	$
	Let $\gamma(d\theta) = e^{-U(\theta)}d\theta$.
	Moreover, assume that there exist constants $C>0$ and $p \in \mathbb N$  such that $|U(\theta)|\le C_U(1+|\theta|^p)$ for all $\theta\in \bbR^d. $
	Then for all $\nu\in\clP_p(\bbR^d)$,
	$$
	\int_{\bbR^d} \ln \frac{d\nu}{d\gamma}(\theta)\nu(d\theta)  < \infty \quad \Leftrightarrow \quad  \int_{\bbR^d}  \left|\ln \frac{d\nu}{d\gamma}(\theta)\right|\nu(d\theta)  <\infty\,,
	$$
	$$
	\Leftrightarrow \quad \int_{\bbR^d} \ln \frac{d\nu}{d\theta}(\theta)\nu(d\theta)<\infty \quad \Leftrightarrow \quad  \int_{\bbR^d}  \left|\ln \frac{d\nu}{d\theta}(\theta)\right|\nu(d\theta)  <\infty\,.$$
\end{lemma}
\begin{proof}
	First notice that
	$$
	\int_{\bbR^d} \ln \frac{d\nu}{d\gamma}(\theta)\nu(d \theta)  =\int_{\bbR^d} \left(\ln \frac{d\nu}{d\gamma}(\theta)\right)^+\nu(d\theta)  -\int_{\bbR^d} \frac{d\nu}{d\gamma}(\theta) \left(\ln \frac{d\nu}{d\gamma}(\theta)\right)^- \gamma(d\theta)
	$$
	and
	$$
	\int_{\bbR^d} \left|\ln \frac{d\nu}{d\gamma}(\theta)\right|\nu(d\theta)  =\int_{\bbR^d} \left(\ln \frac{d\nu}{d\gamma}(\theta)\right)^+ \nu(d\theta)  +\int_{\bbR^d} \frac{d\nu}{d\gamma}(\theta) \left(\ln \frac{d\nu}{d\gamma}(\theta)\right)^- \gamma(d\theta),
	$$
	where for all $x\in \bbR$,  $x^+ = \max(x,0)$ and $x^{-} = - \min(x,0)$.  
	Using the fact that the function  $g(s)=s (\ln s)^-$ is bounded on $\bbR_+$, we find that 
	$$
	\int_{\bbR^d} \ln \frac{d\nu}{d\gamma}(\theta)\nu(d \theta) <\infty \quad \Leftrightarrow \quad 
	\int_{\bbR^d} \left|\ln \frac{d\nu}{d\gamma}(\theta)\right|\nu(d\theta)<\infty\,.
	$$
	The triangle and reverse triangle inequalities imply that 
	$$
	\int_{\bbR^d} \left|\ln  \frac{d\nu}{d\theta}(\theta)\right| \nu(d\theta)  - \int_{\bbR^d} U(\theta)  \nu(d\theta) \le \int_{\bbR^d} \left|\ln  \frac{d\nu}{d\gamma}(\theta)\right|\nu(d\theta)   \le \int_{\bbR^d}  \left|\ln  \frac{d\nu}{d\theta}(\theta)\right| \nu(d\theta)  + \int_{\bbR^d} U(\theta) \nu(d\theta)\,.
	$$
	By assumption, we have
	$$
	\int_{\bbR^d} U(\theta)\nu(d\theta)  \le C \int_{\bbR^d} (1+|\theta|)^p \nu(d\theta)< \infty
	$$
	and hence
	$$
	\int_{\bbR^d} \left|\ln  \frac{d\nu}{d\gamma}(\theta)\right| \nu(d\theta) <\infty \quad \Leftrightarrow \quad \int_{\bbR^d} \nu(\theta) \left|\ln \frac{d\nu}{d\theta}(\theta)\right| \nu(d\theta) <\infty\,.
	$$
	Combining the above equivalences, we complete the proof.
\end{proof}

\subsubsection{Auxiliary bounds}

In the proof of the main results, we will need the following auxiliary lemma, which concerns the boundedness of the i) log-density (relative to reference measure $\mu$) of the class of mean-field policies, ii) value function and iii) state-action value function. 

\begin{lemma}\label{lem:bounds}
	Assume that $f\in \clA_0$. For all  $\tau,\tau'\ge 0$, $\nu\in \mathcal P(\mathbb R^d)$, $s\in S$, and $\mu-a.e. \,a\in A$, we have
	\begin{gather*}
		\left|\ln \frac{d\pi_{\nu}}{d\mu}(a|s)\right| \le 2|f|_{\clA_0} +| \ln \mu(A)|\,,\quad 
		|V^{\pi_\nu}_{\tau}(s)|\le \frac{1}{1-\gamma} \left(|r|_{B_b(S\times A)}+\tau \left(2|f|_{\clA_0} +| \ln \mu(A)|\right)\right)\,,\\
		|Q^{\pi_{\nu}}_{\tau}(s,a)| \le \frac{1}{1-\gamma} \left(|r|_{B_b(S\times A)}+\gamma\tau \left(2|f|_{\clA_0} +| \ln \mu(A)|\right)\right)\,,\\
		|V^{\pi_\nu}_{\tau'}(s)-V^{\pi_\nu}_{\tau}(s)| \le \frac{ |\tau'-\tau|}{1-\gamma}\left(2|f|_{\clA_0} +| \ln \mu(A)|\right)\,,\\
		\textrm{and} \quad |Q^{\pi_\nu}_{\tau'}(s,a)-Q^{\pi_\nu}_{\tau}(s,a)| \le \frac{ \gamma |\tau'-\tau|}{1-\gamma}\left(2|f|_{\clA_0} +| \ln \mu(A)|\right)\,.
	\end{gather*}
\end{lemma}
\begin{proof}
	Let $a\in A$, $s\in S$, and $\nu \in \mathcal{P}(\bbR^d)$ be arbitrarily given. Estimating directly, we find
	\begin{equation}\label{ineq:bound_on_log_pi}
		\begin{split}
			\left|\ln \frac{d\pi_{\nu}}{d\mu}(a|s)\right| &=\left| \int_{\mathbb{R}^d} f(\theta,s,a) \nu(d\theta)  - \ln\left(\int_{A}\exp\left(\int_{\mathbb{R}^d}f(\theta,s,a')\nu(d\theta)\right)\mu(da')\right)\right| \\
			&\le 2|f|_{\clA_0} +| \ln \mu(A)|\,.
		\end{split}
	\end{equation}
	Using \eqref{eq:value_function_occ}, \eqref{ineq:bound_on_log_pi}, $|\pi_{\nu}|_{b\mathcal{K}(A|S)}=1$ and $|d^{\pi_{\nu}}|_{b\mathcal{K}(S|S)}=1$, we  obtain
	\begin{equation}\label{eq:bound_on_V_pi}
		\begin{split}
			|V^{\pi_\nu}_{\tau}(s)|&=\frac{1}{1-\gamma}\left|\int_{S}\int_A \left(r(s',a)-\tau \ln \frac{d\pi_{\nu}}{d\mu}(a|s)\right)d^{\pi_\nu}(ds'|s)\pi_{\nu}(da|s')\right|\\
			&\le \frac{1}{1-\gamma}  \left(|r|_{B_b(S\times A)}+\tau \left|\ln \frac{d\pi_{\nu}}{d\mu}\right|_{B_b(S\times A)}\right)|\pi_{\nu}|_{b\mathcal{K}(A|S)} |d^{\pi_{\nu}}|_{b\mathcal{K}(S|S)}\\
			&\le \frac{1}{1-\gamma} \left(|r|_{B_b(S\times A)}+\tau \left(2|f|_{\clA_0} +| \ln \mu(A)|\right)\right)\,.
		\end{split}
	\end{equation}
	To estimate the state-action value function, we use \eqref{def:Qpi}, \eqref{eq:bound_on_V_pi}, and $|P|_{b\mathcal{K}(S|S\times A)}=1$ to get 
	\begin{align*}
		|Q^{\pi_{\nu}}_{\tau}(s,a)| &\le |r|_{B_b(S\times A)} + \gamma |V^{\pi_{\nu}}_{\tau}|_{B_b(S)}|P|_{b\mathcal{K}(S|S\times A)}\\
		&\le |r|_{B_b(S\times A)} +  \frac{\gamma}{1-\gamma} \left(|r|_{B_b(S\times A)}+\tau \left(2|f|_{\clA_0} +| \ln \mu(A)|\right)\right)\\
		&= \frac{1}{1-\gamma} \left(|r|_{B_b(S\times A)}+\gamma\tau \left(2|f|_{\clA_0} +| \ln \mu(A)|\right)\right)\,.
	\end{align*}
	The remaining inequalities are derived similarly using
	$$
	V^{\pi_\nu}_{\tau'}(s)-V^{\pi_\nu}_{\tau}(s)
	=\frac{\tau-\tau'}{1-\gamma}\int_{S}\int_A\ln \frac{d\pi_{\nu}}{d\mu}(a|s)\pi_{\nu}(da|s')d^{\pi}(ds'|s)\,,
	$$
	which follows from \eqref{eq:value_function_occ}.
\end{proof}

\subsubsection{Proof of Lemma \ref{lem:func_deriv_pi} and a corollary}
\label{sec:proofof:lem:func_deriv_pi}

\begin{proof}
	Let $\nu, \nu' \in \mathcal{P}(\mathbb{R}^d)$ and define $\nu^{\varepsilon}=\nu+\varepsilon (\nu'-\nu)$ for $\varepsilon\in[0,1]$. 
	We must show that 
	\begin{equation}\label{eq:wts_func_deriv_pi}
		\underset{\varepsilon \in [0,1]}{\lim_{\varepsilon\rightarrow 0}}\frac{\pi_{\nu^{\varepsilon}}(da|s)-\pi_{\nu}(da|s)}{\varepsilon}= \int_{\mathbb{R}^d} \frac{\delta \pi_{\nu}}{\delta \nu}(\nu,\theta)(da|s)(\nu'-\nu)(d\theta)\,,
	\end{equation}
	where $\frac{\delta \pi_{\nu}}{\delta \nu}(\nu,\theta)$ is given by the right-hand-side of \eqref{eq:func_deriv_pi} and the limit is understood in $b\mathcal{K}(A|S)$. 
	Recall from Section \ref{sec:notation} that
	\begin{align}
		& \left|\frac{\pi_{\nu^{\varepsilon}}-\pi_{\nu}}{\varepsilon}-\int_{\mathbb{R}^d} \frac{\delta \pi_{\nu}}{\delta \nu}(\nu,\theta)(da|s)(\nu'-\nu)(d\theta)\right|_{b\mathcal{K}(A|S)}\notag\\
		& =\sup_{s\in S}\int_{A}\left|\frac{\frac{d\pi_{\nu^{\varepsilon}}}{d\mu}(\cdot|s)-\frac{d\pi_{\nu}}{d\mu}(\cdot|s)}{\varepsilon}-\int_{\mathbb{R}^d} \frac{d\frac{\delta \pi_{\nu}}{\delta \nu}(\nu,\theta)}{d\mu}(a|s)(\nu'-\nu)(d\theta)\right|\mu(da)\,.\label{eq:func_deriv_pi_diff}
	\end{align}
	For convenience, we introduce the unnormalized policy $\tilde{\pi}: \mathcal{P}(\mathbb{R}^d) \rightarrow b\mathcal{K}_{\mu}(A|S)$ given by
	$$
	\tilde{\pi}(\nu)(da|s)=\tilde{\pi}_{\nu}(da|s)=\exp\left(\int_{\mathbb{R}^d}f(\theta,s,a)\nu(d\theta)\right)\mu(da)\,.
	$$
	For all $(s,a)\in S\times A$, we have
	\begin{align*}
		& \frac{d\pi_{\nu^{\varepsilon}}}{d\mu}(a|s)-\frac{d\pi_{\nu}}{d\mu}(a|s) \\
		& =\frac{\frac{d\tilde{\pi}_{\nu^{\varepsilon}}}{d\mu}(a|s)}{\tilde{\pi}_{\nu^{\varepsilon}}(A|s)}-\frac{\frac{d\tilde{\pi}_{\nu}}{d\mu}(a|s)}{\tilde{\pi}_{\nu}(A|s)}=\frac{\frac{d\tilde{\pi}_{\nu^{\varepsilon}}}{d\mu}(a|s)\tilde{\pi}_{\nu}(A|s)}{\tilde{\pi}_{\nu^{\varepsilon}}(A|s)\tilde{\pi}_{\nu}(A|s)}-\frac{\frac{d\tilde{\pi}_{\nu}}{d\mu}(a|s)\tilde{\pi}_{\nu^{\varepsilon}}(A|s)}{\tilde{\pi}_{\nu^{\varepsilon}}(A|s)\tilde{\pi}_{\nu}(A|s)}\\
		& =\frac{\frac{d\tilde{\pi}_{\nu^{\varepsilon}}}{d\mu}(a|s)-\frac{d\tilde{\pi}_{\nu}}{d\mu}(a|s)}{\tilde{\pi}_{\nu}(A|s)}\frac{\tilde{\pi}_{\nu}(A|s)}{\tilde{\pi}_{\nu^{\varepsilon}}(A|s)}+\frac{\frac{d\tilde{\pi}_{\nu}}{d\mu}(a|s)}{\tilde{\pi}_{\nu^{\varepsilon}}(A|s)}\frac{\tilde{\pi}_{\nu}(A|s)-\tilde{\pi}_{\nu^{\varepsilon}}(A|s)}{\tilde{\pi}_{\nu}(A|s)}\\
		& =\left[\frac{\frac{d\tilde{\pi}_{\nu^{\varepsilon}}}{d\mu}(a|s)-\frac{d\tilde{\pi}_{\nu}}{d\mu}(a|s)}{\tilde{\pi}_{\nu}(A|s)}+\frac{d\pi_{\nu}}{d\mu}(a|s)\frac{\tilde{\pi}_{\nu}(A|s)-\tilde{\pi}_{\nu^{\varepsilon}}(A|s)}{\tilde{\pi}_{\nu}(A|s)}\right]\frac{\tilde{\pi}_{\nu}(A|s)}{\tilde{\pi}_{\nu^{\varepsilon}}(A|s)}\\
		& =\left[\frac{\frac{d\tilde{\pi}_{\nu^{\varepsilon}}}{d\mu}(a|s)-\frac{d\tilde{\pi}_{\nu}}{d\mu}(a|s)}{\tilde{\pi}_{\nu}(A|s)}+\frac{d\pi_{\nu}}{d\mu}(a|s)\int_{A}\left(\frac{\frac{d\tilde{\pi}_{\nu^{\varepsilon}}}{d\mu}(a|s)-\frac{d\tilde{\pi}_{\nu}}{d\mu}(a|s)}{\tilde{\pi}_{\nu}(A|s)}\right)\mu(da)\right]\frac{\tilde{\pi}_{\nu}(A|s)}{\tilde{\pi}_{\nu^{\varepsilon}}(A|s)}\,.
	\end{align*}
	Simple manipulation yields
	$$
	\frac{d\tilde{\pi}_{\nu^{\varepsilon}}}{d\mu}(a|s)-\frac{d\tilde{\pi}_{\nu}}{d\mu}(a|s) =\exp\left(\int_{\mathbb{R}^d}f(\theta,s,a)\nu(d\theta)\right)\left( \exp\left(\varepsilon\int_{\mathbb{R}^d}f(\theta,s,a)(\nu'-\nu)(d\theta)\right)-1\right)\,.
	$$
	Taylor expanding the exponential function, we find
	\begin{align*}
		& \exp\left(\varepsilon\int_{\mathbb{R}^d}f(\theta,s,a)(\nu'-\nu)(d\theta)\right)-1\\
		& = \varepsilon\int_{\mathbb{R}^d}f(\theta,s,a)(\nu'-\nu)(d\theta)+\varepsilon^2\sum_{n=2}^\infty\varepsilon^{n-2} \frac{\left(\int_{\mathbb{R}^d}f(\theta,s,a)(\nu'-\nu)(d\theta)\right)^n}{n!}\,,
	\end{align*}
	and hence
	\begin{equation}
		\begin{aligned}\label{ineq:ratio_diff_pol_eps}
			\varepsilon^{-1}\frac{\frac{d\tilde{\pi}_{\nu^{\varepsilon}}}{d\mu}(a|s)-\frac{d\tilde{\pi}_{\nu}}{d\mu}(a|s)}{\tilde{\pi}_{\nu}(A|s)}&= \frac{d\pi_{\nu}}{d\mu}(a|s)\int_{\mathbb{R}^d}f(\theta,s,a)(\nu'-\nu)(d\theta)\\
			&\quad +\varepsilon \frac{d\pi_{\nu}}{d\mu}(a|s) \sum_{n=2}^\infty \varepsilon^{n-2}\frac{\left(\int_{\mathbb{R}^d}f(\theta,s,a)(\nu'-\nu)(d\theta)\right)^n}{n!}\,.
		\end{aligned}
	\end{equation}
	Thus, using  $\pi_{\nu}(A|s)=1$, we obtain
	\begin{gather*}
		\int_{A}\left|\varepsilon^{-1}\frac{\frac{d\tilde{\pi}_{\nu^{\varepsilon}}}{d\mu}(a|s)-\frac{d\tilde{\pi}_{\nu}}{d\mu}(a|s)}{\tilde{\pi}_{\nu}(A|s)}- \frac{d\pi_{\nu}}{d\mu}(a|s)\int_{\mathbb{R}^d}f(\theta,s,a)(\nu'-\nu)(d\theta)\right|\mu(da)\notag\\
		\le  \varepsilon\exp\left(|f|_{\clA_0}|\nu'-\nu|_{\mathcal{M}(\mathbb{R}^d)}\right)
	\end{gather*}
	and
	$$
	\int_{A}\varepsilon^{-1}\left|\frac{\frac{d\tilde{\pi}_{\nu^{\varepsilon}}}{d\mu}(a|s)-\frac{d\tilde{\pi}_{\nu}}{d\mu}(a|s)}{\tilde{\pi}_{\nu}(A|s)}\right|\mu(da) \le |f|_{\clA_0}|\nu'-\nu|_{\mathcal{M}(\mathbb{R}^d)}+\varepsilon \exp\left(|f|_{\clA_0}|\nu'-\nu|_{\mathcal{M}(\mathbb{R}^d)}\right).
	$$
	The dominated convergence theorem implies that 
	\begin{equation}\label{eq:limit_policy}
		\underset{\varepsilon \in [0,1]}{\lim_{\varepsilon\rightarrow 0}}|\pi_{\nu^{\varepsilon}}-\pi_{\nu}|_{b\mathcal{K}(A|S)}=\underset{\varepsilon \in [0,1]}{\lim_{\varepsilon\rightarrow 0}}\sup_{s\in S}\int_{A}\left|\frac{d\pi_{\nu^{\varepsilon}}}{d\mu}(\cdot|s)-\frac{d\pi_{\nu}}{d\mu}(\cdot|s)\right|\mu(da)=0\,,
	\end{equation}
	and hence
	$$
	\underset{\varepsilon \in [0,1]}{\lim_{\varepsilon\rightarrow 0}}\frac{\tilde{\pi}_{\nu}(A|s)}{\tilde{\pi}_{\nu^{\varepsilon}}(A|s)}=1\,.
	$$
	We also have the bound
	$$
	\sup_{s\in S}\frac{\tilde{\pi}_{\nu}(A|s)}{\tilde{\pi}_{\nu^{\varepsilon}}(A|s)} \le \exp(2|f|_{\clA_0})\,.
	$$
	Putting together the above bounds and limits, we find that limit of \eqref{eq:func_deriv_pi_diff} as $\varepsilon\rightarrow 0$ is zero, and hence we obtain \eqref{eq:wts_func_deriv_pi}. 
	We must now show that $\frac{\delta \pi}{\delta \nu}$ is bounded and continuous. Boundedness follows from
	\begin{align*}
		\sup_{s\in S}\int_A\left|\frac{d\frac{\delta \pi}{\delta \nu}(\nu,\theta)(a|s)}{d\mu}\right|\mu(da) &= \sup_{s\in S}\int_A\left|\left(f(\theta,s,a)-\int_{A}f(\theta,s,a')\pi_{\nu}(da'|s)\right)\frac{d\pi_{\nu}(a|s)}{d\mu}\right|\mu(da)\\
		&\le 2|f|_{\clA_0}\,.
	\end{align*}
	Continuity of $\frac{\delta \pi}{\delta \nu}$ in the product topology then follows from the dominated convergence theorem along with the assumption $|f|_{\clA_0}<\infty$.
\end{proof}

We will now use Lemma \ref{lem:func_deriv_pi} to obtain Lipschitz continuity of the occupancy measure. First, we will show Lipschitz continuity of the occupancy measure with respect to stochastic policies.

\begin{lemma}\label{lem:lip_occ_kernel}
	For given $\pi,\pi' \in \mathcal{P}(A|S)$, we have
	$$
	|d^{\pi'}-d^{\pi}|_{b\mathcal{K}(S|S)}\le \frac{\gamma}{1-\gamma}|\pi'-\pi|_{b\mathcal{K}(A|S)}\,.
	$$
\end{lemma}
\begin{proof}[Proof of Lemma \ref{lem:lip_occ_kernel}] 
	It follows that
	$$
	d^{\pi'} - d^{\pi} =(1-\gamma)\sum_{n=0}^\infty\sum_{i=0}^{n-1} \gamma^n(P^{\pi'})^{n-i-1}(P^{\pi'}-P^{\pi})(P^{\pi})^{i}\,,
	$$
	where we understood both sides as elements of the Banach algebra $b\mathcal{K}(S|S)$, or equivalently as operators $\mathcal{L}(\mathcal{M}(S),\mathcal{M}(S))$. Thus, 
	\begin{align*}
		|d^{\pi'}-d^{\pi}|_{b\mathcal{K}(S|S)}&\le (1-\gamma)\sum_{n=0}^\infty\sum_{i=0}^{n-1} \gamma^n|P^{\pi'}|^{n-i-1}_{b\mathcal{K}(S|S)}|P^{\pi'}-P^{\pi}|_{b\mathcal{K}(S|S)}|P^{\pi}|_{b\mathcal{K}(S|S)}^i\\
		&\le (1-\gamma)\sum_{n=0}^\infty \gamma^n|P|_{b\mathcal{K}(S|A\times S)}|\pi'-\pi|_{b\mathcal{K}(A|S)}\\
		&= \frac{\gamma}{1-\gamma}|\pi'-\pi|_{b\mathcal{K}(A|S)}\,,
	\end{align*}
	which completes the proof.
\end{proof}

\begin{corollary}[Lipschitz continuity of the occupancy measure in the parameter measure]\label{cor:Lip_occ}
	For given $\nu,\nu'\in \mathcal{P}_1(\mathbb{R}^d)$, we have
	$$
	|d^{\pi_{\nu'}}-d^{\pi_{\nu}}|_{b\mathcal{K}(S|S)}\le 2|f|_{\clA_1}\frac{\gamma}{1-\gamma} W_1(\nu',\nu)\,.
	$$
\end{corollary}
\begin{proof}[Proof of Corollary \ref{cor:Lip_occ}]
	By Lemma \ref{lem:lip_occ_kernel} and \eqref{ineq:W1W2}, it is enough to show that
	\begin{equation}\label{ineq:Lip_pi}
		|\pi_{\nu'}-\pi_{\nu}|_{b\mathcal{K}(A|S)} \le 2|f|_{\clA_1} W_1(\nu',\nu)\,.
	\end{equation}
	We have 
	\begin{align*}
		|\pi_{\nu'}-\pi_{\nu}|_{b\mathcal{K}(A|S)}&=\sup_{s\in S}|\pi_{\nu'}(\cdot|s)-\pi_{\nu}(\cdot|s)|_{\mathcal{M}(A)}\\
		&=\sup_{s\in S}\underset{|h|_{B_b(A)}\le 1}{\sup_{h\in B_b(A)}}\int_{A}h(a)\left(\pi_{\nu'}-\pi_{\nu}\right)(da|s)\,.
	\end{align*}
	Let $\nu^{\varepsilon}=\varepsilon \nu'+(1-\varepsilon)\nu$, $\varepsilon\in [0,1]$. Let $s\in S$ and $h\in B_b(A)$ be arbitrarily given. Using Lemma \ref{lem:func_deriv_pi}, we find that 
	$$
	\int_{A}h(a)\left(\pi_{\nu'}-\pi_{\nu}\right)(da|s)=\int_{\mathbb{R}^d}g_s(\theta)(\nu'-\nu)(d\theta)\,,
	$$
	where
	\begin{align*}
		g_s(\theta):=\int_{A}h(a)\int_0^1\left(f(\theta,s,a)-\int_{A}f(\theta,s,a')\pi_{\nu^{\varepsilon}}(da'|s)\right)\,d\varepsilon\,\pi_{\nu^{\varepsilon}}(da|s)\,.
	\end{align*}
	Applying \eqref{ineq:essup_kernel}, we get that for all $\theta,\theta'\in \mathbb{R}^d$ and $s\in S$, 
	\begin{gather*}
		|g_s(\theta')-g_s(\theta)|\\
		= \left|\int_{A}h(a)\int_0^1\left(f(\theta',s,a)-f(\theta,s,a)+\int_{A}\left(f(\theta,s,a')-f(\theta',s,a')\right)\pi_{\nu^{\varepsilon}}(da'|s)    \right)\,d\varepsilon\,\pi_{\nu^{\varepsilon}}(da|s)\right|\\
		\le |h|_{B_b(A)}|\pi_{\nu^{\varepsilon}}|_{b\mathcal{K}(A|S)}\left(|f|_{\clA_1} +|f|_{\clA_1}|\pi_{\nu^{\varepsilon}}|_{b\mathcal{K}(A|S)} \right)|\theta'-\theta|
		\\ 
		\le 2|h|_{B_b(A)}|f|_{\clA_1}|\theta'-\theta|\,.
	\end{gather*}
	Using \eqref{eq:KRW1}, we obtain \eqref{ineq:Lip_pi},
	which completes the proof.
\end{proof}

\subsubsection{Proof of Lemma \ref{lem:func_deriv_J}}
\label{sec:proofof:lem:func_deriv_J}
\begin{proof}
	Let $\nu, \nu' \in \mathcal{P}_1(\mathbb{R}^d)$ and define $\nu^{\varepsilon}=\nu+\varepsilon (\nu'-\nu)$ for $\varepsilon\in[0,1]$. We must show that 
	\begin{equation}\label{eq:wts_func_deriv_J}
		\underset{\varepsilon \in [0,1]}{\lim_{\varepsilon\rightarrow 0}}\frac{J^{\tau,0}(\nu^{\varepsilon})-J^{\tau,0}(\nu)}{\varepsilon}= \int_{\mathbb{R}^d} \frac{\delta J^{\tau,0}}{\delta \nu}(\nu,\theta)(\nu'-\nu)(d\theta)\,,
	\end{equation}
	where  $\frac{\delta J}{\delta \nu}(\nu,\theta)$ is as specified in the statement of the lemma. Noting that  
	\begin{equation}\label{eq:Jtau_diff_quot}
		\frac{J^{\tau,0}(\nu^{\varepsilon})-J^{\tau,0}(\nu)}{\varepsilon} = \int_{S} \frac{V^{\pi_{\nu^{\varepsilon}}}_{\tau}(s)-V^{\pi_{\nu}}_{\tau}(s)}{\varepsilon}\rho(ds)\,,
	\end{equation}
	we first study the difference quotient
	$
	F^{\varepsilon}(s)=\frac{V^{\pi_{\nu^{\varepsilon}}}_{\tau}(s)-V^{\pi_{\nu}}_{\tau}(s)}{\varepsilon}\,.
	$
	By \eqref{eq:policy_Bellman}, we have
	\begin{equation}\label{eq:first_iterate_value}
		\begin{split}
			F^{\varepsilon}(s)&=\frac{1}{\varepsilon}\left[\int_{A}\left(r(s,a)+\gamma\int_{S}P(ds'|s,a)V^{\pi_{\nu^{\varepsilon}}}_{\tau}(s')-\tau\ln\frac{d\pi_{\nu^{\varepsilon}}}{d\mu}(a|s)\right)\pi_{\nu^{\varepsilon}}(da|s)\right.\\
			&\left.\qquad-\int_{A}\left(r(s,a)+\gamma\int_{S}P(ds'|s,a)V^{\pi_{\nu}}_{\tau}(s')-\tau\ln\frac{d\pi_{\nu}}{d\mu}(a|s)\right)\pi_{\nu}(da|s)\right]\\
			&=I^{\varepsilon}_1(s) + I^{\varepsilon}_2(s)+I^{\varepsilon}_3(s) +I^{\varepsilon}_4(s)+I^{\varepsilon}_5(s)\,,
		\end{split}
	\end{equation}
	where
	\begin{align*}
		I^{\varepsilon}_1(s)&:=\int_{A}r(s,a)\frac{\pi_{\nu^{\varepsilon}}(da|s)-\pi_{\nu}(da|s)}{\varepsilon}\,,\\
		I^{\varepsilon}_2(s)&:=\gamma\int_{A}\int_{S}\frac{V^{\pi_{\nu^{\varepsilon}}}_{\tau}(s')-V^{\pi_{\nu}}_{\tau}(s')}{\varepsilon}P(ds'|s,a)\pi_{\nu}(da|s)=\gamma\int_{S} F^{\varepsilon}(s')P_{\pi_{\nu}}(ds'|s)\,,\\
		I^{\varepsilon}_3(s)&:=\gamma\int_{A}\int_{S}V^{\pi_{\nu^{\varepsilon}}}_{\tau}(s')P(ds'|s,a)\frac{\pi_{\nu^{\varepsilon}}(da|s)-\pi_{\nu}(da|s)}{\varepsilon}\,,\\
		I^{\varepsilon}_4(s)&:=-\tau \int_{A}\frac{\ln\frac{d\pi_{\nu^{\varepsilon}}}{d\mu}(a|s) -\ln\frac{d\pi_{\nu}}{d\mu}(a|s) }{\varepsilon}\pi_{\nu^{\varepsilon}}(da|s)\,,\\
		I^{\varepsilon}_5(s)&:=-\tau\int_{A}\ln\frac{d\pi_{\nu}}{d\mu}(a|s)\frac{\pi_{\nu^{\varepsilon}}(da|s)-\pi_{\nu}(da|s)}{\varepsilon}\,. 
	\end{align*}
	Iterating \eqref{eq:first_iterate_value} (i.e., applying the usual policy gradient proof technique), we obtain
	\begin{align*}
		F^{\varepsilon}(s) &= I^{\varepsilon}_1(s) + I^{\varepsilon}_3(s)+I^{\varepsilon}_4(s)+I^{\varepsilon}_5(s) +\gamma  \int_{S}F^{\varepsilon}(s')P_{\pi_{\nu}}(ds'|s)\\
		&=\frac{1}{1-\gamma} \int_{S} (I^{\varepsilon}_1(s')+I^{\varepsilon}_3(s')+I^{\varepsilon}_4(s')+I^{\varepsilon}_5(s'))d^{\pi_{\nu}}(ds'|s)\,,
	\end{align*}
	and hence
	$$
	\frac{J^{\tau,0}(\nu^{\varepsilon})-J^{\tau,0}(\nu)}{\varepsilon} =\frac{1}{1-\gamma} \int_{S} (I^{\varepsilon}_1(s)+I^{\varepsilon}_3(s)+I^{\varepsilon}_4(s)+I^{\varepsilon}_5(s))d^{\pi_{\nu}}_{\rho}(ds)\,.
	$$
	We now will pass to the limit as $\varepsilon\rightarrow 0$. Let us begin with the $I_4^{\varepsilon}$-term. Recalling \eqref{eq:limit_policy}, we have
	\begin{equation}\label{eq:limit_policy_again}
		\underset{\varepsilon \in [0,1]}{\lim_{\varepsilon\rightarrow 0}}|\pi_{\nu^{\varepsilon}}-\pi_{\nu}|_{b\mathcal{K}(A|S)}=\underset{\varepsilon \in [0,1]}{\lim_{\varepsilon\rightarrow 0}}\sup_{s\in S}\int_{A}\left|\frac{d\pi_{\nu^{\varepsilon}}}{d\mu}(\cdot|s)-\frac{d\pi_{\nu}}{d\mu}(\cdot|s)\right|\mu(da)=0\,.
	\end{equation}
	Since
	$$
	\frac{d\pi_{\nu}}{d\mu}(a|s) = \tilde\pi_{\nu}(A)^{-1} \exp\left(\int_{A}f(\theta,s,a)\nu(d\theta)\right)\ge \frac{\exp(-2|f|_{\clA_0})}{\mu(A)}\,,
	$$
	there is an $\varepsilon_0\in (0,1]$ such that  for all $\varepsilon<\varepsilon_0$, $s\in S$, and $\mu-a.e. \,a\in A$,
	$$
	\left|\frac{\frac{d\pi_{\nu^{\varepsilon}}}{d\mu}(a|s)-\frac{d\pi_{\nu}}{d\mu}(a|s)}{\frac{d\pi_{\nu}}{d\mu}(a|s)}\right|<\frac{1}{2}\,.
	$$
	Taylor expanding the logarithm, we get 
	\begin{gather*}
		\ln\frac{d\pi_{\nu^{\varepsilon}}}{d\mu}(a|s) -\ln\frac{d\pi_{\nu}}{d\mu}(a|s) =\ln\left(1+\frac{\frac{d\pi_{\nu^{\varepsilon}}}{d\mu}(a|s)-\frac{d\pi_{\nu}}{d\mu}(a|s)}{\frac{d\pi_{\nu}}{d\mu}(a|s)}\right)\notag\\
		=\frac{1}{\frac{d\pi_{\nu}}{d\mu}(a|s)}\left(\frac{d\pi_{\nu^{\varepsilon}}}{d\mu}(a|s)-\frac{d\pi_{\nu}}{d\mu}(a|s)\right)+\sum_{n=2}^{\infty}(-1)^{n+1}\frac{\left(\frac{1}{\frac{d\pi_{\nu}}{d\mu}(a|s)}\left(\frac{d\pi_{\nu^{\varepsilon}}}{d\mu}(a|s)-\frac{d\pi_{\nu}}{d\mu}(a|s)\right)\right)^n}{n}\,.
	\end{gather*}
	Using \eqref{ineq:ratio_diff_pol_eps}, we find that 
	$$
	\frac{\frac{d\pi_{\nu^{\varepsilon}}}{d\mu}(a|s)-\frac{d\pi_{\nu}}{d\mu}(a|s)}{\frac{d\pi_{\nu}}{d\mu}(a|s)}=\varepsilon \int_{\mathbb{R}^d}f(\theta,s,a)(\nu'-\nu)(d\theta)+\varepsilon^2 \sum_{n=2}^\infty \varepsilon^{n-2}\frac{\left(\int_{\mathbb{R}^d}f(\theta,s,a)(\nu'-\nu)(d\theta)\right)^n}{n!}\,,
	$$
	which implies
	$$
	\left|\frac{\frac{d\pi_{\nu^{\varepsilon}}}{d\mu}(a|s)-\frac{d\pi_{\nu}}{d\mu}(a|s)}{\frac{d\pi_{\nu}}{d\mu}(a|s)}\right|\le \varepsilon |f|_{\clA_0}|\nu'-\nu|_{\mathcal{M}(\mathbb{R}^d)}+\varepsilon^2 \exp\left(|f|_{\clA_0}|\nu'-\nu|_{\mathcal{M}(\mathbb{R}^d)}\right)\,,
	$$
	Thus, by Lemma~\ref{lem:func_deriv_pi}, we have
	\begin{equation}
		\underset{\varepsilon \in [0,1]}{\lim_{\varepsilon\rightarrow 0}}\frac{\ln\frac{d\pi_{\nu^{\varepsilon}}}{d\mu}(a|s) -\ln\frac{d\pi_{\nu}}{d\mu}(a|s) }{\varepsilon}=\int_{\mathbb{R}^d}\left(f(\theta,s,a)-\int_{A}f(\theta,s,a')\pi_{\nu}(da'|s)\right)(\nu'-\nu)(d\theta)\,,\label{lem:func_deriv_log_pi}
	\end{equation}
	and that there exists a constant $M>0$ such that for all $\varepsilon<\varepsilon_0$, $s\in S$, and $\mu-a.e. \,a\in A$,
	$$
	\varepsilon^{-1}\left|\ln\frac{d\pi_{\nu^{\varepsilon}}}{d\mu}(a|s) -\ln\frac{d\pi_{\nu}}{d\mu}(a|s)\right| \le M\,.
	$$
	Therefore, owing to \eqref{eq:limit_policy_again} and \eqref{lem:func_deriv_log_pi}, we find
	$$
	\underset{\varepsilon \in [0,1]}{\lim_{\varepsilon\rightarrow 0}}I_4^{\varepsilon}(s)=-\tau\int_A \int_{\mathbb{R}^d}\left(f(\theta,s,a)-\int_{A}f(\theta,s,a')\pi_{\nu}(da'|s)\right)\pi_{\nu}(da|s)(\nu'-\nu)(d\theta)=0\,.
	$$
	We now turn our attention to  $I_3^{\varepsilon}$. Recalling \eqref{eq:value_function_occ}, we have 
	$$
	V^{\pi_{\nu^{\varepsilon}}}_{\tau}(s)=\frac{1}{1-\gamma}\int_{S}\int_A\left(r(s',a)-\tau\ln\frac{d\pi_{\nu^{\varepsilon}}}{d\mu}(a|s')\right)\pi_{\nu^{\varepsilon}}(da|s')d^{\pi_{\nu^{\varepsilon}}}(ds'|s)\,.
	$$
	It follows from Corollary \ref{cor:Lip_occ}, \eqref{eq:limit_policy_again}, Lemma \ref{lem:bounds}, and the boundedness of the reward $r$ that 
	$$
	\underset{\varepsilon \in [0,1]}{\lim_{\varepsilon\rightarrow 0}}V^{\pi_{\nu^{\varepsilon}}}_{\tau}(s)=V^{\pi_{\nu}}_{\tau}(s)\,.
	$$
	Thus, by Lemmas \ref{lem:bounds} and \ref{lem:func_deriv_pi}, we obtain 
	$$
	\underset{\varepsilon \in [0,1]}{\lim_{\varepsilon\rightarrow 0}}I_3^{\varepsilon}(s)=\gamma\int_A \int_S V^{\pi_{\nu}}_{\tau}(s')P(ds'|s,a)  \int_{\mathbb{R}^d}\frac{\delta \pi_{\nu}}{\delta \nu}(\nu,\theta)(da|s)(\nu'-\nu)(d\theta)\,.
	$$
	Using Lemmas \ref{lem:bounds} and \ref{lem:func_deriv_pi} and the boundedness of the reward, we get
	\begin{equation*}
		\underset{\varepsilon \in [0,1]}{\lim_{\varepsilon\rightarrow 0}}(I_1^{\varepsilon}(s)+I_5^{\varepsilon}(s))=\int_A \left(r(s',a)-\tau\ln\frac{d\pi_{\nu}}{d\mu}(a|s)\right)\int_{\mathbb{R}^d}\frac{\delta \pi_{\nu}}{\delta \nu}(\nu,\theta)(da|s)(\nu'-\nu)(d\theta)\,.
	\end{equation*}
	Putting it all together and using the definition of $Q^{\pi}$, we arrive at
	$$
	\underset{\varepsilon \in [0,1]}{\lim_{\varepsilon\rightarrow 0}}\sum_{j=1}^5 I_j^{\varepsilon}(s)=\int_A \left(Q_{\tau}^{\pi_{\nu}}(s,a)-\tau\ln\frac{d\pi_{\nu}}{d\mu}(a|s)\right)\int_{\mathbb{R}^d}\frac{\delta \pi_{\nu}}{\delta \nu}(\nu,\theta)(da|s)(\nu'-\nu)(d\theta)\,.
	$$
	Since $(I_j)_{1\le j\le 5}$ are bounded uniformly in $\varepsilon$, we may apply the bounded convergence theorem to pass to the limit in \eqref{eq:Jtau_diff_quot} to obtain \eqref{eq:wts_func_deriv_J}. Continuity of $\frac{\delta J^{\tau,0}}{\delta \nu}$ in the product topology then follows from the dominated convergence theorem and the continuity of $\pi=\pi(\nu)$, $d^{\pi_{\nu}}$ (see Corollary \ref{cor:Lip_occ}), and  $Q^{\pi_{\nu}}$ in $\nu$ (see \eqref{eq:value_function_occ}), and the joint continuity of $\frac{\delta \pi}{\delta \nu}$, which completes the proof.
\end{proof} 

\subsubsection{Proof of Theorem \ref{thm:bnd_reg_Lion_deriv}}
\label{sec:proof_of:thm:bnd_reg_Lion_deriv}

\begin{proof}
	For given $(s,a)\in S\times A$ and $\nu\in \mathcal{P}_1(\mathbb{R}^d)$, we denote 
	$$
	\bar{Q}^{\pi_{\nu}}_{\tau}(s,a)=\frac{1}{1-\gamma}\left(Q^{\pi_{\nu}}_{\tau}(s,a)-\tau\ln \frac{d\pi_{\nu}}{d\mu}(a|s)\right)\,.
	$$
	By Lemma \ref{lem:bounds}, for all $\tau\ge 0$, $\nu\in \mathcal P(\mathbb R^d)$, $s\in S$, and $\mu-a.e. \,a\in A$, we have
	\begin{equation}\label{eq:bound_on_Q-log}
		|\bar{Q}^{\pi_{\nu}}_{\tau}(s,a)|\le \frac{1}{(1-\gamma)^2} \left(|r|_{B_b(S\times A)}+\tau \left(2|f|_{\clA_0} +| \ln \mu(A)|\right)\right)\,.
	\end{equation}
	Thus, for all $\theta\in \mathbb{R}^d$, $\nu\in \mathcal{P}_1(\mathbb{R}^d)$ and $k\in \bbN$,
	\begin{align}
		\left|\nabla^k \frac{\delta J^{\tau,0}}{\delta \nu}(\nu,\theta)\right| & \le \int_{S}\int_{A}|\bar{Q}^{\pi_{\nu}}_{\tau}(s,a)|\left|\left(\nabla^{k} f(\theta,s,a)-\int_{A}\nabla^{k} f(\theta,s,a')\pi_{\nu}(da'|s)\right)\right|\pi_{\nu}(da|s)d^{\pi_{\nu}}_{\rho}(ds)\notag \\
		&\le\frac{2}{(1-\gamma)^2} \left(|r|_{B_b(S\times A)}+\tau \left(2|f|_{\clA_0} +| \ln \mu(A)|\right)\right) |f|_{\clA_k}\,,\label{ineq:proof_bound_J0}
	\end{align}
	which yields the first inequality in \eqref{ineq:bound_J0}. For arbitrarily given $\nu,\nu'\in \mathcal{P}_1(\mathbb{R}^d)$, define $\nu^{\varepsilon}=\nu+\varepsilon (\nu'-\nu)$, $\varepsilon\in[0,1]$. By  Lemma \ref{lem:func_deriv_J} and \eqref{eq:lin_der_diff_form} we have
	\begin{equation*}
		J^{\tau,0}(\nu')-J^{\tau,0}(\nu)=\int_0^1 \int_{\mathbb{R}^d} \frac{\delta J^{\tau,0}}{\delta \nu}(\nu^{\varepsilon}, \theta)(\nu'-\nu)(d\theta)\,d\varepsilon\,.
	\end{equation*}
	Applying \eqref{ineq:proof_bound_J0}, we obtain the second inequality in \eqref{ineq:bound_J0}
	\begin{equation*}
		|J^{\tau,0}(\nu')-J^{\tau,0}(\nu)|\le \sup_{\varepsilon \in [0,1],\theta\in \bbR^d}\left|\nabla \frac{\delta J^{\tau,0}}{\delta \nu}(\nu^{\varepsilon},\theta)\right|  W_1(\nu',\nu)\, .\end{equation*}
	For arbitrarily given $\theta\in \bbR^d$ and $\nu,\nu'\in \mathcal{P}_1(\mathbb{R}^d)$, we have
	\begin{equation*}
		\begin{split}
			& \nabla\frac{\delta J^{\tau,0}}{\delta \nu}(\nu',\theta)- \nabla \frac{\delta J^{\tau,0}}{\delta \nu}(\nu,\theta)=\int_{S}\int_{A}\bar{Q}^{\pi_{\nu'}}_{\tau}(s,a)\nabla \frac{\delta \pi}{\delta \nu}(\nu',\theta)(da|s)[d^{\pi_{\nu'}}_{\rho}-d^{\pi_{\nu}}_{\rho}](ds)(:=I_1)\\
			&+\int_{S}\int_{A}\left[\bar{Q}^{\pi_{\nu'}}_{\tau}(s,a)-\bar{Q}^{\pi_{\nu}}_{\tau}(s,a)\right]\nabla \frac{\delta \pi}{\delta \nu}(\nu',\theta')(da|s)d^{\pi_{\nu}}_{\rho}(ds)(:=I_2)\\
			&+ \int_{S}\int_{A}\bar{Q}^{\pi_{\nu}}_{\tau}(s,a)\left[\nabla \frac{\delta \pi}{\delta \nu}(\nu',\theta)(da|s)-\nabla \frac{\delta \pi}{\delta \nu}(\nu,\theta)(da|s)\right]d^{\pi_{\nu}}_{\rho}(ds)\\
			&= I_1+I_2 +\int_{S}\int_{A}\bar{Q}^{\pi_{\nu}}_{\tau}(s,a)\int_{A}\nabla  f(\theta,s,a')[\pi_{\nu'}-\pi_{\nu}](da'|s)\pi_{\nu'}(da|s)d^{\pi_{\nu}}_{\rho}(ds)(:=I_3)\\
			&+ \int_{S}\int_{A}\bar{Q}^{\pi_{\nu}}_{\tau}(s,a)\left(\nabla  f(\theta,s,a) -\int_{A}\nabla  f(\theta,s,a')\pi_{\nu'}(da'|s)\right)[\pi_{\nu'}-\pi_{\nu}](da|s)d^{\pi_{\nu}}_{\rho}(ds)(:=I_4)\,.
		\end{split}
	\end{equation*} 
	We will now estimate $(I_j)_{1\le j\le 4}$. Applying Corollary \ref{cor:Lip_occ} and \eqref{eq:bound_on_Q-log}, we find
	$$
	|I_1| \le  \frac{4\gamma}{(1-\gamma)^3} \left(|r|_{B_b(S\times A)}+\tau \left(2|f|_{\clA_1} +| \ln \mu(A)|\right)\right) |f|_{\clA_1}^2 W_1(\nu',\nu)\,.
	$$
	By virtue of \eqref{ineq:Lip_pi} and the duality description of the $b\mathcal{K}(A|S)$-norm, for arbitrarily given measurable $h:\mathbb{R}^d\times S\times A\rightarrow \mathbb{R}$, we have 
	$$
	\int_{A}h(\theta,s,a)[\pi_{\nu'}-\pi_{\nu}](da|s) \le 2|f|_{\clA_1} |h|_{L^\infty} W_1(\nu',\nu)\,.
	$$
	Thus, applying \eqref{ineq:Lip_pi} and \eqref{eq:bound_on_Q-log}, we deduce
	$$
	|I_3+I_4| \le \frac{4}{(1-\gamma)^2} \left(|r|_{B_b(S\times A)}+\tau \left(2|f|_{\clA_0} +| \ln \mu(A)|\right)\right) |f|_{\clA_1}^2W_1(\nu',\nu)\,.
	$$
	It remains to estimate $I_2$, and, in particular, to study the  Lipschitzness of $\bar{Q}_{\tau}^{\pi_\nu}$ in $\nu$, and hence the Lipschitzness of $Q_{\tau}^{\pi_\nu}$  and  $\ln \frac{d\pi_{\nu}}{d\mu}$ separately. 
	By \eqref{def:Qpi} and the fact that $$V^{\pi_{\nu}}_{\tau}(s)=J^{\tau.0}(\delta_s)(\nu), \;\;\forall s\in S\,,$$ using  Lemma \ref{lem:func_deriv_J}, \eqref{eq:lin_der_diff_form}, and  \eqref{ineq:proof_bound_J0} (with $\rho=\delta_s$), we find 
	\begin{align*}
		Q^{\pi_{\nu'}}_{\tau}(s,a)-Q^{\pi_{\nu}}_{\tau}(s,a)&=\gamma\int_{S}(V^{\pi_{\nu'}}_{\tau}(s')-V^{\pi_{\nu}}_{\tau}(s'))P(ds'|s,a)\\
		&=\gamma\int_{S}\int_0^1\int_{\mathbb{R}^d}\frac{\delta J^{\tau,0}(\delta_s)}{\delta \nu}(\nu^\varepsilon,\theta)(\nu'-\nu)(d\theta)P(ds'|s,a)\,d\varepsilon\\
		&\le \gamma \sup_{\nu\in \mathcal{P}_1(\mathbb{R}^d)}\sup_{\theta\in \mathbb{R}^d}\left|\nabla \frac{\delta J^{\tau,0}(\delta_s)}{\delta \nu}(\nu,\theta)\right| W_1(\nu',\nu) |P|_{b\mathcal{K}(S|A\times S)}\\
		&\le \frac{2\gamma }{(1-\gamma)^2} \left(|r|_{B_b(S\times A)}+\tau \left(2|f|_{\clA_0} +| \ln \mu(A)|\right)\right) |f|_{\clA_1}W_1(\nu',\nu)\, ,
	\end{align*}
	where in the first equality $\nu^{\varepsilon}=\nu+\varepsilon (\nu'-\nu)$ for $\varepsilon\in[0,1]$.
	Owing to \eqref{lem:func_deriv_log_pi} and the fundamental theorem of calculus (noting that $\pi_{\nu}$ is continuous in $W_1$), we find that for all $(s,a)\in S\times A$,
	\begin{align*}
		\ln\frac{d\pi_{\nu'}}{d\mu}(a|s)-\ln \frac{d\pi_{\nu}}{d\mu}(a|s)
		&=\int_0^1\int_{\mathbb{R}^d}\left(f(\theta,s,a)-\int_{A}f(\theta,s,a')\pi_{\nu^{\varepsilon}}(da'|s)\right)(\nu'-\nu)(d\theta)\,d\varepsilon\\
		&\le 2|f|_{\clA_1}W_1(\nu',\nu)\,.
	\end{align*}
	Thus,
	$$
	|I_2| \le  \frac{4}{1-\gamma}\left( \frac{\gamma }{(1-\gamma)^2} \left(|r|_{B_b(S\times A)}+\tau \left(2|f|_{\clA_0} +| \ln \mu(A)|\right)\right)+\tau\right)|f|_{\clA_1}^2W_1(\nu',\nu)\,.
	$$
	Putting it all together, we find that there is a constant $L=L(\gamma,|r|_{B_b(S\times A)},\tau,\mu(A),|f|_{\clA_1})$ such that 
	$$
	\left |\nabla\frac{\delta J^{\tau,0}}{\delta \nu}(\nu',\theta)- \nabla\frac{\delta J^{\tau,0}}{\delta \nu}(\nu,\theta)\right| \le  L W_1(\nu',\nu)\,.
	$$
	By Lemma \ref{lem:bounds}, for all $\tau,\tau'\ge 0$, $\nu\in \mathcal P(\mathbb R^d)$, $s\in S$, and $\mu-a.e. \,a\in A$, we have
	$$
	|\bar{Q}^{\pi_{\nu}}_{\tau'}(s,a)-\bar{Q}^{\pi_{\nu}}_{\tau}(s,a)|\le \frac{  |\tau'-\tau|}{(1-\gamma)^2}\left(2|f|_{\clA_0} +| \ln \mu(A)|\right)\,,
	$$
	and hence for all $\theta\in \bbR^d$,
	\begin{align*}
		\nabla\frac{\delta J^{\tau',0}}{\delta \nu}(\nu,\theta)-\nabla\frac{\delta J^{\tau,0}}{\delta \nu}(\nu,\theta)&=\int_{S}\int_{A}\left(\bar{Q}^{\pi_{\nu}}_{\tau'}(s,a)-\bar{Q}^{\pi_{\nu}}_{\tau}(s,a)\right)\nabla \frac{\delta \pi}{\delta \nu}(\nu,\theta)(da|s)d^{\pi_{\nu}}_{\rho}(ds)\\
		&\le  \frac{ 2|\tau'-\tau|}{(1-\gamma)^2}\left(2|f|_{\clA_0} +| \ln \mu(A)|\right)|f|_{\clA_1}\,,
	\end{align*}
	which completes the proof.
\end{proof}

\subsubsection{Proof of Theorem \ref{thm:optimality} and Corollary \ref{cor:local_max_PDE}}
\label{sec:proofof:thm:optimality}
\begin{proof}[Proof of Theorem~\ref{thm:optimality}]
	Let $\nu\in\mathcal{P}_1(\mathbb R^d)$ be a local maximizer  of $J^{\tau,\sigma}$. 
	Thus, there exists $\delta > 0$ such that $J^{\tau,\sigma}(\nu) \geq J^{\tau,\sigma}(\nu')$  for all $\nu'\in \mathcal{P}_1(\mathbb R^d)$ satisfying $W_1(\nu,\nu') < \delta$.
	Let  $\nu'\in\mathcal{P}_1(\mathbb{R}^d)$ satisfy $W_1(\nu',\nu) < \delta$ and define $\nu^{\varepsilon}=\nu+\varepsilon (\nu'-\nu)$ for $\varepsilon\in[0,1]$. 
	Note that for all $\varepsilon \in [0,1]$, we have $W_1(\nu,\nu^\varepsilon) < \delta$. Throughout the proof, we assume that $\nu \in \mathcal{P}_1^{\textnormal{fe}}(\bbR^d)$ if $\sigma>0$ since otherwise it would not be a local maximizer. Henceforth, we  also restrict to $\nu' \in \mathcal{P}_1^{\textnormal{fe}}(\bbR^d)$ if $\sigma>0$. For brevity of notation, we drop the $\theta$ dependence in all integrands throughout the proof.
	
	Using Lemma \ref{lem:func_deriv_pi}, we find
	$$
	0\le \frac{J^{\tau,\sigma}(\nu)-J^{\tau,\sigma}(\nu^{\varepsilon})}{\varepsilon}=-\frac{1}{\varepsilon}\int_0^\varepsilon\int_{\mathbb R^d}\frac{\delta J^{\tau,0}}{\delta\nu}(\nu^{\varepsilon'})(\nu'-\nu)(d\theta)\,d\varepsilon'+\frac{\sigma^2}{2\varepsilon}\left(\textnormal{KL}(\nu^{\varepsilon}|e^{-U})-\textnormal{KL}(\nu|e^{-U})\right)\,,
	$$
	where we note that if $\sigma > 0$ then $\textnormal{KL}(\nu^{\varepsilon}|e^{-U})<\infty$ since $\nu^{\varepsilon}$ is a convex combination of $\nu$ and $\nu'$, both of which belong to $\mathcal{P}_1^{\textnormal{fe}}(\bbR^d)$. It follows that
	\begin{gather*}
		\textnormal{KL}(\nu^{\varepsilon}|e^{-U})-\textnormal{KL}(\nu|e^{-U})=\int_{\mathbb R^d}\left(\nu^{\varepsilon}\ln\frac{\nu^{\varepsilon}}{e^{-U}}-\nu\ln\frac{\nu}{e^{-U}}\right)\,d\theta\\
		=\int_{\mathbb R^d}\left(\nu^{\varepsilon}\ln\nu^{\varepsilon}-\nu\ln\nu+U(\nu^{\varepsilon}-\nu)\right)\,d\theta\\
		=\int_{\mathbb R^d}\left(g(\nu^{\varepsilon})-g(\nu)+\varepsilon U(\nu'-\nu)\right)\,d\theta\,,
	\end{gather*}
	where $g(x):=x\ln x$. 
	Since $g$ is convex, we have that for all $\theta\in\mathbb R^d$,
	$$
	\frac{1}{\varepsilon}\left(g(\nu^{\varepsilon})-g(\nu)+\varepsilon U(\nu'-\nu)\right)
	\le g(\nu')-g(\nu)+ U(\nu'-\nu) =\nu'\ln\frac{\nu'}{e^{-U}}-\nu\ln\frac{\nu}{e^{-U}}\,.
	$$
	Since $\textnormal{KL}(\nu'|e^{-U})$ and $\textnormal{KL}(\nu|e^{-U})$ are both finite, 
	the right hand side of the above inequality is integrable, and hence applying reverse Fatou's lemma, we get
	$$
	\limsup_{\varepsilon\rightarrow 0}\frac{\textnormal{KL}(\nu^{\varepsilon}|e^{-U})-\textnormal{KL}(\nu|e^{-U})}{\varepsilon}\le \int_{\mathbb{R}^d}\left(g'(\nu) + U\right)(\nu'-\nu)(d\theta)\,.
	$$
	Thus, using the boundedness of $\frac{\delta J^{\tau,0}}{\delta\nu}$ (i.e.,  \eqref{ineq:bound_J0}) and  Fatou's Lemma, we obtain
	\begin{equation}\label{ineq:local_max}
		0\le \int_{\mathbb{R}^d}\left(-\frac{\delta J^{\tau,0}}{\delta\nu}(\nu,\theta)+\frac{\sigma^2}{2}\ln\nu+\frac{\sigma^2}{2}U\right)(\nu'-\nu)(d\theta)\,
	\end{equation}
	for all $\nu'\in \mathcal{P}_1(\bbR^d)$ if $\sigma=0$ and for all $\nu'\in \mathcal{P}_1^{\textnormal{fe}}(\bbR^d)$ if $\sigma>0$.
	Applying Lemma \ref{lem:ft_var}, we find that the function
	$$
	F(\theta)=\frac{\delta J^{\tau.0}}{\delta \nu}(\nu,\theta)-\frac{\sigma^2}{2}  U(\theta)-\frac{\sigma^2}{2}\ln\nu(\theta)
	$$
	is constant in $\theta$, $\nu$-a.e..
	
	We would like to show that if $\sigma > 0$ then $\nu$ is equivalent to the Lebesgue measure $\lambda$ and to that end, we follow the argument of the proof of \cite{hu2019meanode}[Prop.\ 3.9].
	Suppose, by contradiction, that $\nu$ is not equivalent to the Lebesgue measure. 
	Then there is a $\mathcal K \in \mathcal B(\mathbb R^d)$ such that $\nu(\mathcal K) = 0$ and $\lambda(\mathcal K) > 0$. Note that on $\mathcal K$, we have $\ln \nu = -\infty$. Rearranging~\eqref{ineq:local_max}, we get
	\begin{align*}
		\frac{\sigma^2}{2}\int_{\mathcal{K}}\ln\frac{\nu}{e^{-U}}\nu'(d\theta) &\ge 
		\int_{\mathbb{R}^d}\frac{\delta J^{\tau,0}}{\delta\nu}(\nu)(\nu'-
		\nu)(d\theta) - \frac{\sigma^2}{2}\textnormal{KL}(\nu|e^{-U})- \frac{\sigma^2}{2}\int_{\mathbb{R}^d\setminus \mathcal{K}}  \ln\frac{\nu}{e^{-U}}\nu'(d\theta)\,.
	\end{align*}
	Choosing $\nu'=e^{-U} \in \mathcal{P}_1^{\textnormal{fe}}(\bbR^d)$, we get
	$$
	\int_{\mathcal{K}}\ln\frac{\nu}{e^{-U}}\nu'(d\theta) = -\infty  \quad \textrm{and} \quad \int_{\mathbb{R}^d\setminus \mathcal{K}}\ln\frac{\nu}{e^{-U}}\nu'(d\theta)< \infty\,.
	$$
	Noting that $\textnormal{KL}(\nu|e^{-U})<\infty$ is finite because $\nu\in \mathcal{P}_1^{\textnormal{fe}}(\bbR^d)$ and that $\int_{\mathbb{R}^d}\frac{\delta J^{\tau,0}}{\delta\nu}(\nu,\theta)(\nu'-\nu)d\theta<\infty$ by \eqref{ineq:bound_J0}, we get a contradiction, and hence $\nu$ is equivalent to the Lebesgue measure. Therefore, if $\sigma>0$, we have that the function $F$ is constant in $\theta$, $\lambda$-a.e.. In particular, if $\sigma>0$, then for $\lambda$-a.a.\ $\theta\in \bbR^d$,
	$$
	\nu(\theta)=\mathcal{Z}^{-1} \exp\left(\frac{2}{\sigma^2}\frac{\delta J^{\tau,0}}{\delta \nu}(\nu,\theta)- U(\theta)\right), \quad \textnormal{where} \quad 
	\mathcal{Z}= \int_{\bbR^d}\exp\left(\frac{2}{\sigma^2}\frac{\delta J^{\tau,0}}{\delta \nu}(\nu,\theta')- U(\theta')\right)d\theta'\,.
	$$
\end{proof}

\begin{proof}[Proof of Corollary \ref{cor:local_max_PDE}]
Define $b:\clP_1(\bbR^d)\times \bbR^d\rightarrow \bbR^d$ by  $b(\nu,\theta)=\nabla \frac{\delta J^{\tau,0}}{\delta \nu}(\nu,\theta) - \frac{\sigma^2}{2} \nabla U(\theta)$. 
Recall that a measure $\tilde{\nu}$ is a solution of $L_{\tilde{\nu}}^*\tilde{\nu}=0$ if for all $\phi \in C_c^\infty(\mathbb{R}^d)$, we have 
$$
\int_{\mathbb{R}^d} L_{\tilde{\nu}}\phi(\theta) \tilde{\nu}(d\theta)=\int_{\mathbb{R}^d}\left(\frac{\sigma^2}{2}\Delta \phi(\theta) +b(\tilde{\nu}, \theta)\cdot \nabla \phi(\theta)\right)\tilde{\nu}(d\theta) =0\,.
$$
Using Theorem \ref{thm:optimality}, we have that 
$$
F(\theta)=\frac{\delta J^{\tau, 0}}{\delta \nu}(\nu,\theta)-\frac{\sigma^2}{2}  U(\theta)-\frac{\sigma^2}{2}\ln\nu(\theta)
$$
is constant in $\theta$, $\nu$-a.e.. By Theorem \ref{thm:bnd_reg_Lion_deriv} and our assumptions on $U$, we have that $\frac{\delta J^{\tau, 0}}{\delta \nu}+\frac{\sigma^2}{2}U$ is differentiable. Moreover,  in the case $\sigma>0$,
by \eqref{eq:max_exp_form}, we find that $\ln \nu$ is  differentiable. Thus, $\nabla F=0$, $\nu$-a.e., which implies that for all $\phi\in C_c^\infty(\bbR^d)$,
$$
\int_{\bbR^d}  \left( b(\nu,\theta)-\frac{\sigma^2}{2}\nabla \ln\nu(\theta)\right)\cdot\nabla \phi(\theta) \nu(d\theta)=0\,.
$$
Applying the divergence theorem, we find
$$
\int_{\bbR^d} L_{\nu} \phi(\theta) \nu(d\theta)=0\,,
$$
which completes the proof.
\end{proof}

\subsubsection{Proof of Theorem \ref{thm:apriori_grad_flow}}
\label{sec:proof_of_energy_identity}

Let $\sigma \geq 0$ and let $b:\mathcal P(\mathbb R^d) \times \mathbb R^d \to \mathbb R^d$ such that for all $\nu \in \mathcal P(\mathbb R^d) $ we have $b(\nu):\mathbb R^d \to \mathbb R^d$ locally bounded.  
We consider the Cauchy problem for $\nu: \bbR_+\rightarrow \mathcal{P}(\bbR^d)$ given by
\begin{equation}\label{eq:gradient_flow_Cauchy_proof}
\partial_t \nu_t = L_{\nu_t}^*\nu_t\,, \;\;t\in (0,\infty)\,,\quad \nu|_{t=0}=\nu_0\,,
\end{equation}
where for all $\nu\in \mathcal{P}(\bbR^d)$, $\phi \in C_c^\infty(\mathbb{R}^d),$ and $\theta\in \bbR^d$,
$$
L_{\nu} \phi(\theta) :=\frac{\sigma^2}{2}\Delta \phi(\theta) +b(\nu, \theta)\cdot \nabla \phi(\theta)\,. 
$$
Note that if $b(\nu,\theta) = \nabla \frac{\delta J^{\tau,0}}{\delta \nu}(\nu,\theta) - \frac{\sigma^2}{2} \nabla U(\theta)$, then this equation is the gradient flow  \eqref{eq:gradient_flow_cauchy} started at $\nu_0$. 
\begin{definition} 
	A  function $\nu: \mathbb{R}_+ \rightarrow \mathcal{P}(\mathbb{R}^d)$ is called a measure-valued solution of \eqref{eq:gradient_flow_Cauchy_proof} if for  all $\phi \in C_c^\infty(\mathbb{R}^d)$ and $t\in \mathbb{R}_+$,
	$$
	\int_{\bbR^d}\phi(\theta) \nu_t(d\theta)=\int_{\bbR^d}\phi(\theta) \nu_0(d\theta)+\int_0^t\int_{\bbR^d}L_{\nu_s}\phi(\theta) \nu_s(d\theta) \,ds\,.
	$$
\end{definition}

The next theorem recalls results concerning the existence and uniqueness of a solution of \eqref{eq:gradient_flow_Cauchy_proof}. 
Moreover, it documents important properties in the case $\sigma>0$ that will be required in the proof of Theorem \ref{thm:increasing_objective}. 
Let  $W^{1,1}(\bbR^d)=\{u\in \mathcal{D}'(\mathbb{R}^d): Du\in L^1(\mathbb{R}^d)\}$, where $\mathcal{D}'(\mathbb{R}^d)$ is the space of distributions on $\mathbb{R}^d$. Denote by $C^{2,1}(\bbR^d\times (0,\infty))$ the space of functions on $\bbR^d\times (0,\infty)$ that are twice continuously differentiable on $\bbR^d$ and once continuously differentiable on $(0,\infty)$.

\begin{theorem}[Existence, uniqueness, and properties]\label{thm:wp_and_prop} 
	Assume  that there exists a constants $C',L'>0$ such that for all $\nu,\nu'\in \mathcal{P}_1(\bbR^d)$ and $\theta \in \bbR^d$,
	$$
	|b(\nu,\theta)| \le C'(1+|\theta|)\quad  \textnormal{and} \quad |b(\nu',\theta)-b(\nu,\theta)|\le L' W_1(\nu',\nu)\,.
	$$ 
	If $\nu_0\in \mathcal{P}_p(\bbR^d)$ for some $p\in \bbN$, then there exists a  measure-valued solution $\nu \in C(\bbR_+;\clP_p(\bbR^d))$ of \eqref{eq:gradient_flow_Cauchy_proof}. Moreover, if $\sigma>0$, then we have that 
	\begin{enumerate}[i)]
		\item \cite{bogachev2015fokker}[Cor.\   6.3.2] $\nu
		:[0,T]\rightarrow \clP_p^{\textnormal{ac}}(\bbR^d)$;
		\item  \cite{bogachev2015fokker}[Rem.\ 7.3.12] $\nu_t\rightarrow \nu_0$ in $L^1(\bbR^d)$ as $t\downarrow 0$;
		\item  \cite{bogachev2015fokker}[Thm.\ 7.4.1]  If 
		$
		\int_{\bbR^d}|\ln \nu_0(\theta)| \nu_0(d\theta)<\infty,
		$
		then for almost every $t\in (0,\infty)$, $\nu_t\in W^{1,1}(\bbR^d)$, and for all $t>0$
		$$\int_0^t \int_{\bbR^d}|\nabla \ln \nu_s(\theta)|^2\nu_s(\theta)\,d\theta \,ds=\int_0^t \int_{\bbR^d}\frac{|\nabla \nu_s(\theta)|^2}{\nu_s(\theta)}\,d\theta \,ds<\infty\,.$$
		In particular, for all  $\phi \in C^1_b(\mathbb{R}^d)$ and $t_1,t_2\in \bbR_+$, we have 
		$$
		\int_{\mathbb{R}^d} \phi(\theta) \nu_{t_2}(d\theta)=\int_{\bbR^d}\phi(\theta)\nu_{t_1}(d\theta)+ \int_{t_1}^{t_2}\int_{\mathbb{R}^d}\nabla\phi(\theta)\cdot \left(b(\nu_s,\theta)-\frac{\sigma^2}{2}\nabla \ln \nu_s(\theta)\right)\nu_s(d\theta) \,ds\,;
		$$
		\item \cite{bogachev2015fokker}[Cor.\   8.2.2] For each compact interval $[t_1,t_2]\subset (0,\infty)$, there exists a constant $K=K(p,t_1,t_2,C',\sigma)$ such that for all $t\in [t_1,t_2]$ and $\theta\in \bbR^d$, we have 
		$$
		\exp\left[-K(1+|\theta|^2)\right] \le \nu_t(\theta) \le \exp\left[K(1+|\theta|^2)\right]\,;
		$$
		\item \cite{bogachev2015fokker}[Cor.\   8.2.5]  combined with \cite{bogachev2015fokker}[Cor.\   8.2.2] For every  compact interval $[t_1,t_2]\subset (0,\infty)$, we have $$\int_{t_1}^{t_2}\int_{\bbR^d}|\ln \nu_s(\theta)|\nu_s(\theta)\,d\theta \,ds<\infty\,.$$ In particular, if $\nu_0\in \clP_2^{\textnormal{fe}}(\bbR^d)$, then we can take $t_1=0$;
		\item  \cite{lasota2013chaos}[Thm.\ 11.7.1] Assume further that $b(\nu)\in C^3(\bbR^d)$ for all $\nu\in \mathcal{P}_2(\bbR^d)$ and $\nu_0 \in C_c^0(\bbR^d)$.
		Then  $\nu\in C^{2,1}(\bbR^d\times (0,\infty))$ and for all $t>0$
		$$
		|\nu_t(\theta)|, |\partial_t \nu_t(\theta)|, |\nabla \nu_t(\theta)|, |\nabla^2\nu_t(\theta)|\le Kt^{(-p+2)/2}\exp\left(-\frac{1}{2t}\delta |\theta|^2\right)\,.
		$$
	\end{enumerate}
	If, in addition, there exists an $L''>0$ such that for all $\nu\in \mathcal{P}_1(\bbR^d)$ and $\theta,\theta'\in \bbR^d$,
	\begin{equation}
		\label{eq:lip_in_space_for_ext_uniq_thm}
		|b(\nu,\theta')-b(\nu,\theta)|\le L''|\theta'-\theta|\,,
	\end{equation}
	or both $\sigma>0$ and $p\ge 4$, 
	then the solution $\nu$ is the unique solution of \eqref{eq:gradient_flow_Cauchy_proof}.
\end{theorem}
\begin{proof}[Proof of Theorem~\ref{thm:wp_and_prop}.]
	By \cite{funaki1984certain}[Thm.\ 2.1], there exists a measure-valued solution of \eqref{eq:gradient_flow_Cauchy_proof} 
	(see, also, \cite{manita2014nonlinear}[Thm.\ 1.1] or~\cite{hammersley2021mckean}[Thm.\ 2.10]).
	The continuity in time of the solution in $\clP_p(\bbR^d)$ follows from the weak continuity in time  and~\cite{villani2009optimal}[Thm.\ 6.9]. 
	
	Clearly, $\nu$ is a solution of the linear equation given by
	$$
	\partial_t \mu_t = L_{\nu_t}^*\mu_t\,,\,\,\, t\in (0,\infty)\,,\,\,\, \mu|_{t=0}=\nu_0\,.
	$$
	Thus, if $\sigma>0$ then, the statements i)-vi) hold by the references given in the statement of the theorem.
	
	If~\eqref{eq:lip_in_space_for_ext_uniq_thm} holds, then assumptions DH1-DH4 of \cite{manita2015uniqueness}[Thm.\ 4.4] hold (i.e., taking the functions $W, V, U$ and $G$ in DH1-DH4 to be $W(\theta)=1$, $V(\theta)=1+|\theta|^k$, $U(\theta)=\sqrt{1+|\theta|^2}$ (not our potential $U$), and $G(u)=u$), which implies that the solution of~\eqref{eq:gradient_flow_Cauchy_proof} is unique.
	If $\sigma>0$ and $p\geq 4$, then the assumptions H1-H4 of \cite{manita2015uniqueness}[Thm.\ 3.1, Ex.\ 3.2] hold,  which implies that the solution of~\eqref{eq:gradient_flow_Cauchy_proof} is unique.
\end{proof}

The following lemma indicates sufficient conditions on the activation function $f$ and potential $U$ to satisfy the assumptions of Theorem \ref{thm:wp_and_prop}.

\begin{corollary}\label{cor:b_growth_and_Lip}
	Define $b_{\sigma}:\clP_1(\bbR^p)\times \bbR^d\rightarrow \bbR^d$ by $b_{\sigma}(\nu,\theta) = \nabla \frac{\delta J^{\tau,0}}{\delta \nu}(\nu,\theta) - \frac{\sigma^2}{2}\nabla U(\theta)$. 
	\begin{enumerate}[i)]
		\item If $f\in \clA_1$ and Assumption \ref{asm:U_lin_growth} holds, then  Theorem \ref{thm:bnd_reg_Lion_deriv} implies that all $\nu,\nu'\in \mathcal{P}_1(\bbR^d)$ and $\theta\in \bbR^d$,
		\begin{equation}\label{ineq:b_growth}
			|b_{\sigma}(\nu,\theta)|\le C_1 + \frac{\sigma^2}{2}|\nabla U(\theta)|\le C_1+ \frac{\sigma^2 C_U}{2}(1+|\theta|)
		\end{equation}
		\begin{equation}\label{ineq:b_lip_measure}
			\textrm{and} \quad |b_{\sigma}(\nu',\theta)-b_{\sigma}(\nu,\theta)| \le L W_1(\nu',\nu)\le LW_2(\nu',\nu)\,.
		\end{equation}
		\item  If $f\in \clA_2$ and Assumption \ref{asm:GradU_Lip} holds, then   Theorem \ref{thm:bnd_reg_Lion_deriv} implies that for all $\nu\in \mathcal{P}_1(\bbR^d)$ and $\theta,\theta'\in \bbR^d$,
		\begin{equation}\label{ineq:b_lip_theta}
			|b_{\sigma}(\nu,\theta')-b_{\sigma}(\nu,\theta)|\le \left(C_2+\frac{L_U\sigma^2}{2}\right)|\theta' - \theta|\,.
		\end{equation}
		\item If $f\in \clA_4$ and $U\in C^4(\bbR^d)$, then by Theorem \ref{thm:bnd_reg_Lion_deriv},  $b_{\sigma}(\nu)\in C^3(\bbR^d)$ for all $\nu\in \mathcal{P}_1(\bbR^d)$.
	\end{enumerate}
	Therefore, the relevant conclusions of Theorem~\ref{thm:wp_and_prop} hold.
\end{corollary}

In the following theorem, we prove that the objective function $J^{\tau,\sigma}(\nu)= J^{\tau,0}(\nu)-\textnormal{KL}(\nu|e^{-U})$ is increasing along the gradient flow $(\nu_t)_{t\in \bbR_+}$.

\begin{theorem}\label{thm:increasing_objective}
	Let $f\in \clA_1$, Assumption \ref{asm:U_lin_growth} hold, and $\nu_0\in \clP_1(\bbR^d)$ if $\sigma=0$, and $\nu_0\in \clP_2^{\textnormal{fe}}(\bbR^d)$ if $\sigma>0$. 
	Then for all $t\in \bbR_+,$
	\begin{equation}
		\label{eq:energy_identity}
		J^{\tau,\sigma}(\nu_t)= J^{\tau,\sigma}(\nu_0)+\int_0^t\int_{\mathbb{R}^d}\left|\nabla \frac{\delta J^{\tau,0}}{\delta \nu}(\nu_s,\theta) -\frac{\sigma^2}{2}\nabla \ln \frac{\nu_s(\theta)}{e^{-U(\theta)}}\right|^2\nu_{s}(d\theta) \,ds\,.
	\end{equation}
\end{theorem}
If  we assume further that  $f\in \clA_4$, $U\in C^4(\bbR^d)$, and $\nu\in C^2_c(\bbR^d)$, then \eqref{eq:energy_identity} can be proved through classical considerations using Corollary~\ref{cor:b_growth_and_Lip} and Theorem \ref{thm:wp_and_prop} (iiii-vi) and the fact that \eqref{eq:gradient_flow_Cauchy_proof} holds classically for all $t\in \bbR_+$ (i.e., including $t=0$), so one can appeal to the fundamental theorem of calculus. We omit the proof of this (see, for example, the proof of \cite{otto2000generalization}[Lem.\ 1]).
\begin{proof}[Proof of Theorem~\ref{thm:increasing_objective}] 
	We divide the proof into two steps. In the first step, we apply Lemma \ref{lem:func_deriv_J} to compute $\frac{d}{dt}J^{\tau,0}(\nu_t)$, and in the second step, we use ideas from \cite{bogachev2018distances}[Thm.\ 1.1] to  compute $\frac{d}{dt}\textnormal{KL}(\nu_t|e^{-U})$. 
	For brevity of notation, we drop the $\theta$ dependence in all integrands throughout the proof. Define $b:\clP_1(\bbR^d)\times \bbR^d\rightarrow \bbR^d$ by $b(\nu,\theta) = \nabla \frac{\delta J^{\tau,0}}{\delta \nu}(\nu,\theta) - \frac{\sigma^2}{2} \nabla U(\theta).$

	\
	
	\noindent\textbf{Step I.}
	By Corollary~\ref{cor:b_growth_and_Lip}, there exists a measure-valued solution $\nu\in C(\bbR_+;\clP_1(\bbR^d))$ of \eqref{eq:gradient_flow_Cauchy_proof}.
	Using Lemma~\ref{lem:func_deriv_J}, we find that for all $t\in \bbR_+$,
	\begin{align*}
		\frac{d}{dt}J^{\tau,0}(\nu_t)=\lim_{h\rightarrow 0} \frac{J^{\tau,0}(\nu_{t+h})-J^{\tau,0}(\nu_t)}{h}=\lim_{h\rightarrow 0}\int_0^1\left[\int_{\mathbb{R}^d} \frac{\delta J^{\tau,0}}{\delta \nu}(\nu^{\varepsilon,h}_t)\frac{1}{h}(\nu_{t+h}-\nu_t)(d\theta)\right]\,d\varepsilon\,,
	\end{align*}
	where $\nu^{\varepsilon,h}_t := \nu_t + \varepsilon(\nu_{t+h} -\nu_t)$ and $h\ge -t$.  In the case $t=0$, we understand the derivative as a right-hand derivative. Our aim to calculate the limit on the right-hand-side. Using Lemma \ref{cor:b_growth_and_Lip} and Theorem \ref{thm:wp_and_prop}(iii) in the case $\sigma>0$ and a simple density argument in the case $\sigma=0$, we find that for all $\phi \in C^1_b(\mathbb{R}^d)$ and $t_1,t_2\in\bbR_+$,
	$$
	\int_{\mathbb{R}^d} \phi \, \nu_{t_2}(d\theta)=\int_{\bbR^d}\phi \,\nu_{t_1}(d\theta)+ \int_{t_1}^{t_2}\int_{\mathbb{R}^d}\nabla\phi \cdot \left(b(\nu_s)-\frac{\sigma^2}{2}\nabla \ln \nu_s\right)\nu_s(d\theta) \,ds \,.
	$$
	Owing to Theorem \ref{thm:bnd_reg_Lion_deriv}, we have that $\frac{\delta J^{\tau,0}}{\delta \nu}(\nu^{\varepsilon,h}_t)\in C_b^1(\bbR^d)$ for all $t,\varepsilon,$ and $h$.
	Thus, for all $t\in \bbR_+$,
	$$
	\int_{\mathbb{R}^d}  \frac{\delta J^{\tau,0}}{\delta \nu}(\nu^{\varepsilon,h}_t) \frac1h(\nu_{t+h}-\nu_t)(d\theta) = \frac{1}{h}\int_t^{t+h}\int_{\mathbb{R}^d}\nabla  \frac{\delta J^{\tau,0}}{\delta \nu}(\nu^{\varepsilon,h}_s)\cdot \left(b(\nu_s)-\frac{\sigma^2}{2}\nabla \ln \nu_s\right)\nu_s(d\theta)\,ds\,.
	$$
	Owing to Corollary~\ref{cor:b_growth_and_Lip} for all $t\in \bbR_+$, $\lim_{h\rightarrow 0}\nu_{t+h}=\nu_t$ in $W_2$. 
	Thus, by Theorem \ref{thm:bnd_reg_Lion_deriv} and \eqref{ineq:W1W2}, we obtain that for all  $(\theta,t) \in \mathbb R^d \times \mathbb R_+$,
	$$
	\lim_{h\downarrow 0}\frac{\delta J^{\tau,0}}{\delta \nu}(\nu^{\varepsilon,h}_t,\theta)= \frac{\delta J^{\tau,0}}{\delta \nu}(\nu_t,\theta)\,.
	$$ 
	Therefore, by virtue of the Lebesgue differentiation theorem and Lebesgue dominated convergence theorem, we get that for all $t\in \bbR_+$,
	$$
	\frac{d}{dt}J^{\tau,0}(\nu_t) 
	= \int_{\mathbb{R}^d}\nabla  \frac{\delta J^{\tau,0}}{\delta \nu}(\nu_t)\cdot \left(\nabla \frac{\delta J^{\tau,0}}{\delta \nu}(\nu_t) -\frac{\sigma^2}{2}\nabla \ln \frac{\nu_t}{e^{-U}}\right)\nu_t(d\theta)\,.
	$$
	An application of the fundamental theorem of calculus then implies that for all $t\in \bbR_+,$ 
	\begin{equation}
		\label{eq:jtau0withsigma}
		J^{\tau,0}(\nu_t) =J^{\tau,0}(\nu_0)+ \int_0^t\int_{\mathbb{R}^d}\nabla  \frac{\delta J^{\tau,0}}{\delta \nu}(\nu_s)\cdot \left(\nabla \frac{\delta J^{\tau,0}}{\delta \nu}(\nu_s) -\frac{\sigma^2}{2}\nabla \ln \frac{\nu_s}{e^{-U}}\right)\nu_s(d\theta)\,ds\,.	
	\end{equation}
	The case $\sigma=0$ of the theorem has now been proved (i.e. part i)). 
	
	\
	
	\textbf{Step II.}
	From now on, we assume $\sigma > 0$ so that $\nu \in C(\bbR_+;\mathcal P_{\text{ac}}^2(\mathbb R^d))$. 
	We aim to show that 
	$$
	\textnormal{KL}(\nu_t|e^{-U}) 
	= \textnormal{KL}(\nu_0|e^{-U})+\int_0^t \int_{\mathbb{R}^d}\nabla \ln  \frac{\nu_s}{e^{-U}}\cdot\left(\nabla \frac{\delta J^{\tau,0}}{\delta \nu}(\nu_s) -\frac{\sigma^2}{2}\nabla \ln \frac{\nu_s}{e^{-U}}\right)\nu_s \,d\theta\,ds\,.
	$$
	Once we have shown this, multiplying it by $\sigma^2/2$ and subtracting from~\eqref{eq:jtau0withsigma} we obtain~\eqref{eq:energy_identity}.
	Letting  $h=e^{U}\nu$, we find that for all $t\in \bbR_+$,
	$$
	\textnormal{KL}(\nu_t|e^{-U})=\int_{\bbR^d} \nu_t\ln \frac{\nu_t}{e^{-U}}d\theta =\int_{\bbR^d} h_t\ln h_t e^{-U}d\theta\,.
	$$
	In terms of $h$, we aim to show that 
	\begin{equation}
		\label{eq:entropy_identity_h}
		\int_{\bbR^d}h_t\ln h_te^{-U}d\theta = \int_{\bbR^d}h_0\ln h_0e^{-U}d\theta
		+ \int_0^t \int_{\mathbb{R}^d}\nabla \ln h_s\cdot\left(\nabla \frac{\delta J^{\tau,0}}{\delta \nu}(\nu_s) -\frac{\sigma^2}{2}\nabla \ln h_s\right)e^{-U}h_s\,d\theta\,ds\,.
	\end{equation}
	
	Let us establish some preliminary bounds. 
	By Corollary~\ref{cor:b_growth_and_Lip}, Theorem~\ref{thm:wp_and_prop} (iii) and (v) and Lemma \ref{lem:relative_entropy}, we have 
	\begin{align}
		\int_0^t \int_{\bbR^d}|\nabla \ln h_s|^2e^{-U}h_s\,d\theta \,ds
		& =\int_0^t \int_{\bbR^d}|\nabla U + \nabla \ln \nu_s|^2\nu_s\,d\theta \,ds \notag\\
		&\le 2\int_0^t \int_{\bbR^d}\left(|\nabla U|^2+|\nabla \ln \nu_s|^2\right)\nu_s\,d\theta \,ds \label{ineq:log_Sob_h}\\
		&\le \int_0^t \int_{\bbR^d}\left(2C^2_U(1+|\theta|)^2+|\nabla \ln \nu_s|^2\right)\nu_s\,d\theta \,ds<\infty\,,\notag\\
		\int_0^t\int_{\bbR^d} h_s|\ln h_s| e^{-U}\,d\theta \,ds&=\int_0^t \nu_s \left(\left|\ln \nu_s\right| +|U|\right)\,d\theta \,ds \notag\\
		&\le \int_0^t\int_{\bbR^d} \nu_s \left(\left|\ln \nu_s\right| +C_U(1+|\theta|^2)\right)\,d\theta \,ds<\infty\,.\label{ineq:entropy_h}
	\end{align}
	
	Notice that $\rho = e^{-U}$ satisfies for all $(\theta,t)\in \bbR^d\times \bbR_+$,
	$$
	\partial_t \rho_t=0=\frac{\sigma^2}{2}\Delta \rho_t + \frac{\sigma^2}{2} \nabla \cdot (\nabla U \rho_t)\,.
	$$
	Applying \cite{bogachev2018distances}[Lem.\  2.4(i)], 
	we find that for all $g\in C^2(\bbR^d)$, $\psi\in C_c^\infty(\bbR^d)$, and $t\ge \tau>0$,
	\footnote{
		In \cite{bogachev2018distances}, the authors introduce a probabilistic approximation to the identity $(\omega_{\varepsilon})_{\varepsilon\in (0,1)}$ and the approximation $h_{\varepsilon}=\omega_{\varepsilon}\ast \nu/ \omega_{\varepsilon}\ast U$, $\varepsilon \in (0,1)$. 
		Then they write down the  equation for $h_{\varepsilon}$, which, in particular, has smooth coefficients and holds pointwise. At this point, they work classically and easily derive~\eqref{eq:g_cr_before_passing}. 
		To pass to the limit as $\varepsilon\downarrow 0$, they appeal to the fact that $\omega_{\varepsilon}\ast \nu$  to $\nu$ uniformly on $\operatorname{supp} \psi \times [\tau,t]$  and  $\nabla \omega_{\varepsilon}\ast \nu \rightarrow \nabla \nu$ in $L^2(\operatorname{supp} \psi \times [\tau,t])$ since $\nu$ is locally H\"older and Sobolev.}
	\begin{align*}
		\int_{\bbR^d}g(h_t)&\psi e^{-U} \,d\theta  + \frac{\sigma^2}{2}\int_{\tau}^t\int_{\bbR^d}|\nabla h_s|^2 g''(h_s)\psi e^{-U}\,d\theta \,ds \\
		&\quad = \int_{\bbR^d}g(h_{\tau})\psi e^{-U} \,d\theta + \int_{\tau}^t\int_{\bbR^d}g(h_s)\left(\frac{\sigma^2}{2}\Delta \psi+ \frac{\sigma^2}{2}  \nabla U\cdot \nabla \psi\right) e^{-U} \,d\theta \,ds \\
		&\quad + \int_{\tau}^t \int_{\bbR^d}\left[g''(h_s)\nabla h_s\cdot \nabla \frac{\delta J^{\tau, 0}}{\delta \nu}(\nu_s) \psi +  g'(h_s)\nabla \psi \cdot \nabla \frac{\delta J^{\tau, 0}}{\delta \nu}(\nu_s)\right] h_se^{-U}\,d\theta \,ds\,.
	\end{align*}
	Using the divergence theorem, the identity $ h\nabla \ln h = \nabla h$ and Theorem \ref{thm:wp_and_prop} (iii), we find that for all $t\ge \tau>0$,
	\begin{equation}
		\begin{aligned}\label{eq:g_cr_before_passing}
			\int_{\bbR^d}g(h_t)\psi e^{-U} \,d\theta &= \int_{\bbR^d}g(h_{\tau})\psi e^{-U} \,d\theta+ \int_{\tau}^t \int_{\bbR^d}g''(h_s) \nabla h_s \cdot \left(\nabla \frac{\delta J^{\tau, 0}}{\delta \nu}(\nu_s)-\frac{\sigma^2}{2}\nabla \ln h_s\right)\psi h_se^{-U}\,d\theta \,ds\\
			&\quad + \int_{\tau}^t \int_{\bbR^d}g'(h_s)\nabla \psi\left(\nabla \frac{\delta J^{\tau, 0}}{\delta \nu}(\nu_s)-\frac{\sigma^2}{2}\nabla \ln h_s\right)  h_se^{-U}\,d\theta \,ds\,.
		\end{aligned}
	\end{equation}
	
	Consider $g\in C^1(\bbR_+)$ such that $g'\in C^0_b(\bbR_+)$, $g'$ is continuously differentiable except at a finite number of points $D\subset \bbR_+$, and $\sup_{x\in \bbR_+\setminus D}xg''(x)<\infty$. The functions ~\eqref{eq:gkm} we  consider below  satisfy this assumption. 
	Mollifying $g$, we  obtain a sequence $\{g_n\}_{n\in \bbN}\in C^\infty(\bbR_+)$ such that $g_n\rightarrow g$ and $g_n'\rightarrow g'$ pointwise on $\bbR_+$ as $n\rightarrow \infty$, $g_n''\rightarrow g''$ pointwise on $\bbR_+ \setminus D$ as $n\rightarrow \infty$,  $\sup_{n\in \bbN, x\in \bbR_+} |g'_n(x)|<\infty$, and $\sup_{n\in \bbN, x\in \bbR_+\setminus D}x|g''_n(x)|<\infty$. Let $\psi \in C_c^\infty(\bbR^d)$ be such that $\psi\ge 0$, $\psi=1$ for $|\theta|<1$ and $\psi=0$ for $|\theta|\ge 2$. For $n\in \bbN$, define $\psi_n \in C_c^\infty(\bbR^d)$ by  $\psi_n(x)=\psi(x/n)$. By \eqref{eq:g_cr_before_passing} and the fact that the Lebesgue measure of the set $\{h_t \in D\}$ is zero by \cite{kinderlehrer2000introduction}[Ch.\ 2, Lem. A.4], we have that for all $n\in \bbN$ and  $t\ge \tau>0$, 
	\begin{align*}
		\int_{\bbR^d}g_n(h_t)\psi_n e^{-U} \,d\theta &= \int_{\bbR^d}g_n(h_{\tau})\psi_n e^{-U} \,d\theta\\
		&\quad + \int_{\tau}^t \int_{h_s\not \in D}g''_n(h_s) \nabla h_s \cdot \left(\nabla \frac{\delta J^{\tau, 0}}{\delta \nu}(\nu_s)-\frac{\sigma^2}{2}\nabla \ln h_s\right)\psi_n h_se^{-U}\,d\theta \,ds\\
		&\quad + \int_{\tau}^t \int_{\bbR^d}g'_n(h_s)  \nabla \psi_n \left(\nabla \frac{\delta J^{\tau, 0}}{\delta \nu}(\nu_s)-\frac{\sigma^2}{2}\nabla \ln h_s\right) h_se^{-U}\,d\theta \,ds\,.
	\end{align*}
	For all $(\theta, t)\in \bbR^d\times \bbR_+$, we have 
	\begin{align*}
		|g_n(h)|e^{-U}&\le \sup_{n\in \bbN}|g_n(0)|e^{-U} + \sup_{n\in \bbN, x\in \bbR_+}|g'_n(x)|he^{-U},\\
		\mathbf{1}_{h\not \in D}g''_n(h)|\nabla h|\left|\nabla \frac{\delta J^{\tau, 0}}{\delta \nu}(\nu)-\frac{\sigma^2}{2}\nabla \ln h\right|he^{-U}&\le \sup_{n\in \bbN, x\in \bbR_+\setminus D}|xg_n''(x)|\left(C_1|\nabla h|e^{-U}+\frac{\sigma^2}{2}|\nabla \ln h|^2he^{-U}\right),\\
		|g_n'(h)| \left|\nabla \frac{\delta J^{\tau, 0}}{\delta \nu}(\nu)-\frac{\sigma^2}{2}\nabla \ln h\right|  he^{-U}&\le \sup_{n\in \bbN, x\in \bbR_+}|g'_n(x)| \left(C_1 +\frac{\sigma^2}{2}|\nabla \ln h |h e^{-U} \right)\,.
	\end{align*}
	Owing to \eqref{ineq:log_Sob_h}, we have that for all $t\in \bbR_+$,
	\[
	\int_0^t\int_{\bbR^d}|\nabla \ln h_s|h_se^{-U} \,d\theta \,ds  \le \sqrt{t}\left(\int_0^t\int_{\bbR^d}|\nabla \ln h_s|^2h_se^{-U} \,d\theta \,ds \right)^{1/2}<\infty
	\]
	and
	\[
	\int_0^t\int_{\bbR^d}|\nabla h_s|e^{-U} \,d\theta \,ds =\int_0^t\int_{\bbR^d}|\nabla \ln h_s|h_se^{-U} \,d\theta \,ds<\infty\,.
	\]
	Applying the dominated convergence theorem to pass to the limit as $n\rightarrow \infty$, we obtain that for all $t\ge \tau>0$,
	\begin{align*}
		\int_{\bbR^d}g(h_t) e^{-U} \,d\theta &= \int_{\bbR^d}g(h_{\tau}) e^{-U} \,d\theta
		+ \int_{\tau}^t \int_{h_s\not \in D} g''(h_s) \nabla h_s \cdot\left(\nabla \frac{\delta J^{\tau, 0}}{\delta \nu}(\nu_s)-\frac{\sigma^2}{2}\nabla \ln h_s\right) h_s e^{-U}\,d\theta \,ds\,.
	\end{align*}
	By Taylor's theorem and the fact that $g'\in C^0_b(\bbR^d)$, for all $t\in \bbR_+$, we have
	\begin{align*}
		\left|\int_{\bbR^d}g(h_t) e^{-U}d\theta - \int_{\bbR^d}g(h_0) e^{-U}d\theta\right| &\le \int_0^1\int_{\bbR^d}|g'(h_0+\varepsilon (h_t-h_0))| |h_t-h_0| e^{-U}\,d\varepsilon \,d\theta \\
		&\le |g'|_{L^\infty}\int_{\bbR^d} |\nu_t-\nu_0| \,d\theta\,,
	\end{align*}
	which tends to zero as $t\downarrow 0$ by Corollary~\ref{cor:b_growth_and_Lip} and Theorem~\ref{thm:wp_and_prop} (i). 
	Therefore, for all $t\in \bbR_+$,
	\begin{equation}\label{eq:g_cr}
		\begin{aligned}
			\int_{\bbR^d}g(h_t) e^{-U} d\theta &= \int_{\bbR^d}g(h_{0}) e^{-U} d\theta
			+ \int_{0}^t \int_{{h_s\not \in D}}g''(h_s)\nabla h_s \cdot \left(\nabla \frac{\delta J^{\tau, 0}}{\delta \nu}(\nu_s)-\frac{\sigma^2}{2}\nabla \ln h_s\right) h_s e^{-U} \,d\theta \,ds\,.\\
		\end{aligned}
	\end{equation}
	We would like to take $g:\bbR_+\rightarrow \bbR$ given by $g(x)=x\ln x -x$ in \eqref{eq:g_cr}, but $g'\not \in C^0_b(\bbR_+)$. 
	Nevertheless, if we could, then for all $t\in \bbR_+$,
	\begin{align*}
		\int_{\bbR^d}(h_t\ln h_t-h_t) e^{-U} d\theta &= \int_{\bbR^d}(h_0\ln h_0-h_0) e^{-U} d\theta\\
		&\quad + \int_{0}^t \int_{\bbR^d}\nabla \ln h_s\cdot \left(\nabla \frac{\delta J^{\tau, 0}}{\delta \nu}(\nu_s)-\frac{\sigma^2}{2}\nabla \ln h_s\right)h_s e^{-U} \,d\theta \,ds\,,
	\end{align*}
	which yields~\eqref{eq:entropy_identity_h}.
	
	Instead, we will approximate $g(x)=x\ln x -x$ and pass to the limit.  Towards this end, as in the proof of ~\cite{bogachev2018distances}[Thm.\ 1.1], we consider the sequence $\{g_{k,m}\}_{k,m\in \bbN}\subset C^1(\bbR_+)$ defined by
	\begin{equation*}
		g_{k,m}(x)=\begin{cases}
			-x \ln k & x\le k^{-1}\,,\\
			x\ln x - x + k^{-1} & k^{-1}< x \le m\,, \\
			x \ln m - m + k^{-1} & x\ge m\,.
		\end{cases}
	\end{equation*}
	Clearly,
	$$
	g'_{k,m}(x)=\begin{cases}
		- \ln k & x\le k^{-1}\,,\\ \ln x  & k^{-1}< x \le m\,, \\
		\ln m  & x\ge m\,,
	\end{cases}
	\quad \textnormal{and} \quad 
	g''_{k,m}(x)=\begin{cases}
		0 & x\le  k^{-1}\,,\\ 
		x^{-1}  & k^{-1}< x < m \,,\\
		0  & x\ge  m\,,
	\end{cases}
	$$
	from which we deduce that $g'_{k,m}\in C^0_b(\bbR_+)$ and $\sup_{x\in \bbR_+-\{k,m\}} x g''_{k,m}(x)=1$. Applying \eqref{eq:g_cr}, we find that for all $t\in \bbR_+$,
	\begin{equation}\label{eq:gkm}
		\begin{aligned}
			\int_{\bbR^d} g_{k,m}(h_t)  e^{-U}d\theta &=\int_{\bbR^d} g_{k,m}(h_0) e^{-U} d\theta  \\
			&\quad +\int_0^t\int_{k^{-1}<h_s<m}  \nabla \ln h_s  \cdot \left(\nabla \frac{\delta J^{\tau, 0}}{\delta \nu}(\nu_s)   -\frac{\sigma^2}{2}\nabla \ln h_s\right) h_s e^{-U}\,d\theta \,ds\,.
		\end{aligned}
	\end{equation}
	We will now pass to the limit each term of \eqref{eq:gkm} as $k,m\rightarrow \infty$. For all $k,m\in \bbN$ and $t\in \bbR_+$, we have
	\begin{align*}
		\int_{\bbR^d} g_{k,m}(h_t) e^{-U}d\theta
		& = - \ln k\int_{h_t\le k^{-1}}h_t e^{-U}d\theta + \int_{k^{-1}<h_t< m}\left(h_t\ln h_t - h_t\right) e^{-U}d\theta\\
		&\quad + \int_{m\le h_t}\left(h_t\ln m - m\right) e^{-U}d\theta +k^{-1}\int_{k^{-1}<h_t}e^{-U}d\theta\,.
	\end{align*}
	Since $h\ge 0$, for all $k,m\in \bbN$ with $m\ge 3$ (so that $m\ln m\ge m $) and $t\in \bbR_+$,
	$$
	\int_{\bbR^d} g_{k,m}(h_t) e^{-U}d\theta\le  \int_{k^{-1}<h_t}h_t\ln h_t e^{-U}d\theta-\int_{k^{-1}<h_t< m}h_t e^{-U}d\theta +k^{-1}\int_{k^{-1}<h_t}e^{-U}d\theta
	$$
	and
	\begin{align*}
		\int_{\bbR^d} g_{k,m}(h_t)e^{-U}d\theta&\ge -\ln k\int_{h_t\le k^{-1}}h_t e^{-U}d\theta + \int_{k^{-1}<h_t< m}\left(h_t\ln h_t - h_t\right) e^{-U}d\theta\,.
	\end{align*}
	The bound
	\[
	\mathbf{1}_{h_t\le k^{-1}}h_t\ln k\le k^{-1}\ln k \le  \sup_{x\ge 1} x^{-1}\ln x=e^{-1}\,, \quad \forall (\theta, t)\in \bbR^d\times \bbR_+\,,
	\]
	allows us to apply the dominated convergence theorem to obtain that for all $t\in \bbR_+$,
	$$
	\lim_{k\rightarrow \infty } \ln k\int_{h_t\le k^{-1}}h_t e^{-U}d\theta=0\,.
	$$
	Clearly, for all $t\in \bbR_+$,
	$$
	\lim_{k\rightarrow \infty}k^{-1}\int_{k^{-1}<h_t}e^{-U}d\theta=0\,.
	$$
	The bound \eqref{ineq:entropy_h} and another application of the dominated convergence theorem imply that for all $t\in \bbR_+$, 
	$$
	\lim_{k,m\rightarrow \infty}\int_{k^{-1}<|h_t|< m}\left(h_t\ln h_t - h_t\right) e^{-U}d\theta=\int_{\bbR^d}h_t\ln h_te^{-U}d\theta -1
	$$
	and 
	$$
	\lim_{k\rightarrow \infty}\int_{k^{-1}<|h_t|}h_t\ln h_t e^{-U}d\theta-\int_{k^{-1}<|h_t|< m}h_t e^{-U}d\theta =\int_{\bbR^d}h_t\ln h_te^{-U}d\theta -1\,.
	$$
	Therefore, for all $t\in \bbR_+$,
	$$
	\int_{\bbR^d}h_t\ln h_te^{-U}d\theta -1\le \lim_{k,m\rightarrow \infty}\int_{\bbR^d} g_{k,m}(h_t) e^{-U}d\theta\le  \int_{\bbR^d}h_t\ln h_te^{-U}d\theta -1\,,
	$$
	which yields for all $t\in \bbR_+$, 
	$$
	\lim_{k,m\rightarrow \infty}\int_{\bbR^d} g_{k,m}(h_t)  e^{-U}d\theta =\int_{\bbR^d}h_t\ln h_t e^{-U}d\theta-1\,.
	$$
	Using the bounds
	\[
	\nabla \ln h_t  \cdot \left(\nabla \frac{\delta J^{\tau, 0}}{\delta \nu}(\nu_t)   -\frac{\sigma^2}{2}\nabla \ln h_t\right) h_t e^{-U} \le \left(C_1 |\nabla \ln h_t| + \frac{\sigma^2}{2}|\nabla \ln h_t|^2\right) h_t e^{-U}\,, \quad \forall (\theta, t)\in \bbR^d\times \bbR_+\,,
	\]
	and~\eqref{ineq:log_Sob_h}, and the dominated convergence theorem, we find that for all $t\in \bbR_+$, 
	\begin{gather*}
		\lim_{k,m\rightarrow \infty}\int_0^t\int_{k^{-1}<h_s<m}  \nabla \ln h_s  \cdot \left(\nabla \frac{\delta J^{\tau, 0}}{\delta \nu}(\nu_s)   -\frac{\sigma^2}{2}\nabla \ln h_s\right) h_s e^{-U}\,d\theta \,ds\\
		=\int_0^t\int_{\bbR^d}  \nabla \ln h_s  \cdot \left(\nabla \frac{\delta J^{\tau, 0}}{\delta \nu}(\nu_s)   -\frac{\sigma^2}{2}\nabla \ln h_s\right) h_s e^{-U}\,d\theta \,ds\,.
	\end{gather*}
	Combining the above, we pass to the limit in every term of \eqref{eq:gkm} to obtain~\eqref{eq:entropy_identity_h}.
\end{proof}

\subsubsection{Proof of Theorem \ref{thm:cts_dep_flow} and moment estimates}\label{sec:proof_of_cts_dep}

\begin{lemma}[Moment estimates]\label{lem:moment_est}
	Let $f\in \clA_2$ and Assumptions \ref{asm:U_lin_growth}, \ref{asm:GradU_Lip}, and \ref{asm:GradU_diss} hold. Let $(\nu_t)_{t\ge 0}$ be the solution of \eqref{eq:gradient_flow_cauchy} with initial condition $\nu_0\in \clP_2(\bbR^d)$. Then for all $\ell>0$ and $t\in \bbR_+$,
	\begin{equation}\label{ineq:moment_est_par}
		\int_{\bbR^d}|\theta|^2\nu_t(d\theta) \le e^{-\alpha_{1,\ell} t}\int_{\bbR^d}|\theta|^2\nu_0(d\theta) +\frac{1}{\alpha_{1,\ell}}\left(\frac{1}{4\ell}\left(2C_1+ \sigma^2 C_U\right)+d\sigma^2\right)(1-e^{-\alpha_{1,\ell} t})\,,
	\end{equation}
	where $\alpha_{1,\ell}:= \sigma^2\kappa -2C_2-\ell$. Moreover, if $\sigma^2\kappa>2C_2$ and $\nu^{\ast}\in \clP_2(\bbR^d)$ is a solution of \eqref{eq:stationary_eq}, then for all $\ell>0$ such that $\alpha_{1,\ell}>0$,
	\begin{equation}\label{ineq:moment_est_ellip}
		\int_{\bbR^d}|\theta|^2\nu^{\ast}(d\theta) \le \frac{1}{\alpha_{1,\ell}} \left(\frac{1}{4\ell}\left(2C_1+ \sigma^2 C_U\right)+d\sigma^2\right)\,.
	\end{equation}
	If, in addition, $\nabla U(0)=0$, then for all $\ell>0$ and $t\in \bbR_+$,
	$$
	\int_{\bbR^d}|\theta|^2\nu_t(d\theta) \le e^{-\alpha_{2,\ell} t}\int_{\bbR^d}|\theta|^2\nu_0(d\theta) +\frac{1}{\alpha_{2,\ell}}\left(\frac{C_1}{2\ell}+d\sigma^2\right)(1-e^{-\alpha_{2,\ell} t})\,,
	$$
	where $\alpha_{2,\ell} := \sigma^2\kappa -\ell$. An analog of \eqref{ineq:moment_est_ellip} then holds for $\nu^{\ast}$ if $\sigma>0$.
\end{lemma}
\begin{proof}
	Define $b:\clP_2(\bbR^p)\times \bbR^d\rightarrow \bbR^d$ by 
	$
	b(\nu,\theta)=\nabla \frac{\delta J^{\tau,0}}{\delta \nu}(\nu,\theta) - \frac{\sigma^2}{2} \nabla U(\theta)\,.
	$
	Theorem \ref{thm:bnd_reg_Lion_deriv} and Assumption \ref{asm:GradU_diss} implies that
	for all $\theta,\theta'\in \bbR^d$ and $\nu,\nu'\in \clP_2(\bbR^d)$, 
	\begin{equation}\label{ineq:b_diss}
		2(b(\nu,\theta')-b(\nu,\theta))\cdot(\theta-\theta')\le \left(2C_2-\sigma^2\kappa\right)|\theta'-\theta|^2\,.
	\end{equation}
	Using \eqref{ineq:b_growth}, \eqref{ineq:b_diss}, and Young's inequality, we find that for all $\theta\in \bbR^d$ and $\nu\in \clP_1(\bbR^d)$ and $\ell>0$,
	\begin{align}
		2b(\nu,\theta)\cdot \theta &\le  2b(\nu,0)\cdot \theta +\left(2C_2-\sigma^2\kappa\right)|\theta|^2\le \left(2C_1+ \sigma^2 C_U\right)|\theta|+\left(2C_2-\sigma^2\kappa\right)|\theta|^2\notag\\
		&\le\frac{1}{4\ell}\left(2C_1+ \sigma^2 C_U\right)+ \left(2C_2+\ell-\sigma^2\kappa\right)|\theta|^2\,,\label{ineq:b_one_sided_growth}
	\end{align}
	where $C_1$ is the constant from Theorem \ref{thm:bnd_reg_Lion_deriv}. 
	
	There exists a probability triple $(\Omega, \mathcal F, \mathbb P)$  supporting a Wiener process $(W_t)_{t\ge 0}$ and a random variable $\theta_0$ such that $\operatorname{Law}(\theta_0) = \nu_0$ and $\theta_0$ is independent of $W$. Let $\bbF=(\clF_t)_{t\ge 0}$ be the corresponding filtration generated by $\theta_0$ and $(W_t)_{t\ge 0}$.  By Theorem~\ref{thm:MVSDE}, there exists a unique strong solution $(\theta_t)_{t\in \bbR_+}$ of~\eqref{eq:MVSDE} and $\nu_t = \operatorname{Law}(\theta_t)$ for all $t\in \bbR_+.$ 
	Applying It\^o's formula and \eqref{ineq:b_one_sided_growth}, we obtain that for all  $\alpha'\in \bbR$, $\ell>0$, and $t\in \bbR_+$, 
	\begin{align*}
		d(e^{\alpha' t}|\theta_t|^2) &=e^{\alpha' t}\left(\alpha' |\theta_t|^2+2 b(\operatorname{Law}(\theta_t),\theta_t)\cdot \theta_t+d\sigma^2\right)dt+e^{\alpha' t}2\sigma\theta_t\cdot dW_t\\
		&\le e^{\alpha' t}\left((\alpha'+2C_2+\ell-\sigma^2\kappa) |\theta_t|^2+\frac{1}{4\ell}\left(2C_1 + \sigma^2 C_U\right)+d\sigma^2\right)dt+e^{\alpha' t}2\sigma\theta_t\cdot dW_t\,.
	\end{align*}
	Setting $\alpha'=\alpha_{\ell}$, integrating the above, and taking the expectation (using a standard stopping time argument), we complete the proof of \eqref{ineq:moment_est_par}. 
	
	To obtain \eqref{ineq:moment_est_ellip},  notice that $\nu^{\ast}$ is  a solution of \eqref{eq:gradient_flow_cauchy}, and hence by \eqref{ineq:moment_est_par}, for all $t\ge 0$ and $\ell>0$, 
	$$
	\int_{\bbR^d}|\theta|^2\nu^{\ast}(d\theta) \le e^{-\alpha_{1,\ell} t}\int_{\bbR^d}|\theta|^2\nu^{\ast}(d\theta) +\frac{1}{\alpha_{1,\ell}}\left(\frac{1}{4\ell}\left(2C_1+ \sigma^2 C_U\right)+d\sigma^2\right)(1-e^{-\alpha_{1,\ell} t})\,.
	$$
	Under the assumption $\sigma^2\kappa>2C_2$, choosing $\ell>0$ such that $\alpha_{1,\ell}>0$, we may pass to the limit in the right-hand-side as $t\rightarrow \infty$ to derive \eqref{ineq:moment_est_ellip}.
	
	If $\nabla U(0)=0$, then for all $\theta\in \bbR^d$ and $\nu\in \clP_1(\bbR^d)$ and $\ell>0$,
	$$
	2 b(\nu,\theta)\cdot \theta \le 2C_1|\theta|-\sigma^2\kappa|\theta|^2 \le \frac{C_1}{2\ell} +(\ell -\sigma^2\kappa)|\theta|^2\,,
	$$
	from which the statements follow via a similar argument to the one given above.
\end{proof}

\begin{proof}[Proof of Theorem \ref{thm:cts_dep_flow}]
	Define $b_{\tau,\sigma}:\clP_1(\bbR^p)\times \bbR^d\rightarrow \bbR^d$ by 
	$$
	b_{\tau,\sigma}(\nu,\theta)=\nabla \frac{\delta J^{\tau,0}}{\delta \nu}(\nu,\theta) - \frac{\sigma^2}{2} \nabla U(\theta)\,.
	$$
	Theorem \ref{thm:bnd_reg_Lion_deriv} implies that
	for all $\theta,\theta'\in \bbR^d$ and $\nu,\nu'\in \clP_1(\bbR^d)$, 
	\begin{align}
		|b_{\sigma,\tau}(\nu,\theta)-b_{\sigma',\tau}(\nu,\theta)|&\le\frac{|\sigma^2-\sigma'^{2}|}{2} |\nabla U(\theta)|\,,\label{ineq:diff_b_sigma}
		\\
		\textnormal{and} \quad |b_{\sigma,\tau}(\nu,\theta)-b_{\sigma,\tau'}(\nu,\theta)|&\le D|\tau-\tau'|\,,\label{ineq:diff_b_tau}
	\end{align}
	where  $D$ is the constant from Theorem \ref{thm:bnd_reg_Lion_deriv}.
	
	There exists a probability space $(\Omega, \mathcal F, \mathbb P)$ supporting a Wiener process $(W_t)_{t\ge 0}$ and random variables $\theta_0$, $\theta_0'$ such that $\operatorname{Law}(\theta_0) = \nu_0$, $\operatorname{Law}(\theta_0') = \nu_0'$, and $(\theta_0,\theta_0')$ is independent of $(W_t)_{t\ge 0}$. By Theorem~\ref{thm:MVSDE}, there exists a unique strong solution $(\theta_t)_{t\in \bbR_+}$ of~\eqref{eq:MVSDE}, and $\nu_t = \operatorname{Law}(\theta_t)$ for all $t\in \bbR_+.$ Similarly, there exists a unique strong solution $(\theta_t')_{t\in \bbR_+}$ of~\eqref{eq:MVSDE} with $\sigma$ and $\tau$ replaced by $\sigma'$ and $\tau'$,  $\nu_t' = \operatorname{Law}(\theta_t')$ for all $t\in \bbR_+.$
	
	By  Lemma \ref{lem:moment_est},  for all $\beta'\in \bbR$ and  $t\in \bbR_+$,
	$$
	\bbE \int_0^te^{2\beta' s}\left(|\theta_s|^2+|\theta_s'|^2\right)ds<\infty\,.
	$$
	Thus, applying It\^o's formula and taking the expectation, we find that for all $\beta'\in \bbR$ and  $t\in \bbR_+,$
	\begin{align*}
		d\left(e^{2\beta' t} \mathbb E\left[|\theta_t-\theta_t'|^2\right]\right)
		&\le  e^{2\beta' t} \mathbb E\left[2\beta'|\theta_t-\theta_t'|^2 + d|\sigma-\sigma'|^2\right]dt\\
		&\quad +   e^{2\beta' t} \mathbb E\left[2(\theta_t-\theta_t')\cdot\left(b_{\sigma,\tau}(\operatorname{Law}(\theta_t),\theta_t)-b_{\sigma,\tau}(\operatorname{Law}(\theta_t),\theta_t')\right)\right]dt\\
		&\quad + e^{2\beta' t} \mathbb E\left[2|\theta_t-\theta_t'|\left|b_{\sigma,\tau}(\operatorname{Law}(\theta_t),\theta_t')-b_{\sigma,\tau}(\operatorname{Law}(\theta_t'),\theta_t')\right|\right]dt\\
		&\quad + e^{2\beta' t} \mathbb E\left[2|\theta_t-\theta_t'|\left|b_{\sigma,\tau}(\operatorname{Law}(\theta_t'),\theta_t')-b_{\sigma',\tau}(\operatorname{Law}(\theta_t'),\theta_t')\right|\right]dt\\
		&\quad + e^{2\beta' t} \mathbb E\left[2|\theta_t-\theta_t'|\left|b_{\sigma',\tau}(\operatorname{Law}(\theta_t'),\theta_t')-b_{\sigma',\tau'}(\operatorname{Law}(\theta_t'),\theta_t')\right|\right]dt\,.
	\end{align*} 
	Making use of the identity
	\begin{equation}\label{ineq:W2Exp}
		W_2(\operatorname{Law}(\theta_t),\operatorname{Law}(\theta_t')) \le \mathbb E|\theta_t-\theta_t'|^2\,, \quad \forall t\in \bbR_+\,,
	\end{equation}
	the bounds \eqref{ineq:b_diss},\eqref{ineq:b_lip_measure},\eqref{ineq:diff_b_sigma}, \eqref{ineq:diff_b_tau}, and Young's inequality, we get that for all $\beta'\in \bbR$, $\ell>0$, and  $t\in \bbR_+,$
	\begin{align*}
		d&\left(e^{2\beta' t} \mathbb E\left[|\theta_t-\theta_t'|^2\right]\right)\\
		&\le e^{2\beta' t}\mathbb E\left[ (2\beta'-\sigma^2\kappa + 2C_2(\tau)) |\theta_t-\theta_t'|^2+2L(\tau)|\theta_t-\theta_t'|W_2(\operatorname{Law}(\theta_t),\operatorname{Law}(\theta_t'))\right]dt\\
		&\quad + e^{2\beta' t} \mathbb E\left[|\sigma^2-\sigma'^2||\theta_t-\theta_t'| |\nabla U(\theta_t')|+D|\tau-\tau'|+d|\sigma-\sigma'|^2\right]dt\\
		&\le (2\beta'-\sigma^2\kappa + 2C_2(\tau)+2L(\tau)+2\ell|\sigma^2-\sigma'^2|)e^{2\beta' t}\mathbb E\left[  |\theta_t-\theta_t'|^2\right]dt\\
		&\quad + e^{2\beta' t} \mathbb E\left[\frac{|\sigma^2-\sigma'^2|}{8\ell} |\nabla U(\theta_t')|^2+D|\tau-\tau'|+d|\sigma-\sigma'|^2 \right]dt\,.
	\end{align*}
	Setting $\beta'=\beta_{\ell}$, integrating the above, taking the expectation, and then applying \eqref{ineq:W2Exp}, we complete the proof of \eqref{ineq:W2_dist_flow_main}.

	To obtain \eqref{ineq:W2_dist_stationary_main}, notice that $\nu^{\ast}$ and $\nu'^{\ast}$ are solutions of \eqref{eq:gradient_flow_cauchy} with corresponding parameters $\sigma,\tau$ and $\sigma',\tau'$, respectively. Thus, by \eqref{ineq:W2_dist_flow_main}, for all $\ell> 0$ and $t\ge 0$, we have
	\begin{equation*}
		\begin{aligned}
			W_2^2(\nu^{\ast},\nu'^{\ast})&\le  e^{-2\beta_{\ell} t} W_2^2(\nu^{\ast},\nu'^{\ast})+ \frac{|\sigma^2-\sigma'^{2}|}{8\ell}\int_0^te^{2\beta_{\ell} (s-t)}\,ds\int_{\bbR^d} |\nabla U(\theta)|^2\nu'^{\ast}(d\theta)\\
			&\quad + \frac{1}{2\beta_{\ell}}\left(D|\tau-\tau'|+d|\sigma-\sigma'|^2  \right)(1-e^{-2\beta_{\ell} t})\,,
		\end{aligned}
	\end{equation*}
	Under the assumption $\beta:=\frac{\sigma^2}{2}\kappa -C_2(\tau)-L(\tau)>0$, choosing $\ell>0$ such that $\beta_{\ell}>0$, we may pass to the limit in the right-hand-side as $t\rightarrow \infty$ to derive \eqref{ineq:W2_dist_stationary_main}.
	
\end{proof}

\subsubsection{Proof of Theorem \ref{thm:ergodicity}}
\label{sec:proof_of_exp_conv}
\begin{proof}[Proof of Theorem~\ref{thm:ergodicity}]
	Owing to \eqref{ineq:b_growth}, \eqref{ineq:b_diss}, \eqref{ineq:b_lip_measure}, and $\beta := \frac{\sigma^2}{2}\kappa-C_2-L>0$,  \cite{bogachev2018distances}[Thm.\ 4.1] there is a unique solution $\nu^\ast$ of \eqref{eq:stationary_eq}. We could also argument employed in~\cite{komorowski2012central}[Thm.\ 2.1] combined with Theorem \ref{thm:cts_dep_flow}  to prove this result. 
	Moreover, by Theorem \ref{thm:cts_dep_flow}, for any solution $\nu \in C(\bbR_+;\clP_2(\bbR^d))$ of \eqref{eq:gradient_flow_cauchy} and all $t\in \bbR_+$, we have
	\begin{equation}
		\label{eq:W_2_convergence_proof}
		W_2(\nu_t,\nu^\ast)\le e^{-\beta t}W_2(\nu_0,\nu^\ast)\,
	\end{equation}
	because $\tilde{\nu}_t=\nu^{\ast}$, $t\in \bbR_+$, is a solution of \eqref{eq:gradient_flow_cauchy}.
	
	We must show that 
	$
	J^{\tau,\sigma}(\nu^{\ast})\ge J^{\tau,\sigma}(\nu_0)
	$
	for an arbitrary $\nu_0\in \mathcal{P}_2^{\textnormal{fe}}(\mathbb R^d)$. 
	Let $\nu \in C(\mathbb R_+; \mathcal P_2(\mathbb R^d))$ be the unique solution of~\eqref{eq:gradient_flow_cauchy} such that $\nu|_{t=0} = \nu_0$.
	We have that $\lim_{t\rightarrow \infty}\nu_t= \nu^\ast$ in $W_2$ by \eqref{eq:W_2_convergence_proof}.
	By the continuity of $J^{\tau,0}$ in $W_2$ (i.e., Theorem~\ref{thm:bnd_reg_Lion_deriv}), 
	the  lower semi-continuity of relative entropy in the 2-Wasserstein topology (\cite{dupuis1997weak}[Lem.\  1.4.3]), and by Theorem \ref{thm:apriori_grad_flow}, 
	we have
	\[
	\begin{split}
		J^{\tau,\sigma}(\nu^\ast) - J^{\tau,\sigma}( \nu_0) & \geq \limsup_{t\to \infty } J^{\tau,\sigma}(\nu_t) - J^{\tau,\sigma}(\nu_0) \\
		& = \limsup_{t\to \infty } \int_0^t\int_{\mathbb{R}^d}\left|\nabla \frac{\delta J^{\tau,0}}{\delta \nu}(\nu_s,\theta) -\frac{\sigma^2}{2}\nabla \ln \frac{\nu_s(\theta)}{e^{-U(\theta)}}\right|^2\nu_{s}(d\theta) \,ds \geq 0\,.
	\end{split}
	\]
	Hence, $\nu^{\ast}$ is a global maximizer.
	By Corollary~\ref{cor:local_max_PDE}, we know any local maximizer of $J^{\tau,\sigma}$ is a solution of~\eqref{eq:stationary_eq}. 
	However, equation~\eqref{eq:stationary_eq} only has one solution, so $\nu^{\ast}$ is the unique  maximizer of $J^{\tau,\sigma}$.  
\end{proof}
\end{document}